\documentclass[12pt]{amsart}

\usepackage{amsmath}
\usepackage{amssymb}
\usepackage{MnSymbol}
\usepackage{pdfsync}

\newcommand{\R}{{\mathbb R}}       
\newcommand{\Z}{{\mathbb Z}}       
\newcommand{\DD}{{\mathcal D}}

\newcommand{\HH}{{\mathcal H}}
\newcommand{\LL}{{\mathcal L}}

\newcommand{\BZ}{{\mathcal B}}
\newcommand{\GZ}{{\mathcal G}}

\newcommand{\RR}{{\mathcal R}}
\newcommand{\CH}{{\mathcal Ch}}
\newcommand{\EE}{{\mathcal E}}
\newcommand{\SSS}{{\mathcal S}}

\newcommand{\TT}{{\mathcal T}}

\newcommand{\CC}{{\mathcal C}}

\newcommand{\diam}{{\rm diam}}
\newcommand{\dist}{{\rm dist}}

\newcommand{\fiproof}{{\hspace*{\fill} $\square$ \vspace{2pt}}}

\newcommand{\interior}[1]{{\stackrel{\mbox{\scriptsize$\circ$}}{#1}}}

\newcommand{\rf}[1]{{(\ref{#1})}}

\newcommand{\supp}{{\rm supp}}

\newcommand{\vphi}{{\varphi}}
\newcommand{\ve}{{\varepsilon}}
\newcommand{\vv}{}
\newcommand{\vvv}{}
\newcommand{\wt}[1]{{\widetilde{#1}}}
\newcommand{\wh}[1]{{\widehat{#1}}}

\newcommand{\noi}{\noindent}
\newcommand{\lip}{{\rm Lip}}

\newcommand{\sss}{{\mathcal S}}

\newcommand{\reg}{{\rm Reg}}

\newcommand{\rest}{{\lfloor}}
\newcommand{\inter}[1]{{\stackrel{\mbox{\scriptsize$\circ$}}{#1}}}

\newtheorem{theorem}{Theorem}[section]
\newtheorem{lemma}[theorem]{Lemma}

\newtheorem{propo}[theorem]{Proposition}

\newtheorem*{lemma*}{Lemma}

\theoremstyle{definition}
\newtheorem{definition}[theorem]{Definition}

\theoremstyle{remark}
\newtheorem{rem}[theorem]{Remark}

\numberwithin{equation}{section}

\newcommand{\brem}{\begin{rem}}
\newcommand{\erem}{\end{rem}}

\begin{document}

\title{Mass transport and uniform rectifiability}

\author{Xavier Tolsa}
\thanks{Partially supported by grants
2009SGR-000420 (Generalitat de
Catalunya) and MTM-2010-16232 (Spain).}

\address{Xavier Tolsa. Instituci\'{o} Catalana de Recerca
i Estudis Avan\c{c}ats (ICREA) and Departament de
Ma\-te\-m\`a\-ti\-ques, Universitat Aut\`onoma de Bar\-ce\-lo\-na,
Catalonia} \email{xtolsa@mat.uab.cat}

\begin{abstract}
In this paper we characterize the so called uniformly rectifiable sets of David and Semmes 
in terms of the Wasserstein distance $W_2$ from optimal mass transport. To obtain this result, we first prove a localization theorem for the distance $W_2$ which asserts that if $\mu$ and $\nu$ are probability measures
in $\R^n$, $\vphi$ is a radial bump function smooth enough so that $\int\vphi\,d\mu\gtrsim1$, 
and $\mu$ has a density bounded from above and from below on $\supp(\vphi)$, then 
$W_2(\vphi\mu,\,a\vphi\nu)\leq c\,W_2(\mu,\nu),$
where $a=\int\vphi\,d\mu/
\int\vphi\,d\nu$.
\end{abstract}

\maketitle


\vspace{-4mm}
\section{Introduction}\vv

In this paper we characterize the so called uniformly rectifiable sets of David and Semmes \cite{DS2}
in terms of the Wasserstein distance $W_2$ from optimal mass transport. To obtain this result, a fundamental
tool is a new localization theorem for the distance $W_2$, which we think that has its own interest.

To state our results in detail we need to introduce some notation and terminology.
Recall that, for $1\leq p<\infty$, the Wasserstein distance $W_p$ between two (Borel) probability measures 
$\mu,\nu$ on $\R^d$ with finite $p$-th moments (i.e., $\int|x|^p\,d\mu(x)<\infty$ and $\int|x|^p\,d\nu(x)<\infty$) is defined by
$$W_p(\mu,\nu) = \inf_\pi \left(\int_{\R^n\times\R^n} |x-y|^p\,d\pi(x,y)\right)^{1/p},$$
where the infimum is taken over all probability measures $\pi$ on $\R^n\times \R^n$ whose marginals are 
$\mu$ and $\nu$. That is, $\pi(A\times \R^n)=\mu(A)$ and $\pi(\R^n\times A)=\nu(A)$ for all measurable sets
$A\subset\R^n$.
The same definition makes sense if instead of probability measures one considers
measures $\mu,\nu,\pi$ of the same mass.
In the particular case $p=1$, by the Kantorovich duality, the distance $W_1$ can also be characterized as follows: 
\begin{equation}\label{eqw1}
W_1(\mu,\nu) = \sup\Bigl\{ \Bigl|{\textstyle \int f\,d\mu  -
\int f\,d\nu}\Bigr|:\,{\rm Lip}(f) \leq1\Bigr\}.
\end{equation}

\enlargethispage{1mm}
The Wasserstein distances play a key role in many problems of optimal mass transport. Further, quite recently they have also been shown to be useful in questions in connection with PDE's (such as the Boltzmann equation) and Riemannian geometry. See  \cite{Villani-ams} and \cite{AGS} for two basic and modern 
references on Wasserstein distances and mass transport.

Let us turn our attention to uniform rectifiability now.
Given $0<n\leq d$, we say that a Borel measure $\mu$ on $\R^d$ is $n$-dimensional Ahlfors-David regular, or simply 
AD-regular, if there exists 
some constant $c_0$ such that $c_0^{-1}r^n\leq\mu(B(x,r))\leq c_0r^n$ for all $x\in\supp(\mu)$, $0<r\leq\diam(\supp(\mu))$.
It is not difficult to see that such a measure~$\mu$ must be of the form $\mu=\rho\,\HH^n\rest {\supp(\mu)}$, where
 $\rho$ is some positive function bounded from above and from below and $\HH^n$ stands for the $n$-dimensional
Hausdorff measure. A Borel set $E\subset \R^d$ is called AD-regular if the measure $\HH^n_{E}$ is AD-regular, where we used the notation $\HH^n_E = \HH^n\rest E$.

The notion of uniform $n$-rectifiability (or simply, uniform rectifiability) was introduced by David 
and Semmes in \cite{DS2}.
For $n=1$, an AD-regular $1$-dimensional measure is uniformly rectifiable if and only if its support is 
contained in an 
AD-regular curve.
For an arbitrary integer $n\geq1$, the notion is more complicated. One of the many equivalent definitions (see Chapter
 I.1 of \cite{DS2}) is the following: $\mu$ is uniformly rectifiable if there exist constants $\theta,M>0$ so that, for
each $x\in\supp(\mu)$ and $R>0$, there is a Lipschitz mapping $g$ from the $n$-dimensional ball $B_n(0,r)\subset\R^n$
into $\R^d$ such that $g$ has Lipschitz norm $\leq M$ and
$$\mu\bigl(B(x,r)\cap g(B_n(0,r))\bigr) \geq \theta r^n.$$
In the language of \cite{DS2}, this means that {\em $\supp(\mu)$ has big pieces of Lipschitz images of $\R^n$}.
A Borel set $E\subset\R^d$ is called uniformly rectifiable if $\HH^n_{E}$ is uniformly rectifiable.

The reason why uniform rectifiability  has attracted much attention in the last years is 
because this is the natural notion in many problems where rectifiability is involved in a quantitative way, such as in problems in connection with singular integral operators.
 Indeed, as shown in \cite{DS2}, it turns out that an $n$-dimensional AD-regular measure $\mu$ is 
uniformly rectifiable if and only if a sufficiently big class of singular integral operators with an $n$-dimensional odd Calder\'on-Zygmund kernel 
is bounded in $L^2(\mu)$. In the case $n=1$, it was shown in \cite{MMV} that 
uniform rectifiability is equivalent the $L^2(\mu)$ boundedness 
of only one operator, namely the Cauchy transform. The analogous result in higher dimensions involving the
$n$-dimensional Riesz transforms instead of the Cauchy transform is still open and has been subject of much investigation. See \cite{Mattila-Preiss}, \cite{MV}, \cite{MT}, \cite{Vihtila}, or \cite{Volberg}, besides \cite{DS1} and \cite{DS2}, for instance.

In this paper we will characterize uniform rectifiability in terms of some scale invariant coefficients
$\alpha_p$. They are defined as follows.
Let $\vphi:\R^d\to[0,1]$ be a radial Lipschitz function such that $\chi_{B(0,2)}\leq\vphi\leq \chi_{B(0,3)}$
which also satisfies
\begin{equation}\label{eqfi00}
c_2^{-1}\, \dist(x,\partial B(0,3))^2\leq\vphi(x)\leq c_2\, \dist(x,\partial B(0,3))^2,\!\!\!\qquad|\nabla \vphi(x)|\leq c_2 \dist(x,\partial B(0,3)),
\end{equation}
for all $x\in B(0,3)$. For a ball $B=B(x,r)\subset\R^d$, we denote
$$\vphi_B(y) = \vphi\Bigl(\frac{y-x}r\Bigr).$$
Given an $n$-dimensional AD-regular measure $\mu$ on $\R^d$ and a ball $B$ with radius $r(B)$
that intersects $\supp(\mu)$,
for $1\leq p <\infty$, we define
\begin{equation}\label{eqcbl}
\alpha_p(B) = \frac1{r(B)^{1+\frac np}}\,\inf_{L} \,W_p\bigl(\vphi_B\mu,\,
c_{B,L}\vphi_B\HH^n_{L}\bigr),
\end{equation}
where the infimum is taken over all $n$-dimensional affine planes that intersect $B$ and the constant $c_{B,L}$
is chosen so that the measures $\vphi_B\mu$ and $c_{B,L}\vphi_B\HH^n_{L}$ have the same mass. That is,
$
c_{B,L} =  \int \vphi_B\,d\mu/ \int \vphi_B\,d\HH^n_L$. If $B=B(x,r)$, we will also use the notation
$\alpha_p(x,r)=\alpha_p(B)$.

Since $\int\vphi_B\,d\mu\approx r(B)^n$ for any ball $B$ intersecting the support of an $n$-dimensional AD-regular
measure $\mu$, it turns out that
\begin{align}\label{eqdf541}
\alpha_p(B) &\approx \frac1{r(B)\left(\int\vphi_B\,d\mu\right)^{1/p}} \,
\inf_{L} \,W_p\bigl(\vphi_B\mu,\,c_{B,L}\vphi_B\HH^n_{L}\bigr) \\ &= 
\frac1{r(B)} \,
\inf_{L} \,W_p\left(\frac{\vphi_B\mu}{\int\vphi_B\,d\mu},\,c_{B,L}\vphi_B\HH^n_{L}\right)= \frac1{r(B)} \,
\inf_{L} \,W_p\left(\frac{\vphi_B\mu}{\int\vphi_B\,d\mu},\,\frac{\vphi_B\HH^n_{L}}{\int\vphi_B\,d\HH^n_{L}}\right).\nonumber
\end{align}
Observe that the coefficient $\alpha_p(B)$ measures in a scale invariant way how close
$\mu$ is to a flat $n$-dimensional measure in $3B$.
Recall that a flat $n$-dimensional measure is a measure of the form $c\,\HH^n_L$, where $c>0$ and $L$ is
an affine $n$-plane. In particular, notice that if $\mu$ is flat and $B\cap\supp(\mu)\neq\varnothing$, 
then $\alpha_p(B)=0$. Notice also that
$$\alpha_p(B)\lesssim \alpha_q(B)\qquad \text{if $p\leq q$.}$$
This follows easily from \rf{eqdf541}, taking into account that $W_p(\sigma,\nu)\leq W_q(\sigma,\nu)$ for any two probability measures $\sigma$ and $\nu$.

Let us remark that, in the case $p=1$, some coefficients quite similar to the $\alpha_1$'s were already
introduced in \cite{Tolsa-plms}, and they were denoted by $\alpha$ there. The precise relationship between the $\alpha_1$'s and the $\alpha$'s is
explained in Lemma \ref{lemalfas} below. 

The coefficients $\alpha_p$ should also be compared with the coefficients $\beta_p$, well known in the 
area of quantitative rectifiability.
Given $1\leq p < \infty$ and a ball $B\subset\R^d$, one sets
$$\beta_p(B) = \inf_L \,
\biggl\{ \frac1{r(B)^n}\int_{2B} \biggl(\frac{\dist(y,L)}{r(B)}\biggr)^pd\mu(y)\biggr\}^{1/p},$$
where the infimum is taken over all $n$-planes in $\R^d$.
For $p=\infty$ one has to replace the $L^p$ norm by a supremum:
$$\beta_\infty(B) = \inf_L\, \biggl\{ \sup_{y\in \supp(\mu)\cap 2B}
\frac{\dist(y,L)}{r(B)}\biggr\},$$
where the infimum is taken over all $n$-planes $L$ in $\R^{d}$ again.
The coefficients $\beta_p$ first appeared in \cite{Jones-Escorial} and \cite{Jones-traveling}, in the case $n=1$, 
$p=\infty$. In
 \cite{Jones-Escorial} P.~Jones showed, among other results, how the $\beta_\infty$'s can be used to prove the $L^2$
boundedness of the Cauchy transform on Lipschitz graphs. In \cite{Jones-traveling}, he characterized $1$-dimensional 
uniformly rectifiable sets in terms of the~$\beta_\infty$'s. He also obtained other quantitative results on 
rectifiability without the AD regularity assumption. For other $p$'s and $n\geq1$, the $\beta_p$'s were introduced
by David and Semmes in their pioneering study of uniform rectifiability in \cite{DS1}.

Notice that the $\beta_p$'s only give information on how close
$\supp(\mu)\cap 2B$ is to some $n$-plane (more precisely, how close is $\supp(\mu)\cap 2B$ to be contained in some $n$-plane). 
On the other hand, the coefficients $\alpha_p$ contain more precise information. 
Indeed, we will see in Lemma \ref{lembeta2} below that $\beta_p(B)\leq\alpha_p(B)$ for $1<p<\infty$.
It is immediate to check that the converse inequality fails, as $\beta_p(B)=0$ does not
force $\mu$ to be flat in~$2B$.

If $B,B'$ are two balls with comparable radii and $B\subset B'$, then it is straightforward to 
check that $\beta_p(B)\lesssim \beta_p(B')$. However, this property is very far from being clear for
the coefficients $\alpha_p$. So given an AD-regular measure $\mu$, 
suppose that $3B\subset B'$ and that $\mu(B)\approx\mu(B')\approx r(B)^n\approx r(B')^n$. Is it true that 
\begin{equation}\label{eqbbp}
\alpha_p(B)\lesssim\alpha_p(B')\,?
\end{equation}

This question leads naturally to the following.
Let $\mu,\nu$ be probability measures 
on $\R^d$ with finite $p$-th moments and $\vphi$ a non negative bump function 
smooth enough. Consider the measures $\vphi\mu$ and $\vphi\nu$. We wish to estimate how 
different they are in terms of the distance $W_p(\mu,\nu)$. More precisely, let $a=\int\vphi\,d\mu/
\int\vphi\,d\nu$ and assume that the mass of the $\vphi\mu$ is quite big, that is, 
$\int\vphi\,d\mu\geq c_1\approx1.$ Is it true that
$$W_p(\vphi\mu,\,a\vphi\nu)\leq c\,W_p(\mu,\nu)?,$$
with $c$ depending on $c_1$. Using Kantorovich's duality, it is easy to check that  this holds for $p=1$,
 assuming only $\vphi$ to be Lipschitz (see Subsection \ref{secw1}). 
We will show in this paper that this is also true for the quadratic
Wasserstein distance, assuming additional conditions on $\vphi$ and $\mu$. 
The precise result, which we will use to prove that \rf{eqbbp} holds for $p=2$,
is the following.

\begin{theorem}\label{teobola}
 Let $\mu,\nu$ be probability measures on $\R^n$ with finite second moments. Let $B\subset\R^n$ be a closed ball 
with radius $r(B)\geq c_3^{-1}$, 
 and suppose that
 $\mu\bigl(\interior B\bigr),\nu\bigl(\interior B\bigr)>0$ (where $\interior B$ stands for the interior of $B$). Assume that $\mu$ is
absolutely continuous with
 respect to the Lebesgue measure on $B$, so that $\mu\rest  B = f(x)\,dx$, where the density $f$ satisfies $c_4^{-1} \chi_{B}\leq f(x)\leq c_4\chi_B$,
for a.e.\ $x\in B$
 and for some constant $c_4>0$.
Let $\vphi:\R^n\to[0,1]$ be a radial Lipschitz function supported on $B$  such that
\begin{equation}\label{eqfi000}
\,c_2^{-1}\, \dist(x,\partial B)^2\leq\vphi(x)\leq c_2\, \dist(x,\partial B)^2\qquad\!\text{and}\!
\qquad|\nabla \vphi(x)|\leq c_2 \dist(x,\partial B),
\end{equation}
for all $x\in B$. 
 Then
$$W_2(\vphi\mu,\,a\,\vphi\nu)\leq c\, W_2(\mu,\nu),$$
where $a=\int\vphi\,d\mu/\int\vphi\,d\nu$ and $c$ depends only on $c_2$,$c_3$
and $c_4$.
\end{theorem}
\vvv

Notice that the assumptions on $\mu$ and $\vphi$ imply that $r(B)\approx1$ and $\int\vphi\,d\mu\gtrsim1$,
with constants depending on $c_2$, $c_3$, and $c_4$.

The result also holds if, instead of \rf{eqfi000}, one asks
$$\vphi(x)\approx  \dist(x,\partial B)^m \quad \text{ and }
\qquad|\nabla \vphi(x)|\lesssim \dist(x,\partial B)^{m-1},
$$
for $x\in B$ and $m\geq2$. However, for simplicity we have only considered the case $m=2$.

The preceding theorem follows from an analogous result where the ball $B$ is 
 replaced by a cube $R$. We carry out this reduction by a suitable change of coordinates in Subsection
 \ref{seccube}. 
 Roughly speaking, the proof of the corresponding theorem for the cube $R$ consists in 
 estimating $W_2(\vphi\mu,a\,T\#(\vphi\nu))$ in terms of $W_2(\mu,\nu)$, where $T$ is a map such that $T\#\nu=\mu$ which realizes the optimal quadratic transport (the notation $T\#\nu$ stands for the image measure of $\nu$ by $T$, i.e. $T\#\nu(A)=\nu(T^{-1}A)$ for $A\subset\R^n$).
This estimate is obtained by a multi-scale analysis, using Haar wavelets. Some of the
ideas are partially inspired by \cite{Tolsa-plms}.

Roughly speaking, the 
quadratic decay of $\vphi$ as one approaches the boundary is used to bound the interchanges of mass between
$\vphi\mu$ and $(1-\vphi)\nu$, and between $\vphi\nu$ and $(1-\vphi)\mu$.

With the preceding localization result for $W_2$ at hand, we will prove the following.

\begin{theorem}\label{teorectif}
Let $\mu$ be an AD-regular measure on $\R^d$ and $1\leq p\leq2$. Then $\mu$ is uniformly rectifiable if and only if $\alpha_p(x,r)^2\,d\mu(x)\,\frac{dr}r$ is a Carleson measure, that is, if for any ball $B$ 
centered on $\supp(\mu)$ with radius $R$,
\begin{equation}\label{eqcar53}
\int_0^R\!\!\!\int_{B}
 \alpha_p(x,r)^2\,d\mu(x)\,\frac{dr}r \leq c\,R^n.
\end{equation}
\end{theorem}

Some remarks are in order. First, let us mention that the same result is already known to hold with
$\beta_p$ instead of $\alpha_p$ (even with a somewhat wider range of $p$'s). This was
shown by David and Semmes in \cite{DS1}. Also, the case $p=1$ of the above theorem is a straightforward consequence
of an analogous result proved for the coefficients $\alpha$ in \cite{Tolsa-plms}, using the 
relationship between the $\alpha$'s and the $\alpha_1$'s given in Lemma 
\ref{lemalfas}.

Notice that if, for some $p\in[1,2]$, $\alpha_p(x,r)^2\,d\mu(x)\,\frac{dr}r$ is a Carleson measure, then so 
is $\beta_p(x,r)^2\,d\mu(x)\,\frac{dr}r$, since $\beta_p(x,r)\leq\alpha_p(x,r)$, and thus by the results of David and Semmes in \cite{DS1}
$\mu$ is uniformly rectifiable. So the difficult implication in the theorem consists in showing that
if $\mu$ is uniformly rectifiable, then $\alpha_p(x,r)^2\,d\mu(x)\,\frac{dr}r$ is a Carleson measure.
As $\alpha_p(x,r)\lesssim\alpha_2(x,r)$ for $p\leq 2$, it suffices to consider the case $p=2$ for 
this implication. To this end, the localization Theorem \ref{teobola} will play a key role.
First we will prove that $\alpha_2(x,r)^2\,d\mu(x)\,\frac{dr}r$ is a Carleson measure in the particular
case where $\mu$ is comparable to $\HH^n$ on an $n$-dimensional Lipschitz graph, and finally
we will prove the result in full generality by means of a geometric corona type decomposition.
This technique, which takes its name from the corona theorem of Carleson, has been adapted by David
and Semmes to the analysis of uniformly rectifiable sets \cite{DS1,DS2} and has already been shown to be 
useful in a variety of situations (see also, for example, \cite{Leger}, \cite{Tolsa-bilip}, or 
\cite{MT}).

The plan of the paper is the following. In Section \ref{secprelim} we introduce
some additional notation and terminology. The localization Theorem \ref{teobola} is proved in
Sections \ref{seclocal} and \ref{secappendix}.
In Section \ref{sec5} we explain the relationship among the coefficients $\alpha,\alpha_p,\beta_p$
and we show that $\alpha_2(B)\lesssim\alpha_2(B')$, under the assumptions just above \rf{eqbbp}.
In Section \ref{sec6} we prove Theorem \ref{teorectif} for the particular case of Lipschitz graphs,
while the full result is proved in the final Section \ref{sec7}.

\vvv

\vvv


\section{Preliminaries}\label{secprelim}

As usual, in the paper the letter `$c$' stands for an absolute
constant which may change its value at different occurrences. On
the other hand, constants with subscripts, such as $c_1$, retain
their value at different occurrences. The notation $A\lesssim B$
means that there is a positive absolute constant $c$ such that
$A\leq cB$. So $A\approx B$ is equivalent to $A\lesssim B \lesssim
A$. 

Given $x\in\R^n$, $|x|$ stands for its Euclidean norm and $|x|_\infty$ for its $\ell_\infty$ norm.

By a cube in $\R^n$ we mean a cube with edges parallel to the axes. Most of the cubes in our paper will be
dyadic cubes, which are assumed to be half open-closed. The collection of all dyadic cubes is denoted by $\DD$.
The side length of a cube $Q$ is written as $\ell(Q)$, and its center as $z_Q$. 
The lattice of dyadic cubes of side length $2^{-j}$ is
denoted by $\DD_j$. On the other hand, if $R$ is a cube, $\DD(R)$ stands for the collection of cubes contained
in $R$ that are obtained by splitting it dyadically, and $\DD_j(R)$ is the subfamily of those cubes from
$\DD(R)$ with side length $2^{-j}\ell(R)$. On the other hand, if $Q\in\DD$ or $Q\in\DD(R)$, then
$\CH(Q)$ is the family of dyadic children of $Q$. One says that $Q$ is the parent of its children. Two cubes
$Q,Q'\in\DD$ are called brothers if they have the same parent.
Also, given $a>0$ and any cube $Q$, we denote by $a\,Q$ the cube concentric with $Q$ with side length $a\,
\ell(Q)$. 

By a measure on $\R^n$, we mean a Radon measure. Its total variation is denoted by $\|\mu\|$.
Given a set $A\subset \R^n$, we write $\mu\rest A=\chi_A\,\mu$. That is $\mu\rest A$ is the 
restriction of $\mu$ to $A$.
The Lebesgue measure is denoted by $m$ or by $dx$. 
Given a measure $\mu$ and a cube $Q$, we set $m_Q\mu =\mu(Q)/m(Q)$. That is, $m_Q\mu$ is the mean of $\mu$ on $Q$, with respect to Lebesgue measure. $L^p$ is the usual $L^p$ space with respect to Lebesgue measure, while $L^p(\mu)$ is the one with respect to the measure $\mu$. As usual, we consider the pairing between two functions  $\langle f,\,g\rangle=\int f\,g\,dx$, or between
a function and a measure $\langle f,\,\mu\rangle=\int f\,d\mu$.

Given $A\subset\R^n$, we say that a measure $\mu$ on $\R^n$ is doubling on $A$ if there
exists a constant $c_d$ such that
$$\mu(B(x,2r))\leq c_d\,\mu(B(x,r))\qquad \text{for all $x\in A\cap\supp(\mu)$, $0<r\leq\diam(A).$}$$
In the case $A=\R^n$, we just simply say that $\mu$ is doubling.

The Hausdorff $s$-dimensional measure in $\R^n$ is denoted by $\HH^s$. Given a set $E\subset\R^n$,
we write $\HH^s_E=\HH^s\rest E$.

Concerning mass transport, recall that, by Brenier's theorem (see \cite[Chapter 2]{Villani-ams}, for 
example), if $\mu$ and $\nu$ are probability measures
which are absolutely continuous with respect to Lebesgue measure in $\R^n$, then there exists a unique
optimal transference plan (i.e. an optimal measure) $\pi$ for $W_2$, and there are maps $S,T:\R^n\to\R^n$ such that 
$\pi = (Id \times S)\#\mu = (T\times Id)\#\nu.$
So $S\#\mu=\nu$ and $T\#\nu=\mu$, and
$$W_2(\mu,\nu)^2 = \int_{\R^n\times\R^n} |x-y|^2\,d\pi(x,y)
= \int_{\R^n} |x-Sx|^2\,d\mu(x) = \int_{\R^n} |Ty-y|^2\,d\nu(y).$$ 
Moreover, $T\circ S=Id$ $\mu$-a.e.\ and $S\circ T=Id$ $\nu$-a.e.

Let us remark that the ambient space for the proof of Theorem \ref{teobola} below will be $\R^n$, while
the ambient space for the proof of Theorem \ref{teorectif} will be $\R^d$, and in this case we will
reserve the letter $n$ for the dimension of the measure $\mu$.


\section{The localization theorem for Wasserstein distances}\label{seclocal}


\vvv

\subsection{The case of $W_1$}\label{secw1}

In this subsection we prove a localization result for $W_1$ under much weaker assumptions on $\mu$ and $\vphi$
than the ones in Theorem \ref{teobola}.

\begin{propo}\label{propo1}
Let $\mu$ and $\nu$ be probability measures on $\R^n$ with finite first moments and let 
$B$ be a ball of radius $c_3^{-1}\leq r(B)\leq c_3$. Let
$\vphi:\R^n\to[0,\infty)$ be a 
Lipschitz function supported on $B$ with
$\|\nabla\vphi\|_\infty\leq c_6$. 
Suppose that $\int \vphi\,d\mu\geq c_7$ and $\int \vphi\,d\nu>0$.
Then,
\begin{equation}\label{eq310}
W_1(\vphi\mu,\,a\,\vphi\nu)\leq c\, W_1(\mu,\nu),
\end{equation}
where $a= \int \vphi\,d\mu/\int \vphi\,d\nu$, and $c$ depends on $c_3,c_6,c_7$.
\end{propo}

\vvv
\begin{proof}
We may assume that $W_1(\mu,\nu)$ is small enough.
Otherwise, the inequality \rf{eq310} is trivial (for $c$ big enough) because 
$\vphi\,\mu$ and $\vphi\,\nu$ are both supported on $B$ and $r(B)\leq c_3$. Now, notice that by the Kantorovich duality,
\begin{equation}\label{eq311}
\Bigl|\int\vphi\,d\mu - \int\vphi\,d\nu\Bigr| \leq \|\nabla\vphi\|_\infty\,W_1(\mu,\nu).
\end{equation}
Since $a=\int\vphi\,d\mu/\int\vphi\,d\nu$, we have
$$|a-1|= \frac{\Bigl|\int\vphi\,d\mu - \int\vphi\,d\nu\Bigr|}{\int\vphi\,d\nu}\leq
\frac{\|\nabla\vphi\|_\infty\,W_1(\mu,\nu)
}{\int\vphi\,d\nu}.$$
To estimate
$\int\vphi\,d\nu$ from below, we use \rf{eq311} again, and then we obtain
\begin{equation}\label{eqaaa***}
\int\vphi\,d\nu \geq \int\vphi\,d\mu - \|\nabla\vphi\|_\infty\,W_1(\mu,\nu).
\end{equation}
Since  $\int\vphi\,d\mu\geq c_7$, if $W_1(\mu,\nu)$ is small enough, we get 
$\int\vphi\,d\nu \geq  c_7/2.$ Therefore,
\begin{equation}\label{eqaaa}
|a-1|\leq c\,W_1(\mu,\nu),
\end{equation}
with $c$ depending also on $\|\nabla\vphi\|_\infty$. Let $\psi$ be an arbitrary $1$-Lipchitz function
such that $\psi(0)=0$.
By Kantorovich's duality again,
\begin{align*}
\Bigl|\int\psi\,\vphi\,d\mu - \int\psi\,a\,\vphi\,d\nu\Bigr| &\leq 
\Bigl|\int\psi\,\vphi\,d\mu - \int\psi\,\vphi\,d\nu\Bigr|+
|1-a|\int\psi\,\vphi\,d\nu\\
& \leq \|\nabla (\psi\vphi)  \|_\infty\,W_1(\mu,\nu) +c\int|\psi\,\vphi|\,d\nu\,\,W_1(\mu,\nu).
\end{align*}
By the mean value theorem, it is easy to check that $\vphi$ and $|\psi|$ are bounded uniformly above
on $B$ by some constant depending on $c_3$ and $c_6$. Then it follows that
$\|\nabla (\psi\vphi)  \|_\infty +\int|\psi\,\vphi|\,d\nu\lesssim1$, and we deduce \rf{eq310}.
\end{proof}

\vvv

\begin{rem}\label{remaaa}
Notice that \rf{eq311}, \rf{eqaaa***}, \rf{eqaaa} tell us that
$$\int\vphi\,d\nu\approx1 \qquad\mbox{and}\qquad |a-1|\leq c\,W_1(\mu,\nu)\leq c\,W_2(\mu,\nu),$$
assuming $W_1(\mu,\nu)$ small enough.
Clearly, this also holds under the assumptions of Theorem \ref{teobola}, which are more restrictive than
the ones in the preceding proposition.
\end{rem}

\vvv


\subsection{From a ball $B$ to a cube $R$ in Theorem \ref{teobola}}\label{seccube}

\begin{definition}
Given a cube $R$ and a function $\vphi:R\to[0,1]$, we write $\vphi\in \GZ_0(R)$ if $\vphi$ is Lipschitz, 
$\vphi(x)\approx 1$ on $\frac12 R$, and there
exists a constant $c_2$ such that
$$\vphi(x)\leq c_2 \frac{\dist(x,\partial R)^2}{\ell(R)^2 }
\quad\text{ and } \quad|\nabla \vphi(x)|\leq c_2 \frac{\vphi(x)}{\dist(x,\partial R)}
\quad\text{for all $x\in R$}.$$

On the other hand, we write $\vphi\in \GZ(R)$ if 
$\vphi:R\to[0,1]$ is 
 Lipschitz  and, for all $x\in R$, satisfies
$$c_2^{-1} \frac{\dist(x,\partial R)^2}{\ell(R)^2 }\leq\vphi(x)\leq c_2 \frac{\dist(x,\partial R)^2}{\ell(R)^2 },
$$
and 
$$|\nabla \vphi(x)|\leq c_2 \frac{\dist(x,\partial R)}{\ell(R)^2},\qquad|\nabla_T \vphi(x)|\leq c_2 \frac{\dist(x,\partial R)^2}{\ell(R)^3}.
$$
 Here, $|\nabla_T \vphi(x)|= \sup_{ v} |v\cdot \nabla\vphi(x)|$, where the supremum is taken over
all unit vectors $v$ which are parallel to the face of $R$ which is closest to $x$. If the face is not unique,
then we set $\nabla_T \vphi(x)=0$.
\end{definition}

For example, if $\vphi$ is a function supported on $R$ such that $\vphi(x)\approx1$ and 
is smooth  in a neighborhood of $\frac12R$, and
$$\vphi(x) = \frac{\dist(x,\partial R)^2}{\ell(R)^2} \qquad \text{for $x\in R\setminus \frac12R$,}$$
 then $\vphi\in \GZ(R)$. To check that 
$|\nabla_T \vphi(x)|\leq c_2 \dist(x,\partial R)^2/\ell(R)^3$, take $x\in R\setminus \frac12R$ 
such that there exists
a unique face closest to $x$, and let $t$ 
be a vector parallel to this face. Then, $\dist(x+\ve \,t,\,\partial R) = 
\dist(x,\,\partial R)$ for $|\ve|$ small enough, and thus $\nabla_T \vphi(x)=0$.

\vvv

In the next subsections we will prove the following result.

\vvv
\begin{theorem}\label{teoloc}
 Let $\mu,\nu$ be probability measures on $\R^n$ with finite second moments. Let $R\subset\R^n$ be the closed cube 
 with side length $\geq c_5$,
 and suppose that
 $\mu(\interior R),\nu(\interior R)>0$. Assume that $\mu$ is
absolutely continuous with
 respect to the Lebesgue measure on $R$, so that $\mu\lefthalfcup  R = f(x)\,dx$, where the density $f$
satisfies
 $c_4^{-1} \chi_{R}\leq f(x)\leq c_4\, \chi_R$ for a.e.\ $x\in R$
 and for some constant $c_4>0$.  If $\vphi\in \GZ(R)$, then
$$W_2(\vphi\mu,\,a\,\vphi\nu)\leq c\, W_2(\mu,\nu),$$
where $a>0$ is chosen so that $\int\vphi\,d\mu=a\int\vphi\,d\nu$ and $c$ depends only on $c_2$,
$c_4$ and $c_5$.
\end{theorem}
\vvv

Let see how Theorem \ref{teobola} follows from the preceding result.

\vvv

\begin{proof}[\bf Proof of Theorem \ref{teobola} using Theorem \ref{teoloc}]
Suppose first that $R=[-1,1]^n$ and $B$ is the unit ball,
both centered at the origin. Consider the map $F:\R^n\to \R^n$ defined by 
$$F(x) = \frac{|x|}{|x|_\infty}\,x.$$
This maps balls centered at the origin of radius $r$ to concentric cubes with side length $2r$.
In particular, $F(B)=R$. 
It is easy to check that $F$ is bilipschitz. Moreover, it can be checked that its Jacobian equals
$J(F)(x) = 2|x|^n/|x|^n_\infty\approx1.$ 

Now we consider the measures $F\#\mu$ and $F\#\nu$. Notice that 
$$dF\#\mu(x) = J(F^{-1})(x) \,f(F^{-1}(x))\,dx.$$
Thus the density of $F\#\mu$ is bounded above and below in $R$. 
On the other hand, in can be checked that $\vphi\circ F^{-1}\in \GZ(R)$.
Indeed, notice that 
since $\vphi$ is radial, then $\vphi\circ F^{-1}$ is constant on the boundaries of the cubes centered at the 
origin.

Then, from Theorem \ref{teoloc} we deduce that
\begin{equation}\label{eq309}
W_2\bigl((\vphi\circ F^{-1}) F\#\mu,\,a(\vphi\circ F^{-1})F\#\nu\bigr) \lesssim W_2(F\#\mu,F\#\nu).
\end{equation}
Since $F$ if bilipschitz,
$W_2(F\#\mu,F\#\nu) \approx W_2(\mu,\nu).$
Also, since $F\#(\vphi\mu) = (\vphi\circ F^{-1})F\#\mu$ and analogously for $\nu$, we have
$$W_2\bigl((\vphi\circ F^{-1}) F\#\mu,a(\vphi\circ F^{-1})F\#\nu\bigr)\approx W_2(\vphi\mu,a\vphi\nu),$$
and so the theorem follows from \rf{eq309}.

In case $B$ is not the unit ball or $R$ is not the unit cube, we just compose $F$ with an affine map
to obtain another $\wt F$ such that $\wt F(R)=B$ and does the job.
Then, we argue with $\wt F$ instead of $F$.
\end{proof}

\vvv

\subsection{Preliminary lemmas for the proof of Theorem \ref{teoloc}}

 In next lemma we recall the properties of the so called {\em Whitney decomposition} of a proper open set.
\vv

\begin{lemma} \label{lemwhitney}
An open subset $\Omega\subsetneq\R^n$ can be decomposed as follows:
$$\Omega = \bigcup_{k=1}^{\infty} Q_k, $$
where $Q_k$ are disjoint dyadic cubes (the so called ''Whitney cubes'') such
that for some constants $r>20$ and $D_0\geq1$ the following holds,
\begin{itemize}
\item[(i)] $5Q_k \subset \Omega$.
\item[(ii)] $r Q_k \cap \Omega^{c} \neq \varnothing$.
\item[(iii)] For each cube $Q_k$, there are at most $D_0$ squares $Q_j$
such that $5Q_k \cap 5Q_j \neq \varnothing$. Moreover, for such squares $Q_k$, $Q_j$, we have 
$\frac12\ell(Q_k)\leq
\ell(Q_j)\leq 2\,\ell(Q_k)$.
\end{itemize}
\end{lemma}
\vvv

This is a well known result. See for example
\cite[pp.\ 167-169]{Stein0} for the proof.
\vv

Roughly speaking, next lemma deals with the existence of a good big subset $G\subset R$ such 
that the mass on $G$ is transported not too far. 
\vvv

\begin{lemma}\label{lem11}
Let $\mu,\nu$ be probability measures on $\R^n$ with finite second moments which are absolutely continuous with
 respect to the Lebesgue measure.  Take $T$ such that 
$T\#\nu = \mu$ and 
$$W_2(\mu,\nu)^2 = \int|Tx-x|^2\,d\nu(x).$$
 Let $R\subset\R^n$ be a cube and let 
 $\{Q_i\}_{i\in I}$ be a Whitney decomposition of $\inter{R}$ as in Lemma \ref{lemwhitney}.
For $x\in \inter R$, let $Q_x$ be the Whitney cube $Q_i$ that contains $x$. Denote
\begin{equation}\label{eqdefg12}
G= \bigl\{x\in R:\,|x-Tx|\leq \ell(Q_x)\bigr\} \cap 
\bigl\{x\in T^{-1}(R):\,|x-Tx|\leq \ell(Q_{Tx})\bigr\}.
\end{equation}
Let $\vphi\in \GZ_0(R)$ and suppose that
\begin{equation}\label{eqd422} 
W_2(\mu,\nu)^2\leq c_8 \,(\vphi\nu)(\R^n)\,\ell(R)^2,
\end{equation}
where $c_8$ is some positive constant small enough. Then $\int_G\vphi\,d\nu>0$.

 Consider the measure
$\wt \nu= \wt a\,\chi_G\,\vphi\,\nu,$
with $\wt a =(\vphi\nu)(\R^n)/(\vphi\nu)(G)$ (so that $(\vphi\nu)(\R^n)=\wt\nu(\R^n)$).
Then we have
\begin{equation}\label{eqd423a} 
\|\vphi\nu - \wt\nu\| \leq \frac{c}{\ell(R)^2}\,W_2(\mu,\nu)^2,
\end{equation}
and consequently,
\begin{equation}\label{eqd423} 
W_2(\vphi\nu,\wt\nu) \lesssim W_2(\mu,\nu).
\end{equation}
\end{lemma}
\vv

Below, to simplify notation we will write $\vphi\nu(E):=(\vphi\nu)(E)$ for any $E\subset\R^n$, and 
analogously for $\mu$.

\vv
\begin{proof}
The last estimate follows from \rf{eqd423a}, because of the control of the Wasserstein distances by the total variation.
Indeed, recall that for two arbitrary finite measures $\sigma$ and $\tau$ with finite $p$-th moments and the same mass, 
$$W_p(\sigma,\tau)\leq 2^{1-\frac1p} \left( \int |x-x_0|^p\,d|\sigma-\tau|(x)\right)^{1/p},$$ 
for $1\leq p<\infty$ and any arbitrary $x_0\in\R^d$. See \cite[Theorem 6.13]{Villani-oldnew}, for example. So one infers 
that
$$W_2(\vphi\nu,\wt\nu)\leq 2^{1/2}\,\diam(R)\,
\|\vphi\nu - \wt\nu\|^{1/2},$$
since  $\supp(\vphi\nu)\cup\supp(\wt\nu)\subset R$. Together with \rf{eqd423a}, this yields \rf{eqd423}.

To prove \rf{eqd423a}, we write
\begin{align}\label{eqd3100}
\|\vphi\nu - \wt\nu\| & \leq \int |1- \wt a|\,\vphi\,d\nu + 
\wt a\int|\vphi- \vphi\chi_G|\,d\nu\\
& = |1- \wt a|\,\vphi\nu(\R^n) + \wt a \int_{G^c}\vphi\,d\nu.\nonumber
\end{align}
To estimate  $\int_{G^c}\vphi\,d\nu$, 
denote by $G_1$ and $G_2$ the two sets appearing on the right side of \rf{eqdefg12}, respectively.
Then we set
$$\int_{G^c}\vphi\,d\nu = \int_{G_1^c}\vphi\,d\nu + \int_{G_1\cap G_2^c}\vphi\,d\nu = I_1+ I_2.$$
We will use the fact that for every $x\in R$, 
$$\vphi(x)\lesssim\frac{\ell(Q_x)^2}{\ell(R)^2}.
$$
Since $|Tx-x|\geq \ell(Q_x)$ for $x\in R\setminus G_1$, we deduce
$$\vphi(x)\lesssim \frac{|Tx-x|^2}{\ell(R)^2} \qquad \text{for $x\in R\setminus G_1$}.
$$
Therefore,
\begin{equation}\label{eqd32}
I_1\lesssim \frac{1}{\ell(R)^2}\int_{\supp(\vphi)} |Tx-x|^2\,d\nu
\leq \frac{1}{\ell(R)^2}\,W_2(\mu,\nu)^2.
\end{equation}

Let us turn our attention to the integral 
$$I_2 = \int_{R\cap G_1\cap G_2^c}\vphi\,d\nu.$$
Observe that if $x\in R\cap G_1\cap G_2^c$, then $|x-Tx|\leq \ell(Q_x)$ and so $Tx\in 3Q_x\subset R$.
 Thus $3Q_{x}\cap Q_{Tx}\neq\varnothing$. Thus,
 $$\ell(Q_x)\approx \ell(Q_{Tx})\leq |x-Tx|,$$
using that $Tx\in R$ and $x\nin G_2$ for the last inequality.
Therefore,
$$\vphi(x)\lesssim\frac{\ell(Q_x)^2}{\ell(R)^2} \lesssim \frac{|Tx-x|^2}{\ell(R)^2},$$
and then \rf{eqd32} also holds with $I_2$ instead of $I_1$ (multiplying by a constant if necessary). Thus,
\begin{equation}\label{eqd32'}
\int_{G^c}\vphi\,d\nu\lesssim \frac{1}{\ell(R)^2}\,W_2(\mu,\nu)^2.
\end{equation}

Now we deal with the term $|1- \wt a|$. Observe that 
$\wt a = \int \vphi\,d\nu/\int_G\vphi\,d\nu,$
and so
$$\wt a -1 = 
\frac{\int_{G^c}\vphi\,d\nu}{\int_{G}\vphi\,d\nu}
.
$$
From the assumption \rf{eqd422} and \rf{eqd32'} we infer that 
\begin{equation}\label{eqdf5}
\int_{G}\vphi\,d\nu \geq \frac12
 \int\vphi\,d\nu =\frac12\,\vphi\nu(\R^n)
\end{equation}
if $c_8$ is chosen small enough,
and thus 
$\wt a\leq2$ and
\begin{equation}\label{eqaat}
|1- \wt a| \leq \frac2{\vphi\nu(\R^n)}\vphi\nu(G^c)
\lesssim \frac1{\vphi\nu(\R^n)\,\ell(R)^{2}}\,W_2(\mu,\nu)^2.
\end{equation}
Plugging this estimate and \rf{eqd32} into \rf{eqd3100}, the lemma follows.
\end{proof}
\vvv

\begin{rem}\label{rema41}
If $W_2(\mu,\nu)$ is small enough, then from Remark \ref{remaaa}
it turns out that $\vphi\nu(\R^n)\approx 1$, and then from \rf{eqaat} we deduce
$$|1- \wt a| \lesssim W_2(\mu,\nu)^2.$$
\end{rem}

\vvv

\begin{lemma}\label{lem12}
Suppose the same notation and assumptions of Lemma \ref{lem11}. In particular, assume that
\begin{equation}\label{eqd422'} 
W_2(\mu,\nu)^2\leq c_8 \,(\vphi\mu)(\R^n)\,\ell(R)^2,
\end{equation}
where $c_8$ is some positive constant small enough. Then  $\vphi\bigl(T\#(\chi_G\,\nu)\bigr)(\R^n)
>0$.

Consider the measure
$$
\wt\mu = \wt b\,\vphi\bigl(T\#(\chi_G\,\nu)\bigr)$$
with $\wt b$ chosen so that $(\vphi\mu)(\R^n)=\wt\mu(\R^n)$.
Then we have
\begin{equation}\label{eqd423a'} 
\|\vphi\mu - \wt\mu\| \leq \frac{c}{\ell(R)^2}\,W_2(\mu,\nu)^2,
\end{equation}
and consequently,
\begin{equation}\label{eqd423'} 
W_2(\vphi\mu,\wt\mu) \lesssim W_2(\mu,\nu).
\end{equation}
\end{lemma}
\vv

\vv
\begin{proof}
The arguments are quite similar to the ones of the preceding lemma. However, for completeness, we
show the details. The last estimate follows from \rf{eqd423a'}, taking into account that
 $\supp(\vphi\mu)\cup\supp(\wt\mu)\subset R$.

To prove \rf{eqd423a'}, we write
\begin{align}\label{eqd31}
\|\vphi\mu - \wt\mu\| & \leq \int |1- \wt b|\,\vphi\,d\mu + 
\wt b\,\biggl|\int\bigl(\vphi \,d T\#\nu- \vphi\,dT\#(\chi_G\,\nu)\bigr)\biggr|\\
& = |1-\wt b|\,\vphi\mu(\R^n) + \wt b \int_{G^c}\vphi(T(x))\,d\nu(x).\nonumber
\end{align}

To estimate the last integral on the right side, 
denote by $G_1$ and $G_2$ the two sets appearing on the right side of \rf{eqdefg12} respectively.
Then we have
$$\int_{G^c}\vphi(T(x))\,d\nu(x) = \int_{G_2^c}\vphi(T(x))\,d\nu(x) + \int_{G_2\cap G_1^c}\vphi(T(x))\,d\nu(x)
 = I_1+ I_2.$$
First we deal with $I_1$. Notice that we may assume that for $x$ in the domain of integration of $I_1$ 
(and $I_2$), we 
have $x\in T^{-1}(R)$. Since $x\not\in G_2$, this implies that 
$|x-Tx|>\ell(Q_{T_x})$, and thus
$$\vphi(Tx)\lesssim \frac{\ell(Q_{Tx})^2}{\ell(R)^2}\leq \frac{|x-Tx|^2}{\ell(R)^2}.$$
Therefore,
\begin{equation}\label{eqd32*}
I_1\lesssim \frac{1}{\ell(R)^2}\int |Tx-x|^2\,d\nu(x)
= \frac{1}{\ell(R)^2}\,W_2(\mu,\nu)^2.
\end{equation}

Let us turn our attention to the integral 
$I_2 = \int_{G_2\cap G_1^c}\vphi(Tx)\,d\nu(x).$
Observe that if $x\in T^{-1}(R)\cap G_2\cap G_1^c$, then $|x-Tx|\leq\ell(Q_{Tx})$, and so
$x\in 3Q_{Tx}\subset R$. Thus, $Q_x\cap 3Q_{Tx}\neq\varnothing$, and thus 
$\ell(Q_x)\approx\ell(Q_{T_x})$.
So we deduce
$$\vphi(Tx)\lesssim \frac{\ell(Q_{Tx})^2}{\ell(R)^2} \lesssim \frac{\ell(Q_{x})^2}{\ell(R)^2}
\leq \frac{|x-Tx|^2}{\ell(R)^2},$$
using that $x\in R\setminus G_1$ in the last inequality.
Then \rf{eqd32*} also holds with $I_2$ instead of $I_1$. Thus,
\begin{equation}\label{eqd32*'}
\int_{G^c}\vphi(Tx)\,d\nu(x)\lesssim \frac{1}{\ell(R)^2}\,W_2(\mu,\nu)^2.
\end{equation}

Now we deal with the term $|1- \wt b|$. Observe that 
$\wt b = \int \vphi\circ T\,d\nu/\int_G\vphi\circ T\,d\nu,$
and so
$$\wt b -1 = 
\frac{\int_{G^c}\vphi\circ T\,d\nu}{\int_G\vphi\circ T\,d\nu}
.
$$
From the assumption \rf{eqd422'} and \rf{eqd32*'} we infer that 
$$\int_{G^c}\vphi\circ T\,d\nu\lesssim c_8\,\vphi\mu(\R^n).$$
Therefore,
\begin{equation}\label{eqdf5'}
\int_{G}\vphi\circ T\,d\nu = \vphi\mu(\R^n) - 
\int_{G^c}\vphi\circ T\,d\nu\geq \frac12\,\vphi\mu(\R^n),
\end{equation}
choosing $c_8$ small enough.
Thus $\wt b\leq2$ and
$$|1- \wt b\,| \leq \frac2{\vphi\mu(\R^n)}\int_{G^c}\vphi\circ T\,d\nu
\lesssim \frac1{\vphi\mu(\R^n)\,\ell(R)^{2}}\,W_2(\mu,\nu)^2.$$
Plugging this estimate and \rf{eqd32*'} into \rf{eqd31}, the lemma follows.
\end{proof}

\vvv

\begin{rem}\label{remb41}
For the record, observe that the inequality \rf{eqd32*'} says that
$$\|\vphi\mu- \wt b\,^{-1}\wt \mu\|\lesssim \frac{1}{\ell(R)^2}\,W_2(\mu,\nu)^2.
$$
\end{rem}

\vvv


\subsection{The key lemma}

We need the following auxiliary result.

\begin{lemma}\label{lemkr0}
Let  $R$ and $\vphi\in \GZ(R)$ be as in Theorem \ref{teoloc}. 
Consider the measure $d\wt\sigma=\vphi\,dm$ and let 
$Q\in\DD(R)$. Consider a bounded function $h$ supported on $Q$ such that $\|h\|_\infty\leq1$ and $\int h\,d\wt\sigma=0$.
Then there exists a map $U:Q\to Q$ such that
$U\# (\wt\sigma\rest Q) = (1+h)\,\wt\sigma\rest Q$ which satisfies
$$
|Ux - x|\lesssim \ell(Q)\,\|h\|_\infty\qquad\text{for all $x\in Q$.}
$$
\end{lemma}

\vvv
We will prove this lemma by means of the Knothe-Rosenblatt coupling. The arguments involve 
some calculations that are quite lengthy
and rather tedious, and so we will defer the proof to  Section \ref{secappendix}.

\begin{lemma}\label{lemkr1}
Let $\mu$, $R$, and $\vphi\in \GZ(R)$ be as in Theorem \ref{teoloc}. 
Denote $\sigma=\vphi\mu$ and let 
$Q\in\DD(R)$. Consider a bounded function $h$ supported on $Q$ such that $\|h\|_\infty\leq1$ and $\int h\,d\sigma=0$.
Then 
$$W_2(\sigma\rest Q,\,(1+h)\,\sigma\rest Q)\lesssim \ell(Q)\,\|h\|_\infty\sigma(Q)^{1/2}.$$
\end{lemma}

\begin{proof}
Recall that the density of $\mu$ satisfies $c_4^{-1}\leq f(x) \leq c_4$ for $x\in R$.
So if we consider the measure $d\wt\sigma= \frac12\,c_4^{-1}\vphi\,dm$, then we have 
$2\wt\sigma\leq \sigma\leq 2c_4^2\,\wt\sigma$. Notice also that we may assume that $\|h\|_\infty$ is small enough, because otherwise the lemma is trivial. So we suppose that $\|h\|_\infty\leq c_4^{-2}/4$. 
We write $\sigma = (\sigma - \wt\sigma) + \wt\sigma$, and
$$(1+h)\,\sigma = (\sigma - \wt\sigma) + \wt\sigma + h \,\frac{d\sigma}{d\wt\sigma}\,\wt\sigma =
(\sigma - \wt\sigma) + \Bigl(1+ 2c_4\,f\,h\Bigr)\,\wt\sigma.$$
Notice that the function $\wt h= 2c_4\,f\,h$ satisfies $\|\wt h\|_\infty\leq 1/2$ and $\int_Q \wt h\,d\wt\sigma
=0$.
By Lemma \ref{lemkr0} applied to $\wt\sigma$ and $\wt h$, we deduce that
there exists a map $\wt U:Q\to Q$ such that
$\wt U\# (\wt\sigma\rest Q) = (1+\wt h)\,\wt\sigma\rest Q$ which satisfies
$$
|\wt Ux - x|\lesssim \ell(Q)\,\|\wt h\|_\infty\approx \ell(Q)\,\|h\|_\infty\qquad\text{for all $x\in Q$.}
$$

Since $\sigma-\wt\sigma$ is a positive measure, we can consider the transference plan
$$d\pi(x,y) = d(\sigma-\wt\sigma)\rest Q(x)\,\delta_{x=y} + (Id \times \wt U)\#\wt\sigma\rest Q(x,y).$$
The marginal measures of $\pi$ are  $\sigma\rest Q$ and 
$$\sigma\rest Q-\wt\sigma\rest Q + \wt U\#
\wt\sigma\rest Q = \sigma\rest Q + \wt h \,\wt\sigma\rest Q = (1+h)\,\sigma\rest Q.$$
Then,
$$W_2(\sigma\rest Q,\,(1+h)\,\sigma\rest Q)^2\leq\int|x-y|^2\,d\pi(x,y) = 
\int_Q |x-\wt U(x)|^2\,d\wt\sigma\lesssim \ell(Q)^2\,\|h\|_\infty^2\,\wt\sigma(Q).$$
Since $\wt\sigma(Q)\approx\sigma(Q)$, the lemma follows.
\end{proof}

The idea of subtracting a smaller nicer measure $\wt\sigma$ to $\sigma$ is inspired by some techniques from Peyre \cite{Peyre}.

Notice, by the way, that the measures $\wt\sigma$ and $\sigma$ in the previous lemmas are doubling, because of the assumptions on $\mu$ and
$\vphi$. This property is used strongly in many of the arguments below.

\vvv

The following is an easy consequence of Lemma \ref{lemkr1}.

\begin{lemma}\label{lemkr2}
Let $\mu$, $R$, and $\vphi\in \GZ(R)$ be as in Theorem \ref{teoloc}. 
Denote $\sigma=\vphi\mu$ and let $Q\in\DD(R)$. 
Consider a function $h$ supported on $Q$ that is constant on the children of $Q$ and vanishes out 
of $Q$, that 
is,
$$h=\sum_{P\in\CH(Q)}c_P\,\chi_P.$$
Assume moreover that $\int h\,d\sigma=0$ and $\|h\|_\infty\leq1$. Then,
$$W_2(\sigma\rest Q,\,(1+h)\,\sigma\rest Q)\lesssim \ell(Q)\,\|h\|_\infty\sigma(Q)^{1/2}
\approx \ell(Q)\,\|h\|_{L^2(\sigma)}
.$$
\end{lemma}

\vvv
Observe that the first inequality is the one that follows from Lemma \ref{lemkr1}, while the second one is a direct consequence
of the special form of $h$, using that $\sigma(P)\approx\sigma(Q)$ for all $P\in \CH(Q)$.

\vvv
Given measures $\sigma, \tau$ and a dyadic cube $Q$, 
we consider the function defined by
$$\Delta_Q^\sigma \tau(x) = \left\{\begin{array}{ll}\dfrac{\tau(P)}{\sigma(P)} - \dfrac{\tau(Q)}{\sigma(Q)}\qquad &\text{if $x\in P\in\CH(Q)$,}\\ &\text{}\\
0 &\text{if $x\not\in Q$,}\end{array}
\right.$$
assuming that $\sigma(P),\sigma(Q)>0$ for the cubes $P,Q$ above.
If $\sigma$ coincides with the Lebesgue measure on $\R^n$, then we set $\Delta_Q \tau:=\Delta_Q^\sigma \tau$.
If $\tau = g(x)\,dx$, then we write $\Delta_Q g :=\Delta (g\,dx) = \Delta_Q\tau$. Recall that if $g\in L^2$, then 
\begin{equation}\label{eqortog}
g=\sum_{Q\in\DD} \Delta_Q g,
\end{equation}
with the sum converging in $L^2$.
Moreover, the functions $\Delta_Q \tau$, $Q\in\DD$, are mutually orthogonal, and thus 
$\|g\|_2^2 = \sum_{Q\in\DD} \|\Delta_Q g\|_2^2.$
In fact, \rf{eqortog} coincides with the decomposition of $g$ into Haar wavelets: recall
that $g$ can be written as
$$g=\sum_{Q\in\DD} \sum_{\ve\in\EE}\langle g,\,h_Q^\ve\rangle\,h_Q^\ve,$$
where $\EE=\{0,1\}^n\setminus (0,\ldots,0)$
and each $h_Q^\ve$ is a Haar $n$-dimensional wavelet supported on $Q$, associated to the index   
$\ve\in \EE$. Then it is easy to check that for every $Q\in\DD$,
$$\Delta_Q g = \sum_{\ve\in\EE}\langle g,\,h_Q^\ve\rangle\,h_Q^\ve,$$
which we will also write as
$\Delta_Q \tau = \sum_{\ve\in\EE}\langle \tau,\,h_Q^\ve\rangle\,h_Q^\ve.$
See \cite[Part I]{David-wavelets} for more details, for example.

Now we are going to introduce the definition of a {\em tree} of dyadic cubes. For $R\in\DD$, consider a family
$\SSS$ (possibly empty) of disjoint cubes from $\DD(R)$ which satisfies the following property:
$$\mbox{if $Q\in\SSS$, then every brother of $Q$ also belongs to $\SSS$.}$$
 Denote 
$$\TT= \bigl\{Q\in\DD(R):\,Q\mbox{ is not contained in any cube from $\SSS$}\bigr\}.$$
Then we say that $\TT$ is a tree with {\em root} $R$, and the cubes from the family $\SSS$ are called the {\em stopping cubes of $\TT$}. Notice that in our the definition, the stopping cubes do not belong to $\TT$.
Observe also that
$\TT$ has the property that if $Q\in
\TT$ and $P\in\DD(R)$ is such that $Q\subset P\subset R$, then $P\in\TT$ too. This is why $\TT$ is called  tree.
 
\vv
Using Lemma \ref{lemkr2} as a kind of building block, we will prove the following.\vv

\begin{lemma}[Key lemma]\label{lemkey}
Let $\mu$, $R$, and $\vphi\in \GZ(R)$ be as in Theorem \ref{teoloc}. 
Denote $\sigma=\vphi\mu$, let $\TT$ be a tree with root $R$ and $\SSS\subset R$ be its family of stopping cubes.
Consider
a measure $\tau$ supported on $R$ such that $\tau(R)=\sigma(R)$ and
 $\tau(Q)\approx\vphi\mu(Q)=\sigma(Q)$ for all $Q\in \TT$.
Then, for any $\alpha>0$, we have
\begin{align}\label{eqkey7}
W_2(\sigma,\tau)^2& \lesssim 
\sum_{Q\in \TT}\frac1{m_Q\sigma}\,\|\Delta_Q(\sigma-\tau)\|_{L^2}^2\,\ell(Q)^{2-\alpha}\ell(R)^\alpha\\
& \quad + \sum_{Q\in \TT} \frac{|m_Q\sigma - m_Q\tau|^2}{(m_Q\sigma)^3}\,
\,\|\Delta_Q\sigma\|_{L^2}^2\,\ell(Q)^{2-\alpha}\ell(R)^\alpha  
 + \sum_{Q\in \SSS} \ell(Q)^2\,\tau(Q),\nonumber
\end{align}
with constants depending on $\alpha$.
\end{lemma}

\begin{proof} For simplicity suppose first that $\bigcup_{Q\in\SSS}Q= R$ and that $\TT$ contains a finite number
of cubes (and thus all stopping cubes have side length uniformly bounded from below).
For $Q\in\DD(R)$, consider the functions $\Delta_Q^\sigma \tau$ defined above. 
Then, $\tau$ can be written as 
$$\tau = \sigma + \biggl(\sum_{Q\in\TT}\Delta_Q^\sigma \tau\biggr)\,\sigma + \sum_{Q\in\SSS}\biggl(\tau\rest Q -
\frac{\tau(Q)}{\sigma(Q)}\,\sigma\rest Q\biggr).$$
Observe that the measures
$\bigr(\Delta_Q^\sigma \tau\bigl)\,\sigma$, for $Q\in\TT$, have zero mean, as well as all the measures inside the last parenthesis above.

For $Q\in\TT$, consider a map $U_Q:Q\to Q$ such that
$$U_Q\#\biggl(\frac{\tau(Q)}{\sigma(Q)}\,\sigma\rest Q\biggr) = \biggl(\Delta_Q^\sigma\tau
 + \frac{\tau(Q)}{\sigma(Q)}\biggr)\,\sigma\rest Q,$$
or,
 equivalently, 
 $$U_Q\#\bigl(\sigma\rest Q\bigr) = \biggl(\frac{\sigma(Q)}{\tau(Q)}\,\Delta_Q^\sigma\tau
 + 1\biggr)\,\sigma\rest Q.$$
Since $\tau(Q)/\sigma(Q)\approx1$, by Lemma \ref{lemkr2} applied to $h=\frac{\sigma(Q)}{\tau(Q)}\,\Delta_Q^\sigma\tau$, we can take such map $U_Q$ so that
\begin{equation}\label{eqint51}
\int |U_Q x-x|^2\,d\sigma(x)\lesssim \|\Delta_Q^\sigma\tau\|_{L^2(\sigma)}^2\,\ell(Q)^2.
\end{equation}

For $j>0$, define
$U_j(x)=U_Q(x)$ if $x\in Q\in\DD_j(R)\cap\TT$, and
$U_j(x)=x$ if $x\in Q\in\DD_j(R)\setminus\TT$. 
Notice that, by definition, we have
$$U_j
\biggl(\,\sum_{Q\in\DD_j(R)\cap\TT} \frac{\tau(Q)}{\sigma(Q)}\,\sigma\rest Q\biggr) = \!
\sum_{Q\in\DD_j(R)\cap\TT} \biggl(\Delta_Q^\sigma\tau
 + \frac{\tau(Q)}{\sigma(Q)}\biggr)\,\sigma\rest Q =\!\!
\sum_{Q\in\DD_j(R)\cap\TT}\sum_{P\in\CH(Q)} \frac{\tau(P)}{\sigma(P)}\,\sigma\rest P.$$
Set also $U_0(x)=x$ and
consider the map
$$T= U_m\circ U_{m-1}\circ\cdots \circ U_1,$$
for $m$ big enough so that $\TT\subset \bigcup_{j=0}^m D_j(R)$.
Observe that
$$T\#\sigma = \sum_{Q\in\SSS}\frac{\tau(Q)}{\sigma(Q)}\,\sigma\rest Q =:\tau_0.$$
The measure $\tau_0$ should be considered as an approximation of $\tau$ at the scale of the stopping cubes.
Let $V_j = U_j\circ \ldots\circ U_0$, for $0\leq j\leq m$. We set
\begin{equation}\label{eqjj0}
|x-Tx|^2\leq \Bigl(\sum_{j=1}^m |V_{j-1}(x) - V_{j}(x)|\Bigr)^2 \lesssim 
\sum_{j=1}^m 2^{\alpha j}|V_{j-1}(x) - V_{j}(x)|^2,
\end{equation}
with a constant depending on $\alpha$.
As a consequence, using that $V_j = U_j\circ V_{j-1}$, 
$$W_2(\sigma,\tau_0)^2 \leq \int \sum_{j=1}^m 2^{\alpha j}|V_{j-1}x - V_{j}x|^2\,d\sigma(x)=
 \sum_{j=1}^m 2^{\alpha j}\int |x - U_jx|^2\,d V_{j-1}\#\sigma(x)
.$$
Recalling the definition of $U_j$, we get
$$\int |x - U_jx|^2\,d V_{j-1}\#\sigma(x) = \sum_{Q\in\DD_j(R)\cap\TT} \int |x - U_Q x|^2\,d V_{j-1}\#
\sigma(x).$$
For the cubes $Q$ in the last sum it turns out that $$V_{j-1}\#
\sigma\rest Q=\frac{\tau(Q)}{\sigma(Q)}\,\sigma\rest Q\qquad \text{and}\qquad\tau(Q)\approx \sigma(Q),$$
 and so we obtain
$$W_2(\sigma,\tau_0)^2 \lesssim \sum_{Q\in\TT} \frac{\ell(R)^\alpha}{\ell(Q)^\alpha}
\int |x - U_Q x|^2\,d\sigma(x).$$
From \rf{eqint51}, we infer that
\begin{equation}\label{eqdf6245}
W_2(\sigma,\tau_0)^2 \lesssim \sum_{Q\in\TT} \frac{\ell(R)^\alpha}{\ell(Q)^\alpha}\,
\|\Delta_Q^\sigma\tau\|_{L^2(\sigma)}^2\,\ell(Q)^2 = \sum_{Q\in\TT} 
\|\Delta_Q^\sigma\tau\|_{L^2(\sigma)}^2\,\ell(Q)^{2-\alpha}\,\ell(R)^\alpha.
\end{equation}

Now we wish to write $\Delta_Q^\sigma\tau$ in terms of $\Delta_Q(\sigma-\tau)$. Given $x\in P\in\CH(Q)$,
from the definition of $\Delta_Q^\sigma\tau$, we get
\begin{align}\label{eqdf623}
\Delta_Q^\sigma \tau(x) & = \frac{\tau(P)-\sigma(P)}{\sigma(P)} - \frac{\tau(Q)-\sigma(Q)}{\sigma(Q)}\\
& = \frac1{m_P\sigma}\biggl(\frac{\tau(P)-\sigma(P)}{m(P)} - \frac{\tau(Q)-\sigma(Q)}{m(Q)}\biggr) +
\biggl(\frac1{m_P\sigma} - \frac1{m_Q\sigma}\biggr)\frac{\tau(Q)-\sigma(Q)}{m(Q)} \nonumber\\
& =
\frac1{m_P\sigma}\,\Delta_Q(\sigma-\tau)(x) + 
\frac1{m_P\sigma\,m_Q\sigma}\,\Delta_Q\sigma(x)\,\frac{\sigma(Q)-\tau(Q)}{m(Q)}.\nonumber
\end{align}
Taking into account that, for $Q\in\DD(R)$ and $P\in\CH(Q)$, we have $m_P\sigma\approx m_Q\sigma$ (because of the 
assumptions on $\mu$ and $\vphi$) and since $\Delta_Q^\sigma\tau$ is a function that is constant on the cubes 
$P\in\CH(Q)$, we infer that $\|\Delta_Q^\sigma\tau\|_{L^2(\sigma)}^2\approx 
m_Q\sigma\,\|\Delta_Q^\sigma\tau\|_{L^2}^2$.
Thus, 
by \rf{eqdf623},
$$\|\Delta_Q^\sigma\tau\|_{L^2(\sigma)}^2\lesssim \frac1{m_Q\sigma}\,\|\Delta_Q(\sigma-\tau)\|_{L^2}^2
+ 
 \frac{|m_Q\sigma - m_Q\tau|^2}{(m_Q\sigma)^3}\,
\,\|\Delta_Q\sigma\|_{L^2}^2.$$
Plugging this estimate into \rf{eqdf6245}, we get
\begin{align*}
W_2(\sigma,\tau_0)^2  & \lesssim 
\sum_{Q\in \TT}\frac1{m_Q\sigma}\,\|\Delta_Q(\sigma-\tau)\|_{L^2}^2\,\ell(Q)^{2-\alpha}\ell(R)^\alpha\\
&\quad + \sum_{Q\in \TT} \frac{|m_Q\sigma - m_Q\tau|^2}{(m_Q\sigma)^3}\,
\,\|\Delta_Q\sigma\|_{L^2}^2\,\ell(Q)^{2-\alpha}\ell(R)^\alpha.
\end{align*}

To complete the proof of the lemma, it is enough to show that
\begin{equation}\label{eqw22}
W_2(\tau,\tau_0)^2\lesssim \sum_{Q\in \SSS} \ell(Q)^2\,\tau(Q).
\end{equation}
To check this, just take
$$\pi = \sum_{Q\in\SSS}\frac1{\tau(Q)}\,(\tau\rest Q \otimes \tau_0\rest Q).$$
Using that $\tau(Q)=\tau_0(Q)$  for all $Q\in\SSS$, it is easy to check that
$\pi$ has marginals $\tau$ and $\tau_0$. Therefore,
\begin{align*}
W_2(\tau,\tau_0)^2 & \leq \int |x-y|^2\,d\pi(x,y) =
\sum_{Q\in\SSS}\frac1{\tau(Q)}\,\int_{Q\times Q} |x-y|^2\,d\tau(x)\,d\tau_0(y)\\
& \leq 2\sum_{Q\in\SSS}\left(\int_{Q} |x-z_Q|^2\,d\tau(x) + \int_{Q} |y-z_Q|^2\,d\tau_0(y)\right)\lesssim \sum_{Q\in\SSS}\ell(Q)^2\tau(Q),
\end{align*}
which gives \rf{eqw22}. This ends the proof of the lemma 
under the assumption that $\TT$ contains finitely many cubes.

If $\TT$ has infinitely many cubes, then we apply the lemma to the tree  $\TT_m=\TT\cap\DD_m(R)$ 
for any fixed $m>0$, and we let $m\to\infty$.
\end{proof}

\begin{rem}\label{remkey}
It is very easy to check that the Lemmas \ref{lemkr0}, \ref{lemkr1}, \ref{lemkr2}, and \ref{lemkey} also hold if $\vphi\approx\chi_R$ and 
$\vphi$ is Lipschitz on $R$. In fact, the calculations much easier in this situation (specially the ones
in Lemma \ref{lemkr0}).
In the particular case where $\sigma=c\chi_R\,\LL^n$, where $\LL^n$ stands for the Lebesgue measure
and the constant $c$ is comparable to $1$,
Lemma \ref{lemkey} becomes
\begin{align}\label{eqkey777}
W_2(\sigma,\tau)^2& \lesssim 
\sum_{Q\in \TT}\|\Delta_Q \tau\|_{L^2}^2\,\ell(Q)^{2-\alpha}\ell(R)^\alpha
 + \sum_{Q\in \SSS} \ell(Q)^2\,\tau(Q),
\end{align}
still assuming that 
 $\tau(Q)\approx\sigma(Q)$ for all $Q\in \TT$.
\end{rem}


\vvv

\subsection{ Proof of Theorem \ref{teoloc}}
 
We have to prove that $W_2(\vphi\mu,a\vphi\nu)\lesssim  W_2(\mu,\nu).$
To this end, we may assume that
\begin{equation}\label{eqass9}
W_2(\mu,\nu)^2\leq c_8 \,\ell(R)^2\,\|\vphi\mu\|.
\end{equation}
Otherwise, the theorem follows easily, because
$$W_2(\vphi\mu,a\vphi\nu)^2\lesssim \ell(R)^2\,\|\vphi\mu\|\lesssim W_2(\mu,\nu)^2.$$
Moreover, it also easy to check that $\mu$ and $\nu$ can be assumed to be absolutely continuous with respect to Lebesgue 
measure, by an approximation argument. Indeed, consider a non negative $\CC^\infty$ radial function $\psi$ supported on $B(0,1)$ 
with $L^1$ norm equal to $1$, and for $t>0$ denote $\psi_t(x)=t^{-n}\psi(t^{-1}x)$. Consider the measures given
by the convolutions $\mu_t=\psi_t*\mu$ and $\nu_t=\psi_t*\nu$. They are absolutely continuous with respect to 
Lebesgue measure and moreover,  $\mu_t\to\mu$ and $\nu_t\to\nu$
weakly as $t\to0$. By Theorem 7.12 of \cite{Villani-ams}, it turns out that  
$W_2(\mu,\mu_t)\to 0$ and $W_2(\nu,\nu_t)\to 0$, and thus $W_2(\mu_t,\nu_t)\to W_2(\mu,\nu)$. By similar arguments, 
$W_2(\vphi\mu,a_t\vphi\mu_t)\to 0$ and $W_2(\vphi\nu,b_t\vphi\mu_t)\to 0$, where $a_t= \int\vphi d\mu
/\int\vphi d\mu_t$ and  $b_t= \int\vphi d\nu
/\int\vphi d\nu_t$. In particular, $a_t,b_t\to1$ as $t\to0$. Then we have
\begin{align*}
W_2(\vphi\mu,a\vphi\nu)& \leq 
W_2(\vphi\mu,a_t\vphi\mu_t)+ W_2(a_t\vphi\mu_t,a\,b_t\vphi\nu_t)+ W_2(a\,b_t\vphi\nu_t,a\vphi\nu)\\
& \lesssim
W_2(\vphi\mu,a_t\vphi\mu_t)+ W_2(\mu_t,\nu_t)+ W_2(a\,b_t\vphi\nu_t,a\vphi\nu),
\end{align*}
assuming that the theorem holds for absolutely continuous measures for the last inequality.
Letting $t\to0$, we infer that $W_2(\vphi\mu,a\vphi\nu)\lesssim W_2(\mu,\nu)$, which proves our claim.

So suppose that \rf{eqass9} holds and that $\mu$ and $\nu$ are absolutely continuous with respect to Lebesgue measure.
Let $G,T$ and $\wt\mu,\,\wt\nu$ be as in Lemmas \ref{lem11} and \ref{lem12}. We  write
$$W_2(\vphi\mu,a\vphi\nu) \leq W_2(\vphi\mu, a\,T\#\wt\nu) + W_2(a\,T\#\wt\nu,a\wt\nu) + 
W_2(a\wt\nu,a\vphi\nu).$$
Notice that, by Lemma \ref{lem11}, using also that $a\approx1$, we have
$$W_2(a\vphi\nu,a\wt\nu) \lesssim W_2(\mu,\nu).$$
Also,
$$W_2(a\,T\#\wt\nu,a\,\wt\nu)^2 \leq a\int |Tx-x|^2\,d\wt\nu 
\leq  c\int |Tx-x|^2\,d\nu = c\,W_2(\mu,\nu)^2,$$
where the second inequality follows from the fact that $\wt\nu = \wt a\,\chi_{G}\,\vphi\nu$, 
with $\wt a\lesssim1$ (see Lemma \ref{lem11}).
Thus, to prove the theorem it is enough to show that
\begin{equation}\label{eqd56}
W_2(\vphi\mu, a\,T\#\wt\nu) \lesssim W_2(\mu,\nu).
\end{equation}
To this end, we will apply Lemma \ref{lemkey} to $\sigma=\vphi\mu$ and 
$\tau= a\,T\#\wt\nu$ (observe that $\sigma(R) = a\vphi\nu(R) =a\wt\nu(R) = \tau(R)$). We construct a tree $\TT$ whose family $\SSS$ of stopping cubes is defined as follows:
\begin{align}\label{eqal51} 
Q\in \SSS\quad &\text{if $Q$ is a maximal cube from $\DD(R)$ such that } \tau(Q)\leq \delta\,\sigma(Q)
\end{align}
for some small constant $0<\delta<1$ to be chosen below.
That $Q$ is maximal means that there does not exist another cube $Q'\in\DD(R)$ which contains $Q$ satisfying the same estimate. It is easy to check that the cubes from $\SSS$ are disjoint. Also, 
from the previous definition, we infer that every $Q\in\TT$
satisfies
$ \tau(Q)> \delta\,\sigma(Q)$.
To apply Lemma \ref{lemkey} to $\sigma$, $\tau$ and to the tree $\TT$ we need to show that there exists some 
constant $c$ such that $\tau(Q)\leq
c\,\sigma(Q)$. This follows quite easily:
$$
\tau(Q) = a\,T\#\wt\nu(Q) = a\,\wt a\,T\#(\chi_G\,\vphi\,\nu)(Q)
=a\,\wt a\,\int\chi_Q(Tx)\,\chi_G(x)\, \vphi(x)\,d\nu(x).
$$
From Remarks \ref{remaaa} and \ref{rema41}, taking into account the assumption \rf{eqass9} we deduce
that $a\,\wt a\approx1$. On the other hand, for $x\in G$, we have $|x-Tx|\leq \ell(Q_x)$, where
$Q_x$ is the Whitney cube that contains $x$. This implies that $\dist(x,\partial R)\approx
\dist(Tx,\partial R)$ and thus $\vphi(x)\approx \vphi(Tx)$. Therefore,
\begin{align}\label{eqtauq1}
\tau(Q)& \approx\int\chi_Q(Tx)\,\chi_G(x)\, \vphi(Tx)\,d\nu(x)\\&
\leq \int\chi_Q(Tx)\, \vphi(Tx)\,d\nu(x) = \int\chi_Q(x)\, \vphi(x)\,dT\#\nu(x)=\sigma(Q),\nonumber
\end{align}
which proves our claim.

Now the key Lemma \ref{lemkey} tells us that
\begin{align}\label{eqde9}
W_2(\sigma,\tau)^2& \lesssim 
\sum_{Q\in \TT}\frac1{m_Q\sigma}\,\biggl(\|\Delta_Q(\sigma-\tau)\|_{L^2}
+ \frac{|m_Q\sigma - m_Q\tau|\,\|\Delta_Q\sigma\|_{L^2}}{m_Q\sigma}\biggr)^2
\,\ell(Q)\ell(R)\\
& \quad   
 + \sum_{Q\in \SSS} \ell(Q)^2\,\tau(Q),\nonumber
\end{align}
with constants depending on $\delta$.
We will estimate each of the terms on the right side of this inequality in separate lemmas. The first one deals with 
 $\sum_{Q\in \SSS} \ell(Q)^2\,\tau(Q)$.

\vvv

\begin{lemma}\label{lemstop}
If $\delta$ is chosen small enough, then
$$\sum_{Q\in \SSS} \ell(Q)^2\,\tau(Q)\lesssim W_2(\mu,\nu)^2.$$
\end{lemma}

\begin{proof}
We will show that 
\begin{equation}\label{eqcla73}
\tau\Bigl(\,\bigcup_{Q\in\SSS} Q\Bigr)\lesssim \frac{W_2(\mu,\nu)^2}{\ell(R)^2}.
\end{equation}
Clearly, the lemma follows from this estimate.

For a given $Q\in\SSS$, notice that the first estimate in \rf{eqtauq1} says that
$$\tau(Q) \approx \int\chi_Q(Tx)\,\chi_G(x)\, \vphi(Tx)\,d\nu(x)=:I(Q),$$
while the last equality in \rf{eqtauq1} states that
$$\sigma(Q) = \int\chi_Q(Tx)\, \vphi(Tx)\,d\nu(x).$$
Therefore, if $\delta$ is chosen small enough, then $I(Q)\leq \sigma(Q)/2$, and thus
$$\tau(Q)\approx I(Q)\leq \sigma(Q) - I(Q) = \int\chi_Q(Tx)\, \vphi(Tx) \,\bigl(1-\chi_G(x)\bigr)\,d\nu(x).$$
Summing over all $Q\in \SSS$, we obtain
\begin{align*}
\tau\Bigl(\,\bigcup_{Q\in\SSS} Q\Bigr)&\lesssim
\int\chi_R(Tx)\, \vphi(Tx) \,\bigl(1-\chi_G(x)\bigr)\,d\nu(x)\\&
= \int_R \vphi(x) \,d(T\#\nu - T\#\chi_G\nu)(x)
.\nonumber
\end{align*}
Recall now that $T\#\nu=\mu$ and $\wt\mu = \wt b\,\vphi\bigl(T\#(\chi_G\,\nu)\bigr)$ (see Lemma \ref{lem12}). So the preceding inequality says that
$$\tau\Bigl(\,\bigcup_{Q\in\SSS} Q\Bigr)\lesssim \|\vphi\mu- {\wt b\,}^{-1}\,\wt\mu\|.$$
Hence, the claim \rf{eqcla73} follows from this inequality and Remark \ref{remb41}.
\end{proof}

\vvv
To deal with the first sum on the right of \rf{eqde9}, first we will estimate the $L^2$ norm of $\Delta_P(\sigma-\tau)$. In the rest of this subsection, we denote by $S$ the map such that $S\#\mu=\nu$ which
minimizes $W_2(\mu,\nu)$. Recall that $S\circ T=Id$ $\nu$-a.e. and $T\circ S=Id$ $\mu$-a.e.
Also,
 $\{Q_i\}_{i\in I}$ stands for a Whitney decomposition of $\inter{R}$, as in Lemma \ref{lemwhitney}.

\begin{lemma}\label{lemdelta}
Assume that \rf{eqass9} holds and denote $\eta=\vphi\,T\#(\chi_{G^c}\,\nu)$.
For $P\in\TT$, we have
\begin{equation}\label{eqrep92}
\|\Delta_P(\sigma-\tau)\|_{L^2}\lesssim \frac{\eta(P)}{\ell(P)^{n/2}}
+ \frac{W_2(\mu,\nu)}{\ell(P)^{n/2}\,
  \ell(R)^{1+n/2}}\,
\int_{P\cap S^{-1}(G)} \vphi\circ S\,d\mu + E(P),
\end{equation}
where $E(P)$ is as follows.
If $P$ is contained in some Whitney cube $Q_P$ from the family $\{Q_i\}_{i\in I}$, then
$$E(P) = \frac{\ell(Q_P)}{\ell(R)^2}\,\|Sx-x\|_{L^p(\mu\rest P)}.$$
On the other hand, if $P$ is not contained in any Whitney cube from the family $\{Q_i\}_{i\in I}$,
then
$$E(P)= \biggl(\,
\sum_{i:3Q_i\cap P\neq\varnothing}
 \frac{\ell(Q_i)^2}{\ell(R)^4}\,\|Sx-x\|_{L^2(\mu\rest3Q_i)}^2\biggr)^{1/2}.$$
\end{lemma}

\begin{proof}
Let $h^\ve_P$ be the Haar wavelet associated to $P$ and $\ve\in\EE$. Then we have
$$\|\Delta_P(\sigma-\tau)\|_{L^2} \leq 
\sum_{\ve\in \EE} \bigl|\langle \sigma-\tau,\,h^\ve_P\rangle\bigr|=\sum_{\ve\in \EE}
\left| \int  h^\ve_P\,d(\sigma-\tau)\right|.$$
Recall that $\sigma =\vphi\mu = \vphi\,T\#\nu$ and $\tau= a\,T\#\wt\nu = a
\,T\#(\wt a\,\chi_{G}\,\vphi\nu)$. So, for each $\ve\in\EE$, the last integral equals
\begin{align}\label{eqde45}
\int  h^\ve_P\,d[\vphi\,T\#\nu-a\,T\#(\wt a\,\chi_{G}\,\vphi\nu)] & =
\int h^\ve_P(Tx)\, \bigl[\vphi(Tx) - a\,\wt a\,\vphi(x)\chi_{G}(x)\bigr]\,d\nu(x)\\
& = \int_{G^c} h^\ve_P(Tx)\, \vphi(Tx) \,d\nu(x) \nonumber\\
& \quad + \int_G h^\ve_P(Tx)\, \bigl[\vphi(Tx) - \vphi(x)\bigr]\,d\nu(x)\nonumber\\
& \quad + (1-a\wt a)\int_G h^\ve_P(Tx)\, \vphi(x)\,d\nu(x) =: I_1 + I_2 + I_3. \nonumber
\end{align}

Concerning $I_1$, using again that $\|h_p^\ve\|_\infty\lesssim\ell(P)^{-n/2}$ and $\eta=\vphi\,T\#(\chi_{G^c}\,\nu)$, we get
\begin{equation}\label{eqi1}
|I_1| \lesssim{\ell(P)^{-n/2}}
 \int_{G^c} \chi_P(Tx)\,\vphi(Tx) \,d\nu(x) = \frac{\eta(P)}{\ell(P)^{n/2}}.
\end{equation}
To deal with $I_2$,
recall that, for $x\in Q_i\cap G$, we have $|x-Tx|\leq \ell(Q_i)$, and thus $Tx\in 3Q_i$ and
$$|\vphi(Tx) - \vphi(x)|\leq \|\nabla\vphi\|_{\infty,3Q_i}\,|Tx-x|.$$
Using that
$\|\nabla\vphi\|_{\infty,3Q_i}\lesssim\ell(Q_i)/\ell(R)^2,$
we get
\begin{equation}\label{eqdi9}
|I_2|\lesssim \sum_{i\in I} \frac{\ell(Q_i)}{\ell(R)^2}
\int_{Q_i\cap G}  |h^\ve_P(Tx)|\,|Tx - x|\,d\nu(x).
\end{equation}

Now we distinguish the two cases in the statement of the lemma, according to whether   $P$ is contained in 
some square $Q_i$, $i\in I$,  or not.
In the first case we write $P\in\TT_a$, and in the second $P\in\TT_b$. 
Suppose first that $P\in\TT_a$, and denote by $Q_P$ the square $Q_i$, $i\in I$, that contains $P$.
Thus, if for $x\in Q_i\cap G$, $h^\ve_P(Tx)\neq0$, then  
$Tx\in P\subset Q_P$ and thus $x\in 3Q_P$ (since $|x-Tx|
\leq\ell(Q_P)$). As a consequence, $Q_i\cap 3Q_P\neq \varnothing$, and there is bounded number of such cubes 
$Q_i$. Moreover, for them we have $\ell(Q_i)\approx\ell(Q_P)$. Thus we obtain
\begin{equation}\label{eqi2}
|I_2|\lesssim \sum_{i\in I:Q_i\cap 3Q_P\neq\varnothing}
\frac{\ell(Q_P)}{\ell(R)^2}\,
\|h^\ve_P\circ T\|_{L^2(\nu)}\,\|(Tx-x)\chi_{T^{-1}(P)\cap G}\|_{L^2(\nu)}.
\end{equation}
We have
$$\|h^\ve_P\circ T\|_{L^2(\nu)} = \|h^\ve_P\|_{L^2(\mu)}\lesssim  1,$$
since $\|h^\ve_P\|_{L^2}=1$ and $\mu=f(x)\,dx$ on $R$, with $f(x)\leq c_4$.
 Therefore, using also that $\nu = S\#\mu$, we get
$$|I_2|\lesssim\frac{\ell(Q_P)}{\ell(R)^2}\,\|(Tx-x)\chi_{T^{-1}(P)}\|_{L^2(\nu)}
=\frac{\ell(Q_P)}{\ell(R)^2}\,\|(x-Sx)\chi_{P}\|_{L^p(\mu)}.$$

Suppose now that there does not exist any cube $Q_i$, $i\in I$, which contains $P$. Notice that if, for some $i
\in I$, the integral on the right side of \rf{eqdi9} does not vanish, then there exists some $x\in Q_i\cap G$
such that $Tx\in P$. From the fact that $x\in G$, we deduce that $|x-Tx|\leq \ell(Q_i)$, and
thus $Tx\in 3Q_i$. As a consequence, $P\cap 3Q_i\neq\varnothing$.
Therefore,
\begin{align}\label{eqal82}
|I_2| & \lesssim \sum_{i:3Q_i\cap P\neq\varnothing}
 \frac{\ell(Q_i)}{\ell(R)^2}
\int_{Q_i\cap G}  |h^\ve_P(Tx)|\,|Tx - x|\,d\nu(x)\\
& \lesssim \ell(P)^{-n/2}
\sum_{i:3Q_i\cap P\neq\varnothing}
 \frac{\ell(Q_i)}{\ell(R)^2}
\int_{Q_i\cap G} |Tx - x|\,d\nu(x),\nonumber
\end{align}
since $\|h^\ve_P\|_\infty\lesssim\ell(P)^{-n/2}$. 
Using that $\nu = S\#\mu$ and that $TSx=x$ $\mu$-a.e., we have
$$\int_{Q_i\cap G} |Tx - x|\,d\nu(x) = 
\int\chi_{Q_i\cap G}(Sx)\, |x - Sx|\,d\mu(x).$$
Observe that for $\mu$-a.e $x$ such that $Sx\in G\cap Q_i$, 
$$|Sx-x| = |Sx - TSx|\leq  \ell(Q_{Sx})=\ell(Q_i).$$
Thus, $x\in 3Q_i$, and so we infer that
$$\int_{Q_i\cap G} |Tx - x|\,d\nu(x) \leq
\int_{3Q_i} |x - Sx|\,d\mu(x).$$

From \rf{eqal82} and the preceding estimate, by Cauchy-Schwartz, we obtain 
\begin{align*}
|I_2| &\lesssim \ell(P)^{-n/2}
\sum_{i:3Q_i\cap P\neq\varnothing}
 \frac{\ell(Q_i)\,\mu(3Q_i)^{1/2}}{\ell(R)^2}\,\|Sx-x\|_{L^2(\mu\rest3Q_i)}\\
& \leq \ell(P)^{-n/2}\biggl(\,
\sum_{i:3Q_i\cap P\neq\varnothing}
 \frac{\ell(Q_i)^2}{\ell(R)^4}\,\|Sx-x\|_{L^2(\mu\rest3Q_i)}^2\biggr)^{1/2}
 \biggl(\,
\sum_{i:3Q_i\cap P\neq\varnothing}
\mu(3Q_i)\biggr)^{1/2}\nonumber.
 \end{align*}
Taking into account 
$ \sum_{i:3Q_i\cap P\neq\varnothing}
\mu(3Q_i)\lesssim \ell(P)^n$, we deduce
$$|I_2| \lesssim \biggl(\,\sum_{i:3Q_i\cap P\neq\varnothing}
 \frac{\ell(Q_i)^2}{\ell(R)^4}\,\|Sx-x\|_{L^2(\mu\rest3Q_i)}^2\biggr)^{1/2}.$$

Finally we consider the term $I_3$ in \rf{eqde45}. Recall that in Remark \ref{remaaa} we saw that
$|1-a|\lesssim W_2(\mu,\nu) \approx W_2(\mu,\nu)/\ell(R)^{1+n/2}$. On the other hand, from \rf{eqaat} we get
$$|1- \wt a| \leq \frac{2}{\vphi\nu(\R^n)\,\ell(R)^{2}}\,W_2(\mu,\nu)^2\lesssim 
\frac{1}{\ell(R)^{n+2}}\,W_2(\mu,\nu)^2 \leq \frac{1}{\ell(R)^{1+n/2}}W_2(\mu,\nu),$$
where the last inequality follows from the estimate \rf{eqass9}. Therefore,
$$|1-a\wt a|\leq |1-a| + a\,|1-\wt a| \lesssim \frac{1}{\ell(R)^{1+n/2}}W_2(\mu,\nu).$$
Then, using that $\nu=S\#\mu$,
\begin{align*}
|I_3| & \lesssim \frac{W_2(\mu,\nu)}{\ell(R)^{1+n/2}} \left|\int_G h^\ve_P(Tx)\, \vphi(x)\,d\nu(x)\right|  \\
& = \frac{W_2(\mu,\nu)}{\ell(R)^{1+n/2}}
  \left|\int_{S^{-1}(G)} h^\ve_P(x)\, \vphi(Sx)\,d\mu(x) \right|\leq \frac{W_2(\mu,\nu)}{\ell(P)^{n/2}\,
  \ell(R)^{1+n/2}}\,
\int_{P\cap S^{-1}(G)} \vphi\circ S\,d\mu. \nonumber
\end{align*}
\end{proof}

\vvv

In next lemma we estimate the term $|m_P\sigma - m_P\tau|$.

\begin{lemma}\label{lemdiffe}
Assume that \rf{eqass9} holds and let $\eta=\vphi\,T\#(\chi_{G^c}\,\nu)$.
For $P\in\TT$, we have
$$\frac{|m_P\sigma - m_P\tau|\,\|\Delta_P\sigma\|_{L^2}}{m_P\sigma}\lesssim
\frac{\eta(P)}{\ell(P)^{n/2}}
+ \frac{W_2(\mu,\nu)}{\ell(P)^{n/2}\,
  \ell(R)^{1+n/2}}\,
\int_{P\cap S^{-1}(G)} \vphi\circ S\,d\mu + E(P),
$$
with $E(P)$ defined as in Lemma \ref{lemdelta}.
\end{lemma}
\vv

\begin{proof} 
From the doubling properties of $\sigma$, it follows that $\|\Delta_P\sigma\|_\infty\lesssim m_P\sigma$, and
thus
$$\frac{|m_P\sigma - m_P\tau|\,\|\Delta_P\sigma\|_{L^2}}{m_P\sigma}\lesssim 
|m_P\sigma - m_P\tau|\,\ell(P)^{n/2} = \bigl|\langle \sigma-\tau,\,w_P\rangle\bigr|,$$
with $w_P= \chi_P/\ell(P)^{n/2}$.

Now the arguments to finish the proof are very similar to the ones in the preceding lemma.
 Instead of 
$\langle \sigma-\tau,\,h^\ve_P\rangle$, we have to estimate the term
$\langle \sigma-\tau,\,w_P\rangle$.
So, all we have to do is to replace $h^\ve_P$ by $w_P$ in the preceding proof. Indeed, notice that 
the cancellation property of $h_P^\ve$ was not necessary in the Lemma \ref{lemdelta}. We only used that $\supp h^\ve_P\subset P$, that  $\|h^\ve_P\|_\infty\lesssim\ell(P)^{-n/2}$, and that
$\|h^\ve_P\|_{L^2}=1$. These estimates also hold for $w_P$, and so we deduce that \rf{eqrep92} also
holds replacing its left side by $|m_P\sigma - m_P\tau|\,\|\Delta_P\sigma\|_{L^2}/{m_P\sigma}$.
\end{proof}

\vvv

Our next objective consists in estimating the first sum on the right side of \rf{eqde9}. 
By Lemmas \ref{lemdelta} and \ref{lemdiffe}, we have
\begin{align}\label{eqde10}
\sum_{P\in \TT}\frac1{m_P\sigma}\,\biggl(\|&\Delta_P(\sigma-\tau)\|_{L^2}
+ \frac{|m_P\sigma - m_P\tau|\,\|\Delta_P\sigma\|_{L^2}}{m_P\sigma}\biggr)^2
\,\ell(P)\ell(R)
 \\
& \lesssim \sum_{P\in \TT} \frac{\ell(P)\ell(R)}{m_P\sigma} \,  E(P)^2 +
\sum_{P\in \TT} 
\frac{\ell(R)\,\eta(P)^2}{\ell(P)^{n-1}\,m_P\sigma}  \nonumber\\
&\quad + W_2(\mu,\nu)^2\sum_{P\in \TT}   
\frac{1}{m_P\sigma\,\ell(P)^{n-1}\,\ell(R)^{n+1}}\,
\left(\int_{P\cap S^{-1}(G)} \vphi\circ S\,d\mu\right)
^2\nonumber\\
  & =: S_1 + S_2 + S_3,\nonumber
\end{align}
where $\eta=\vphi\,T\#(\chi_{G^c}\,\nu)$.
\vvv

In the next lemma we consider the sum $S_1$:

\begin{lemma}\label{lems1}
We have
$$S_1 = \sum_{P\in \TT} \frac{\ell(P)\ell(R)}{m_P\sigma} \,  E(P)^2 \lesssim W_2(\mu,\nu)^2.$$
\end{lemma}

\begin{proof}
We separate the sum $S_1$ into two sums, according to whether $P$ is contained in 
some square $Q_i$, $i\in I$, or not. As above, in the first case we write $P\in\TT_a$, and in the second 
$P\in\TT_b$. So we have
$$S_1 = \sum_{P\in \TT_a} \frac{\ell(P)\ell(R)}{m_P\sigma} \,  E(P)^2 + 
\sum_{P\in \TT_b} \frac{\ell(P)\ell(R)}{m_P\sigma} \,  E(P)^2=: S_1^a + S_1^b.$$
We consider first the case $P\in \TT_a$. Then recall that
$$E(P) = \frac{\ell(Q_P)}{\ell(R)^2}\,\|(x-Sx)\chi_{P}\|_{L^2(\mu)}.$$
Therefore,
$$S_1^a\leq \sum_{P\in \TT_a} \frac{\ell(P)\,\ell(Q_P)^2}{m_P\sigma\,\ell(R)^3}\,\|(x-Sx)\chi_{P}\|_{L^2(\mu)}^2.$$
Since $\vphi \approx \ell(Q_P)^2/\ell(R)^2$ on $P\subset Q_P$, we have $m_P \sigma\approx \ell(Q_P)^2/\ell(R)^2$.
Thus,
\begin{align*}
S_1^a  &\lesssim \sum_{P\in \TT_a} \frac{\ell(P)}{\ell(R)}\,\|(x-Sx)\chi_{P}\|_{L^2(\mu)}^2
\lesssim \int_R|Sx-x|^2\,d\mu \leq
W_2(\mu,\nu)^2.
\end{align*}

Let us turn our attention to $S_1^b$. Recall that if $P\in\TT_b$, then
$$E(P)= \biggl(\,
\sum_{i:3Q_i\cap P\neq\varnothing}
 \frac{\ell(Q_i)^2}{\ell(R)^4}\,\|Sx-x\|_{L^2(\mu|3Q_i)}^2\biggr)^{1/2},$$
and thus
$$S_1^b  = \sum_{P\in \TT_b} \frac{\ell(P)\ell(R)}{m_P\sigma}\,
\sum_{i:3Q_i\cap P\neq\varnothing}
 \frac{\ell(Q_i)^2}{\ell(R)^4}\,\|Sx-x\|_{L^2(\mu|3Q_i)}^2.
$$
Using that 
$m_P \sigma\gtrsim\ell(P)^2/\ell(R)^2,$
 we get
\begin{align}\label{eqnop65}
S_1^b & \lesssim \sum_{P\in \TT_b} \sum_{i:3Q_i\cap P\neq\varnothing}\frac{\ell(Q_i)^2}
{\ell(P)\,\ell(R)}\,
\|Sx-x\|_{L^2(\mu|3Q_i)}^2.
\end{align}
It is easy to check that for each fixed $Q_i$,
$$\sum_{P\in \TT_b:3Q_i\cap P\neq\varnothing}\frac1{\ell(P)}\lesssim \frac1{\ell(Q_i)}.$$
Therefore, changing the order of summation in \rf{eqnop65} we obtain
\begin{align*}
S_1^b 
&\lesssim \sum_{i\in I}\frac{\ell(Q_i)}
{\ell(R)}\,
\|Sx-x\|_{L^2(\mu|3Q_i)}^2\lesssim \int_R|Sx-x|^2\,d\mu \leq
W_2(\mu,\nu)^2 .
\end{align*}
\end{proof}
\vvv

Next, we deal with the sum $S_2$ from \rf{eqde10}.

\begin{lemma}\label{lems2}
We have
$$S_2 = \sum_{P\in \TT} 
\frac{\ell(R)\,\eta(P)^2}{\ell(P)^{n-1}\,m_P\sigma} \lesssim W_2(\mu,\nu)^2.$$
\end{lemma}
\vv

\begin{proof}
Since $\eta(P)= \vphi\,T\#(\chi_{G^c}\,\nu)(P)\leq \vphi\,T\#\nu(P)=\sigma(P)$,
 we have
$$S_2\leq \sum_{P\in \TT} 
\ell(R)\,\ell(P)\,\eta(P)\lesssim 
\ell(R)^2\,\eta(R).$$
By \rf{eqd32*'}, we know that
$$\eta(R) = \|\vphi\mu -\vphi\,T\#(\chi_{G}\,\nu)\|
\lesssim \frac{W_2(\mu,\nu)^2}{\ell(R)^2},$$
which proves the lemma.
\end{proof}
\vvv

Now we consider the last term $S_3$ from \rf{eqde10}.

\begin{lemma}\label{lems3}
We have
$$S_3 = W_2(\mu,\nu)^2\sum_{P\in \TT}   
\frac{1}{m_P\sigma\,\ell(P)^{n-1}\,\ell(R)^{n+1}}\,
\left(\int_{P\cap S^{-1}(G)} \vphi\circ S\,d\mu\right)
^2 \lesssim W_2(\mu,\nu)^2.$$
\end{lemma}

\vv
\begin{proof}
Notice that if $x\in P\cap S^{-1}(G)$, 
then $Sx\in G$, and from the definition of $G$, it turns out that
$|T(Sx)-Sx|\leq \ell(Q_{Sx})$, where $Q_{Sx}$ stands for the Whitney cube that contains $Sx$. Since
$T(Sx)=x$ for $\mu$-a.e. $x$, then 
$$|x-Sx|\leq \ell(Q_{Sx})\qquad \text{$\mu$-a.e. $x\in  P\cap S^{-1}(G)$.}$$
Then, for these points $x$, $\vphi(Sx)\approx\vphi(x)$, and thus
$$\int_{P\cap S^{-1}(G)} \vphi\circ S\,d\mu\lesssim \int_{P\cap S^{-1}(G)} \vphi\,d\mu
\lesssim \sigma(P).$$
Consequently,
$$S_3 \lesssim W_2(\mu,\nu)^2\sum_{P\in \TT}   
\frac{\sigma(P)^2}{m_P\sigma\,\ell(P)^{n-1}\,\ell(R)^{n+1}} = 
W_2(\mu,\nu)^2\sum_{P\in \TT} \frac{\ell(P)\,\sigma(P)}{\ell(R)^{n+1}}.$$
Using that $\sigma(P)\lesssim\ell(P)^n$, we derive
$$S_3 \lesssim W_2(\mu,\nu)^2\sum_{P\in \TT} \frac{\ell(P)^{n+1}}{\ell(R)^{n+1}}\lesssim W_2(\mu,\nu)^2.$$
\end{proof}

\vvv

In the preceding lemmas we have shown that $S_1+S_2+S_3\lesssim W_2(\mu,\nu)^2$.
So we have shown that the left side of \rf{eqde10}, which equals the first sum of \rf{eqde9}, is bounded
above by $c\,W_2(\mu,\nu)^2$. This
 completes the proof of Theorem \ref{teoloc}.\fiproof


\section{Proof of Lemma \ref{lemkr0}}\label{secappendix}

Recall that, to conclude the proof of Theorem \ref{teobola}, it still remains to prove Lemma \ref{lemkr0}, which we rewrite here for the reader's convenience.

\begin{lemma*}
Let  $R$ and $\vphi\in \GZ(R)$ be as in Theorem \ref{teoloc}. 
Consider the measure $d\wt\sigma=\vphi\,dm$ and let 
$Q\in\DD(R)$. Consider a bounded function $h$ supported on $Q$ such that $\|h\|_\infty\leq1$ and $\int h\,d\wt\sigma=0$.
Then there exists a map $U:Q\to Q$ such that
$U\# (\wt\sigma\rest Q) = (1+h)\,\wt\sigma\rest Q$ which satisfies
\begin{equation}\label{equxu1}
|Ux - x|\lesssim \ell(Q)\,\|h\|_\infty\qquad\text{for all $x\in Q$.}
\end{equation}
\end{lemma*}

\vvv
\begin{proof} Denote $\sigma_0=\wt\sigma\rest Q$ and
$\rho_0=\sigma_0 + h\,\sigma_0$.
We may assume that $\|h\|_\infty$ is small enough. Otherwise \rf{equxu1} holds trivially for the
optimal map for quadratic cost, taking into account that $\sigma_0$ and $\rho_0$ are both
supported on $Q$. 
 
We will prove the lemma using the Knothe-Rosenblatt coupling between $\sigma_0$ and $\rho_0$, as explained in 
\cite[p.8]{Villani-oldnew}.  
 Let us recall briefly in what consists this coupling.
First, one considers the marginals of $\sigma_0$ and $\rho_0$ on the first variable, 
obtaining $\sigma_1=\sigma_1(dx_1)$ and $\rho_1=\rho_1(dx_1)$. Define $y_1=U_1(x_1)$ by the formula of the increasing rearrangement such that $U_1\#\sigma_1=\rho_1$. That is,
given 
$$F_1(t) = \sigma_1(-\infty,t], \qquad G_1(t) = \rho_1(-\infty,t],$$
set $U_1= G_1^{-1}\circ F_1$.

Take now the marginals with respect to the first two variables, and obtain $d\sigma_2(x_1,x_2)$ and 
$d\rho_2(x_1,x_2)$. Disintegrate them so that for each $x_1$ there exist measures $d\sigma_{2;x_1}(x_2)$
and $d\rho_{2;x_1}(x_2)$ such that
\begin{equation}\label{eqdenss8}
d\sigma_2(x_1,x_2) = d\sigma_1(x_1)\,d\sigma_{2;x_1}(x_2),\qquad
d\rho_2(x_1,x_2) = d\rho_1(x_1)\,d\rho_{2;x_1}(x_2).
\end{equation}
Let us remark that in our precise situation we do not need to apply any delicate 
disintegration theorem since the densities of $d\sigma_{2;x_1}(x_2)$ and $d\rho_{2;x_1}(x_2)$
can be explicitly calculated by the identities in \rf{eqdenss8}.
Consider the map $y_2=U_2(x_1,x_2)$ such that for each $x_1$, $U_2(x_1,\cdot)\#\sigma_{2;x_1} 
=\rho_{2;U_1x_1}$,
where $U_2$ is given by the corresponding increasing rearrangement. 

Repeat the construction adding variables one after the other, and obtain also $U_3(x_1,x_2,x_3)$,\ldots, 
$U_n(x_1,\ldots,x_n)$. Finally, set
$$U(x) = \bigl(U_1(x_1),U_2(x_1,x_2),\ldots,U_n(x_1,\ldots,x_n)\bigr).$$
It is easy to check that $U\#\sigma_0=\rho$.

By translating $R$, we may assume that $R= [0,1]^n$ (or $(0,1]^n$, to be more precise), 
and that $Q\cap [0,1/2]^n\neq \varnothing$, so that
either $Q$ is contained in $(0,1/2]^n$ or $Q=R$. Let $a_j,b_j$ be such that
$$Q=(a_1,b_1]\times \cdots \times(a_n,b_n].$$
By interchanging the coordinates if necessary, we will also suppose that
$a_1\leq a_2\leq\cdots \leq a_n.$
We will show below that for every $x\in Q$, and $1\leq j \leq n$,
\begin{equation}\label{eqkr1}
|U_j(x_1,\ldots,x_j) - x_j|\lesssim \|h\|_\infty\,\min\bigl(x_j-a_j,b_j-x_j\bigr).
\end{equation}
Clearly, \rf{equxu1} follows from this estimate.

To prove \rf{eqkr1}, recall that 
$$\vphi(x)\approx \frac{\dist(x,\partial R)^2}{\ell(R)^2 }\qquad \text{for all $x\in R$}.$$
If $Q\cap \partial R=\varnothing$, then $\dist(Q,\partial R)\geq\ell(Q)$,
and from the preceding estimate, it turns out that $\vphi(x)\approx m_Q\vphi$ for every $x\in Q$. 
This fails if $Q\cap \partial
R\neq\varnothing$. Indeed, one can easily check that $m_Q\vphi\approx\ell(Q)^2/\ell(R)^2$, while 
$\vphi$ vanishes on $\partial Q\cap\partial R$.

\vvv
We will derive \rf{eqkr1} by induction on $j$, first assuming that  {\boldmath $Q\neq R$} and thus $Q\subset (0,1/2]^n$. In this case, $\dist(x,\partial R)=\min(x_1,\ldots,x_n)$ for every $x\in Q$, and so
\begin{equation}\label{eqkr0}
\vphi(x)\approx \frac{\min(x_1,\ldots,x_n)^2}{\ell(R)^2 }\qquad \text{for $x\in Q$}.
\end{equation}
\vvv

{\bf The case {\boldmath $j=1$}.} Recall that $U_1= G_1^{-1}\circ F_1$, with 
$F_1(t) = \sigma_1(-\infty,t]$ and $ G_1(t) = \rho_1(-\infty,t]$. To estimate $\sigma_1$, notice
that, using
the assumption $a_1\leq a_2\leq\cdots \leq a_n$, it turns out that
$$x_1\geq \min(x_1,\ldots,x_n)\geq \frac{x_1}2$$
for $x_2,\ldots,x_n$ in a subset of $[a_2,b_2]\times \cdots\times[a_n,b_n]$ of $(n-1)$-dimensional Lebesgue measure comparable to $\ell(Q)^{n-1}$ (possibly depending on 
$n$). From this fact, one infers that
$$d\sigma_1(x_1) \approx \left(\int_{a_2}^{b_2}\cdots \int_{a_n}^{b_n} \frac{\min(x_1,\ldots,x_n)^2}{\ell(R)^2
}\,dx_2\ldots dx_n\right)dx_1\approx \frac{\ell(Q)^{n-1}\,x_1^2}{\ell(R)^2}\,dx_1.$$
Now we distinguish two cases, according to whether $a_1=0$ or not.

\vvv
Suppose first that {\boldmath{$a_1=0$}}. Then,
\begin{equation}\label{eqf19}
F_1(t) = \sigma_1(-\infty,t] \approx \int_0^t\frac{\ell(Q)^{n-1}\,x_1^2}{\ell(R)^2}\,dx_1 \approx 
\frac{\ell(Q)^{n-1}}{\ell(R)^2}\,t^3= A_0\,t^3, \quad\text{for $t\in[a_1,b_1]$,}
\end{equation}
where, to simplify notation, we set $A_0=\frac{\ell(Q)^{n-1}}{\ell(R)^2}$.
Assuming $\|h\|_\infty\leq1/2$, the density functions of $\sigma_0$ and $\rho_0$ are comparable, and thus we also 
have 
$$
d\rho_1(x_1)\approx \chi_{[a_1,b_1]}(x_1)\,A_0 \,x_1^2\,dx_1,\qquad
G_1(t) = \rho_1(-\infty,t]  \approx 
A_0\,t^3\quad\text{for $t\in[a_1,b_1]$},
$$
and therefore
$$G_1^{-1}(s) \approx A_0^{-1/3}\,s^{1/3},$$
and the derivative of $G_1^{-1}$ satisfies
$$(G_1^{-1})'(s) = \frac1{G_1'(G_1^{-1}(s))}\approx \frac1{A_0\,G_1^{-1}(s)^2}
\approx  A_0^{-1/3}\,s^{-2/3}.$$
Then we deduce
\begin{align}\label{eqf29}
|U_1(x_1) - x_1| &= |G_1^{-1}(F_1(x_1)) - G_1^{-1}(G_1(x_1))| =\left|\int_{G_1(x_1)}^{F_1(x_1)}(G_1^{-1})'(s)
\,ds\right|\nonumber\\
& \approx A_0^{-1/3}\,\left|\int_{G_1(x_1)}^{F_1(x_1)} s^{-2/3}
\,ds\right|\approx A_0^{-1/3}\,|F_1(x_1)^{1/3}- G_1(x_1)^{1/3}|.
\end{align}
Since $|c^{1/3}-1|\approx|c-1|$ for $c$ close to $1$, assuming $\|h\|_\infty$ small enough we get
\begin{align}\label{eqff30}
|U_1(x_1) - x_1|& \approx 
A_0^{-1/3}\, F_1(x_1)^{1/3}\,
\biggl|\frac{G_1(x_1)}{F_1(x_1)}-1\biggr| =
A_0^{-1/3}\, F_1(x_1)^{-2/3}\,
|F_1(x_1) - G_1(x_1)|.
\end{align}
Using that $|F_1(x_1) - G_1(x_1)|\leq F_1(x_1)\,\|h\|_\infty$ and 
$F_1(x_1)\approx A_0\,x_1^3$,
 we obtain
$$|U_1(x_1) - x_1|\lesssim
A_0^{-1/3}\, F_1(x_1)^{-2/3}
F_1(x_1)\,\|h\|_\infty\approx x_1\,\|h\|_\infty,
$$
which proves \rf{eqkr1} in this case.

\vvv
Suppose now that {\boldmath $a_1\neq 0$}. 
Then we have $x_1\approx a_1$ for all $x\in Q$, and we obtain
\begin{equation}\label{eqf208}
F_1(t) = \sigma_1(-\infty,t] \approx \int_{a_1}^t\frac{\ell(Q)^{n-1}\,a_1^2}{\ell(R)^2}\,dx_1 
= (t-a_1)\frac{\ell(Q)^{n-1}\,a_1^2}{\ell(R)^2} \quad\text{for $t\in[a_1,b_1]$.}
\end{equation}
To simplify notation, we set $A_1=\frac{\ell(Q)^{n-1}\,a_1^2}{\ell(R)^2}$, so that 
$$F_1(t) \approx A_1\,(t-a_1)\quad\text{for $t\in[a_1,b_1]$.}$$
By the comparability of the measures $\sigma_0$ and $\rho_0$, we also get
$$d\rho_1(x_1)\approx \chi_{[a_1,b_1]}(x_1)\,A_1\,dx_1,\qquad
G_1(t) = \rho_1(-\infty,t]  \approx 
A_1\,(t-a_1)\quad\text{for $t\in[a_1,b_1]$},$$
and thus we derive
$$G_1^{-1}(s)-a_1\approx \frac s{A_1}$$
and
$$(G_1^{-1})'(s) = \frac1{G_1'(G_1^{-1}(s))}\approx \frac1{A_1}.$$
Arguing as in \rf{eqf29}, we deduce
$$|U_1(x_1) - x_1|  =\left|\int_{G_1(x_1)}^{F_1(x_1)}(G_1^{-1})'(s)
\,ds\right| \approx \frac1{A_1}\,|F_1(x_1)- G_1(x_1)|.$$
Using that $|F_1(x_1) - G_1(x_1)|\leq F_1(x_1)\,\|h\|_\infty$ and 
$F_1(x_1)\approx A_1\,(x_1-a_1)$, we obtain
$$|U_1(x_1) - x_1|\lesssim \frac1{A_1} F_1(x_1)\,\|h\|_\infty\approx |x_1-a_1|\,\|h\|_\infty.$$
Thus \rf{eqkr1} also holds in this case.

\vvv
{\bf The case {\boldmath $j>1$}.} 
Suppose that $U_1(x_1),\ldots,U_{j-1}(x_1,\ldots,x_{j-1})$ satisfy \rf{eqkr1}.
Recall that 
$$U_j(x_1,\ldots,x_{j-1},\,\cdot\,)\#\sigma_{j;x_1,\ldots,x_{j-1}} = \rho_{j;y_1,\ldots,y_{j-1}},$$
where $\sigma_{j;x_1,\ldots,x_{j-1}}$ is defined by
$$d\sigma_j(x_1,\ldots,x_j) = d\sigma_{j-1}(x_1,\ldots,x_{j-1})\,d\sigma_{j;x_1,\ldots,x_{j-1}}(x_j),$$
and analogously for $\rho_0$. Recall also that $y_1=U_1(x_1)$, $y_2=U_2(x_1,x_2), \ldots$ and, further,
 $\sigma_{j-1}$ and $\sigma_j$ stand for the marginals of $\sigma_0$
with respect to the first $j-1$ and $j$ variables, respectively. Thus,
$$d\sigma_{j-1}(x_1,\ldots,x_{j-1}) 
\approx \left(\int_{a_j}^{b_j}\cdots \int_{a_n}^{b_n} \frac{\min(x_1,\ldots,x_n)^2}{\ell(R)^2 }\,dx_j\ldots 
dx_n\right)dx_1\ldots dx_{j-1},$$
and similarly for $\sigma_j$.

To estimate the integral inside the parentheses we argue as in the case $j=1$: since $a_1\leq\ldots\leq
 a_n$, there exists a subset of $[a_j,b_j]\times\cdots\times[a_n,b_n]$ of measure comparable to 
 $\ell(Q)^{n-j+1}$ such that
 $$\min(x_1,\ldots,x_{j-1})\geq \min(x_1,\ldots,x_n)\geq \frac12\,\min(x_1,\ldots,x_{j-1}).$$
Consequently,
\begin{equation}\label{eqsj-1}
d\sigma_{j-1}(x_1,\ldots,x_{j-1}) \approx \frac{\ell(Q)^{n-j+1}}{\ell(R)^2}\min(x_1,\ldots,x_{j-1})^2\, dx_1\ldots dx_{j-1}.
\end{equation}
Analogously, we have
\begin{equation}\label{eqsj-0}
d\sigma_j(x_1,\ldots,x_j) \approx \frac{\ell(Q)^{n-j}}{\ell(R)^2}\min(x_1,\ldots,x_j)^2\, dx_1\ldots dx_j,
\end{equation}
and thus
$$d\sigma_{j;x_1,\ldots,x_{j-1}}(x_j) \approx \frac{\min(x_1,\ldots,x_j)^2}{\ell(Q)\,\min(x_1,\ldots,x_{j-1})^2}\,dx_j.$$

Since the densities of $\rho_0$ and $\sigma_0$ are comparable,  a similar estimate holds for 
$\rho_{j;y_1,\ldots,y_{j-1}}$:
$$d\rho_{j;y_1,\ldots,y_{j-1}}(x_j) \approx 
\frac{\min(y_1,\ldots,y_{j-1},x_j)^2}{\ell(Q)\,\min(y_1,\ldots,y_{j-1})^2}\,dx_j.$$
Moreover, since $|y_k-x_k|\lesssim\|h\|_\infty x_k$ for $1\leq k\leq j-1$, by \rf{eqkr1} and the induction
hypothesis, we deduce that $y_k\approx x_k$ for these $k$'s, assuming $\|h\|_\infty$ small enough, and then
we get
$$d\sigma_{j;x_1,\ldots,x_{j-1}}(x_j) \approx d\rho_{j;y_1,\ldots,y_{j-1}}(x_j).$$

For fixed points $x_1,\ldots,x_{j-1}$, $y_1,\ldots,y_{j-1}$, denote
$$F_j(t) = \sigma_{j;x_1,\ldots,x_{j-1}}(-\infty,t], \qquad
G_j(t) = \rho_{j;y_1,\ldots,y_{j-1}}(-\infty,t],$$  
so that $U_j(x_1,\ldots,x_{j-1},x_j) = G_j^{-1}\circ F_j(x_j)$. Then, as in the case $j=1$, we have
\begin{equation}\label{eqff31}
|U_j(x_1,\ldots,x_{j-1},x_j) - x_j| = |G_j^{-1}\circ F_j(x_j) - G_j^{-1}\circ G_j(x_j)| =
\left|\int_{G_j(x_j)}^{F_j(x_j)}(G_j^{-1})'(s)\,ds\right|.
\end{equation}
We need now to estimate $G_j(x_j)$, $F_j(x_j)$, and $(G_j^{-1})'(s)$. To this end, notice that
\begin{equation}\label{eqkr00}
d\sigma_{j;x_1,\ldots,x_{j-1}}(x_j) \approx \left\{\begin{array}{ll}
\dfrac{x_j^2}{\ell(Q)\,\min(x_1,\ldots,x_{j-1})^2}\,dx_j& \text{if $x_j\leq\min(x_1,\ldots,x_{j-1}),$}\\
&\\
\dfrac1{\ell(Q)}\,dx_j &  \text{if $x_j>\min(x_1,\ldots,x_{j-1}).$}
\end{array}
\right.
\end{equation}

Below we will show that
\begin{equation}\label{eqkr51}
|d\sigma_{j;x_1,\ldots,x_{j-1}}(t) - d\rho_{j;y_1,\ldots,y_{j-1}}(t)| \lesssim
\|h\|_\infty \,d\sigma_{j;x_1,\ldots,x_{j-1}}(t)
\end{equation}
(this notation means that the difference of the densities of $\sigma_{j;x_1,\ldots,x_{j-1}}$ and 
$\rho_{j;y_1,\ldots,y_{j-1}}$ is bounded above by some constant times
$\|h\|_\infty$ times the density $\sigma_{j;x_1,\ldots,x_{j-1}}$).
We defer the proof of this estimate to Lemma \ref{lemdef1}. Integrating on $t$, we derive
\begin{equation}\label{eqff33}
|F_j(t)- G_j(t)|\lesssim F_j(t)\, \|h\|_\infty.
\end{equation}

\vvv
Suppose first that {\boldmath $a_j=0$}. For $t\leq\min(x_1,\ldots,x_{j-1})$, we have
$$F_j(t) \approx \int_0^t \dfrac{x_j^2}{\ell(Q)\,\min(x_1,\ldots,x_{j-1})^2}\,dx_j \approx 
\dfrac{t^3}{\ell(Q)\,\min(x_1,\ldots,x_{j-1})^2} =: B_0\,t^3,$$
while for $t>\min(x_1,\ldots,x_{j-1})$,
$$F_j(t) \approx \int_0^{\min(x_1,\ldots,x_{j-1})} \dfrac{x_j^2}{\ell(Q)\,\min(x_1,\ldots,x_{j-1})^2}\,dx_j
+ \int_{\min(x_1,\ldots,x_{j-1})}^t\dfrac1{\ell(Q)}\,dx_j \approx \frac t{\ell(Q)}.$$
For $G_j(t)$ we have analogous estimates, because $G_j(t)\approx F_j(t)$.

Observe that, for $t\leq\min(x_1,\ldots,x_{j-1})$, this is the same estimate as the one in \rf{eqf19}, replacing
$A_0$ by $B_0$. Thus, for $0\leq s\leq\min(x_1,\ldots,x_{j-1})/\ell(Q)$ we deduce
$$G_j^{-1}(s) \approx B_0^{-1/3}\,s^{1/3},\qquad(G_j^{-1})'(s) = \frac1{G_j'(G_j^{-1}(s))}\approx   
B_0^{-1/3}\,s^{-2/3}.$$
On the other hand, for $s>\min(x_1,\ldots,x_{j-1})/\ell(Q)$, 
$$G_j^{-1}(s) \approx \ell(Q)\,s, \qquad(G_j^{-1})'(s) = \frac1{G_j'(G_j^{-1}(s))}\approx {\ell(Q)}.$$

For $x_j\leq \min(x_1,\ldots,x_{j-1})$, by \rf{eqff31} and arguing as in \rf{eqff30}, we will obtain
$$|U_j(x_1,\ldots,x_{j-1},x_j) - x_j|\lesssim B_0^{-1/3}\, F_j(x_j)^{-2/3}\,
|F_j(x_j) - G_j(x_j)|.$$
Hence, from \rf{eqff33} and the fact that 
$F_j(x_j)\approx B_0\,x_j^3$, we get
$$|U_j(x_1,\ldots,x_{j-1},x_j) - x_j|\lesssim
B_0^{-1/3}\, F_j(x_j)^{-2/3}
F_j(x_j)\,\|h\|_\infty\approx x_j\,\|h\|_\infty.
$$

For $x_j> \min(x_1,\ldots,x_{j-1})$, $F_j(x_j)\approx G_j(x_j)\gtrsim\min(x_1,\ldots,x_{j-1})/\ell(Q)$,
and then, for values of $s$ between $F_j(x_j)$ and $G_j(x_j)$, we have $(G_j^{-1})'(s)\approx\ell(Q)$,
and so, by \rf{eqff31},
$$|U_j(x_1,\ldots,x_{j-1},x_j) - x_j| =\left|\int_{G_j(x_j)}^{F_j(x_j)}(G_j^{-1})'(s)\,ds\right|
\approx {\ell(Q)}\,|F_j(x_j) - G_j(x_j)|$$
By \rf{eqff33} and the fact that $F_j(x_j)\approx x_j/\ell(Q)$,
$$|U_j(x_1,\ldots,x_{j-1},x_j) - x_j|\lesssim {\ell(Q)}\,F_j(x_j)\,\|h\|_\infty\approx
x_j\,\|h\|_\infty.$$
Therefore, \rf{eqkr1} holds if $a_j=0$.

\vvv
Consider now the case {\boldmath $a_j\neq0$}, and so $Q\subset [0,1/2]^n$. 
In this case, for all $x\in Q$ we have
$x_j\approx a_j$. Then, since $a_j\geq a_{j-1}\geq\cdots\geq a_1$, it easily follows that $x_j\gtrsim x_k$
for $1\leq k\leq j-1$. So from \rf{eqkr00}, we infer that
$$d\sigma_{j;x_1,\ldots,x_{j-1}}(x_j)\approx\dfrac1{\ell(Q)}\,dx_j.$$
In this case, we get
\label{eqf20}
$$F_j(t) = \sigma_{j;x_1,\ldots,x_{j-1}}(-\infty,t] \approx  \frac1{\ell(Q)}\,(t-a_j) \quad\text{for $t\in[a_j,b_j]$.}$$
This estimate is analogous to the one in \rf{eqf208}, replacing $A_1$ by $B_1=1/\ell(Q)$ and $a_1$ by $a_j$.
By similar arguments, we get again
$$|U_j(x_1,\ldots,x_{j-1},x_j) - x_j|\lesssim x_j\,\|h\|_\infty.$$
This finishes the proof of the lemma for $Q\neq R$.

\vvv
For {\boldmath $Q=R$},  the proof of \rf{eqkr1} for $x\in[0,1/2]^n$ is analogous to the one above for
$Q\neq R$ in the case $a_j=0$ (for every $j$). The details are left for the reader. The estimate \rf{eqkr1}
 for $x\in R\setminus [0,1/2]^n$
also holds, because of the symmetry of the assumptions on $R$, $\sigma_0$, and $\rho_0$.
Indeed, notice that we 
also have $U_1=\wt G_1^{-1}\circ \wt F_1$, with $\wt F_1(t) = \sigma_1[t,+\infty)$, $\wt G_1(t) = \rho_1[t,+\infty)$, and analogously for the other indices $j>1$.
\end{proof}
\vvv

To complete the proof of the preceding lemma, it remains to prove the estimate \rf{eqkr51}. This is what we 
do below.

\begin{lemma}\label {lemdef1}
Under the notation and assumptions of Lemma \ref{lemkr0}, if \rf{eqkr1} holds for $1\leq k\leq j-1$, then
$$|d\sigma_{j;x_1,\ldots,x_{j-1}}(x_j) - d\rho_{j;y_1,\ldots,y_{j-1}}(x_j)| \lesssim
\|h\|_\infty \,d\sigma_{j;x_1,\ldots,x_{j-1}}(x_j).$$
\end{lemma}
\vvv

\begin{proof}
We assume that $Q\neq R$ and $Q\subset[0,1/2]^n$.
Let $s_j$, $r_j$ be such that
$$d\sigma_j(x_1,\ldots,x_j) =s_j(x_1,\ldots,x_j)\,dx_1\ldots dx_j,$$
and
$$d\rho_j(x_1,\ldots,x_j) =r_j(x_1,\ldots,x_j)\,dx_1\ldots dx_j.$$
So we have
\begin{align*}
d\sigma_{j;x_1,\ldots,x_{j-1}}(x_j)- d\rho_{j;y_1,\ldots,y_{j-1}}(x_j)
& =  \biggl(\frac{s_j(x_1,\ldots,x_j)}{s_{j-1}(x_1,\ldots,x_{j-1})} -
\frac{r_j(y_1,\ldots,y_{j-1},x_j)}{r_{j-1}(y_1,\ldots,y_{j-1})}\biggr)\,dx_j
\end{align*}
We split the term inside the parentheses on the right side as follows:
\begin{multline}\label{eqfifiab}
\biggl(
\frac{s_j(x_1,\ldots,x_j)}{s_{j-1}(x_1,\ldots,x_{j-1})} -
\frac{s_j(y_1,\ldots,y_{j-1},x_j)}{s_{j-1}(y_1,\ldots,y_{j-1})}
\biggr) \,+ \,
\biggl(
\frac{s_j(y_1,\ldots,y_{j-1},x_j)}{s_{j-1}(y_1,\ldots,y_{j-1})} -
\frac{r_j(y_1,\ldots,y_{j-1},x_j)}{r_{j-1}(y_1,\ldots,y_{j-1})}
\biggr) \\ = A+ B.
\end{multline}
Denote by $s$ and $r$  the densities of $\sigma_0$ and $\rho_0$ on $Q$. Recall that $s(x)=\vphi(x)$, where $\vphi$ satisfies
the assumptions of Theorem \ref{teoloc}. 
In particular, recall that
\begin{equation}\label{eqnabno}
|\nabla s(x)|  \lesssim
\frac{\dist(x,\partial R)}{\ell(R)^2}\qquad\text{and}\qquad
|\nabla_T s(x)|\lesssim\frac{\dist(x,\partial R)^2}{\ell(R)^3}.
\end{equation}

We write $A$ as follows:
\begin{align}\label{eqfifi-1}
A & = \frac{s_j(x_1,\ldots,x_j)- s_j(y_1,\ldots,y_{j-1},x_j)}{s_{j-1}(x_1,\ldots,x_{j-1})} \\
&\quad + 
s_j(y_1,\ldots,y_{j-1},x_j)\,\biggl( 
\frac1{s_{j-1}(x_1,\ldots,x_{j-1})} - \frac1{s_{j-1}(y_1,\ldots,y_{j-1})}
\biggr) = A_1+ A_2.\nonumber
\end{align}
First we consider the term $A_2$. To this end, notice that
\begin{multline}\label{eqfifi0}
s_{j-1}(x_1,\ldots,x_{j-1})- s_{j-1}(y_1,\ldots,y_{j-1}) \\= 
\int_{a_j}^{b_j}\!\!\cdots\! \int_{a_n}^{b_n} \bigl(s(x_1,\ldots,x_{j-1},z_j,\ldots,z_n) - 
s(y_1,\ldots,y_{j-1},z_j,\ldots,z_n)\bigr)\,dz_j\ldots dz_n.
\end{multline}
To estimate the difference inside the integral, we distinguish two cases. Assume first that
$$\min(z_j,\ldots,z_n)<2\,\min(x_1,\ldots,x_{j-1}).$$
Recall that, for $u\in Q$, we have $\dist(u,\partial R) = \min(u_1,\ldots,u_n)$ and denote
$$Q_{z_j,\ldots,z_n} = \{u\in Q:\,u_k=z_k\text{ for $k\geq j$}\}.$$
From \rf{eqnabno}, taking into account that $|x_k-y_k|\lesssim x_k\,\|h\|_\infty\leq \ell(Q)\,\|h\|_\infty$ 
for $1\leq k\leq j-1$,
we deduce that
\begin{align}\label{eqfifi1}
\bigl|s(x_1,\ldots,x_{j-1}, z_j,&\ldots,z_n ) - 
s(y_1,\ldots,y_{j-1},z_j,\ldots,z_n)\bigr|\lesssim \|\nabla s\|_{\infty,Q_{z_j,\ldots,z_n}}
\, \sum_{k=1}^{j-1}|x_k-y_k|\\
& \lesssim \frac{\min(z_j,\ldots,z_n)}{\ell(R)^2} \,\ell(Q) \,\|h\|_\infty\lesssim
\frac{\min(x_1,\ldots,x_{j-1})}{\ell(R)^2} \,\ell(Q) \,\|h\|_\infty.\nonumber
\end{align}
 
Suppose now that
$$\min(z_j,\ldots,z_n)\geq 2\,\min(x_1,\ldots,x_{j-1}).$$
Since $|y_k-x_k|\lesssim x_k\|h\|_\infty$ for $0\leq k\leq j-1$, for $\|h\|_\infty$ small enough we have
$y_k\leq 2x_k$, and thus
$$\min(z_j,\ldots,z_n)>\min(y_1,\ldots,y_{j-1}).$$
Assume that $\min(x_1,\ldots,x_{j-1})\leq \min(y_1,\ldots,y_{j-1})$, and let $i$ be such that
$x_i=\min(x_1,\ldots,x_{j-1})$. Then we have
$$\dist\bigl((x_1,..,x_{j-1},z_j,..,z_n),\,\partial R\bigr) = 
\dist\bigl((y_1,..,x_i,..,y_{j-1},z_j,..,z_n),\,\partial R\bigr) = x_i.$$
where $(y_1,..,x_i,..,y_{j-1},z_j,..,z_n)$ is obtained by replacing $y_i$ by $x_i$ in 
$(y_1,..,y_{j-1},z_j,..,z_n)$.
By \rf{eqnabno}, denoting
$$Q_{x_i,z_j,\ldots,z_n} = \{u\in Q:\,u_i=x_i,\,u_k=z_k\text{ for $k\geq j$}\}$$
and
$$L= L_{y_1,..,y_{i-1},y_i,..,y_{j-1},z_j,..,z_n}= \bigl\{(y_1,..,u_i,..,y_{j-1},z_j,..,z_n):\, x_i\leq u_i\leq y_i\bigr\},$$
we get
\begin{align*}
\bigl|s(x_1,..,x_{j-1}, z_j,..,&z_n ) - 
s(y_1,..,y_{j-1},z_j,..,z_n)\bigr|\\
& \leq 
\bigl|s(x_1,..,x_{j-1}, z_j,..,z_n ) - s(y_1,..,x_i,..,y_{j-1},z_j,..,z_n)\bigr|\\ 
&\quad +
\bigl|s(y_1,..,x_i,..,y_{j-1},z_j,..,z_n) - s(y_1,..,y_{j-1},z_j,..,z_n)\bigr|\\
&\lesssim \|\nabla_T s\|_{\infty,Q_{x_i,z_j,\ldots,z_n}}\, \sum_{
\substack{1\leq k\leq j-1\\ k\neq i}} 
|x_k-y_k|
+ 
\|\nabla s\|_{\infty,L}\,|x_i-y_i| \\
& \lesssim \frac{x_i^2}{\ell(R)^3}\sum_{
\substack{1\leq k\leq j-1\\ k\neq i}} 
|x_k-y_k|
+ \frac{x_i}{\ell(R)^2}\,|x_i-y_i|.
\end{align*} 
Since $|x_k-y_k|\lesssim x_k\,\|h\|_\infty\leq \ell(Q)\,\|h\|_\infty$, we get
\begin{align}\label{eqfifi2}
\bigl|s(x_1,..,x_{j-1}, z_j,..,z_n ) - 
s(y_1,..,y_{j-1},z_j,..,z_n)\bigr| & \lesssim 
\frac{x_i^2}{\ell(R)^3}\,\ell(Q)\,\|h\|_\infty 
+ \frac{x_i}{\ell(R)^2}\,x_i\,\,\|h\|_\infty\\
& \!\!\!\!\!\!\lesssim \frac{x_i^2}{\ell(R)^2}\,\|h\|_\infty =
\frac{\min(x_1,\ldots,x_{j-1})^2}{\ell(R)^2}\,\|h\|_\infty.\nonumber
\end{align}
If  $\min(x_1,\ldots,x_{j-1})>\min(y_1,\ldots,y_{j-1})=y_i$, we get the same estimate just interchanging the 
roles of $x_k$ and $y_k$, $1\leq k\leq j-1$.

If we plug the estimates \rf{eqfifi1} and \rf{eqfifi2} into the identity \rf{eqfifi0}, we obtain
\begin{align*}
\bigl|s_{j-1}(x_1,\ldots,x_{j-1})- s_{j-1}(y_1,\ldots,y_{j-1})\bigr| & \leq 
\ell(Q) \,\|h\|_\infty \int_{E} \frac{\min(x_1,\ldots,x_{j-1})}{\ell(R)^2} 
\,dz_j\ldots dz_n \\
& \quad + \|h\|_\infty
\int_{F} \frac{\min(x_1,\ldots,x_{j-1})^2}{\ell(R)^2}
\,dz_j\ldots dz_n,
\end{align*}
where $E$ denotes the subset of those points $(z_j,..,z_n)\in [a_j,b_j]\times\cdots\times[a_n,b_n]$ 
such that
$\min(z_j,\ldots,z_n)<2\,\min(x_1,\ldots,x_{j-1})$, and $F=[a_j,b_j]\times\cdots\times[a_n,b_n]\setminus E$. 
It is easy to check that the $(n-j+1)$-dimensional
Lebesgue measure of $E$ is smaller than $c\,\min(x_1,\ldots,x_{j-1})\,\ell(Q)^{n-j}$. Then we deduce
\begin{equation}\label{eqfifi4}
\bigl|s_{j-1}(x_1,\ldots,x_{j-1})- s_{j-1}(y_1,\ldots,y_{j-1})\bigr|\lesssim 
 \frac{\min(x_1,\ldots,x_{j-1})^2}{\ell(R)^2}\,\ell(Q)^{n-j+1}\,\|h\|_\infty.
 \end{equation}
Recalling the definition of $A_2$ in \rf{eqfifi-1}, we derive
$$|A_2| \lesssim 
s_j(y_1,\ldots,y_{j-1},x_j)\,
\frac{\ell(Q)^{n-j+1}\,\min(x_1,\ldots,x_{j-1})^2}
{\ell(R)^2\,s_{j-1}(x_1,\ldots,x_{j-1})\,s_{j-1}(y_1,\ldots,y_{j-1})}\,\|h\|_\infty.$$
By \rf{eqsj-1} we know that
$$s_{j-1}(x_1,\ldots,x_{j-1}) \approx \frac{\ell(Q)^{n-j+1}}{\ell(R)^2}\,\min(x_1,\ldots,x_{j-1})^2.$$
Using also that
$$s_j(y_1,\ldots,y_{j-1},x_j)\approx s_j(x_1,\ldots,x_{j-1},x_j),\qquad
s_{j-1}(y_1,\ldots,y_{j-1})\approx s_{j-1}(x_1,\ldots,x_{j-1}),$$
by \rf{eqsj-0} and \rf{eqsj-1}, we deduce that
\begin{equation}\label{eqfifia2}
|A_2|\lesssim \frac{s_j(x_1,\ldots,x_{j-1},x_j)}{s_{j-1}(x_1,\ldots,x_{j-1})}\,\|h\|_\infty.
\end{equation}

To deal with the term $A_1$ in \rf{eqfifi-1}, first we have to estimate the difference
\begin{equation}\label{eqfifi7}
s_j(x_1,\ldots,x_j)- s_j(y_1,\ldots,y_{j-1},x_j).
\end{equation}
The calculations are analogous to the ones above for 
$$s_{j-1}(x_1,\ldots,x_{j-1})- s_{j-1}(y_1,\ldots,y_{j-1}).$$ 
Indeed, recall that for the latter difference we used the estimate $|x_k-y_k|\lesssim x_k\|h\|_\infty,$
for $1\leq k\leq j-1$. The same inequalities are the ones required to estimate \rf{eqfifi7},
taking into account the $j$-th coordinate is the same in the two terms involving $s_j(\ldots)$.
Then, as in \rf{eqfifi4}, we get
$$\bigl|s_j(x_1,\ldots,x_j)- s_j(y_1,\ldots,y_{j-1},x_j)\bigr|\lesssim 
 \frac{\min(x_1,\ldots,x_j)^2}{\ell(R)^2}\,\ell(Q)^{n-j}\,\|h\|_\infty.
$$
By the definition of $A_1$ in \rf{eqfifi-1}, we deduce
$$|A_1|\lesssim \frac{\ell(Q)^{n-j}\,\min(x_1,\ldots,x_j)^2}{\ell(R)^2\,s_{j-1}(x_1,\ldots,x_{j-1})}  
\,\|h\|_\infty.$$
Recalling that, by \rf{eqsj-1},
$$s_j(x_1,\ldots,x_j) \approx \frac{\ell(Q)^{n-j}}{\ell(R)^2}\,\min(x_1,\ldots,x_j)^2,$$
we obtain
\begin{equation}\label{eqfifia1}
|A_1|\lesssim \frac{s_j(x_1,\ldots,x_{j-1},x_j)}{s_{j-1}(x_1,\ldots,x_{j-1})}\,\|h\|_\infty.
\end{equation} 
 
Now it remains to estimate the term $B$ in \rf{eqfifiab}:
\begin{align*}
B & = 
\frac{s_j(y_1,\ldots,y_{j-1},x_j)}{s_{j-1}(y_1,\ldots,y_{j-1})} -
\frac{r_j(y_1,\ldots,y_{j-1},x_j)}{r_{j-1}(y_1,\ldots,y_{j-1})}\\
 & = \frac{s_j(y_1,\ldots,y_{j-1},x_j)-r_j(y_1,\ldots,y_{j-1},x_j)}{s_{j-1}(y_1,\ldots,y_{j-1})} \\
&\quad + 
r_j(y_1,\ldots,y_{j-1},x_j)\,\biggl( 
\frac1{s_{j-1}(y_1,\ldots,y_{j-1})} - \frac1{r_{j-1}(y_1,\ldots,y_{j-1})}
\biggr) = B_1+ B_2.\nonumber
\end{align*}
Taking into account that $|s(x)-r(x)|\leq s(x)\,\|h\|_\infty$ for all $x\in Q$, one easily gets
$$|s_j(y_1,\ldots,y_{j-1},x_j)-r_j(y_1,\ldots,y_{j-1},x_j)|\lesssim s_j(y_1,\ldots,y_{j-1},x_j)\,\|h\|_\infty$$
and 
$$|s_{j-1}(y_1,\ldots,y_{j-1})-r_{j-1}(y_1,\ldots,y_{j-1})|\lesssim s_{j-1}(y_1,\ldots,y_{j-1})\,\|h\|_\infty.$$
The first estimate readily implies that 
$$|B_1|\lesssim \frac{s_j(y_1,\ldots,y_{j-1},x_j)}{s_{j-1}(y_1,\ldots,y_{j-1})}\,\|h\|_\infty
\approx \frac{s_j(x_1,\ldots,x_j)}{s_{j-1}(x_1,\ldots,x_{j-1})}\,\|h\|_\infty.$$
The calculations for $B_2$ are also straightforward:
$$|B_2|\lesssim 
r_j(y_1,...,y_{j-1},x_j)\,\frac{s_{j-1}(y_1,\ldots,y_{j-1})\,\|h\|_\infty}
{s_{j-1}(y_1,\ldots,y_{j-1})\, r_{j-1}(y_1,\ldots,y_{j-1})}
\lesssim\frac{s_j(x_1,\ldots,x_j)}{s_{j-1}(x_1,\ldots,x_{j-1})}\,\|h\|_\infty,
$$
where we used that
$r_j(y_1,\ldots,y_{j-1},x_j)\approx s_j(x_1,\ldots,x_j)$ and the analogous estimate for $r_{j-1}(y_1,\ldots,y_{j-1})$.

The lemma follows by gathering the inequalities obtained for $A_1$, $A_2$, $B_1$, and~$B_2$.
\end{proof}

Notice that the preceding lemma is the only place of the paper where we used  the smallness of the tangential derivatives of $\vphi$ near the boundary of $R$.


\section{Relationship between the coefficients $\alpha$, $\alpha_1$, $\alpha_2$, and $\beta_2$}
\label{sec5}

Let us recall the definition of the coefficients $\alpha$ from \cite{Tolsa-plms}. 
Given a closed ball $B\subset \R^d$ and two finite Borel measures $\sigma$, $\nu$
on $\R^d$ , we set
$$\dist_B(\sigma,\nu):= \sup\Bigl\{ \Bigl|{\textstyle \int f\,d\sigma  -
\int f\,d\nu}\Bigr|:\,{\rm Lip}(f) \leq1,\,\supp(f)\subset
B\Bigr\},$$
where ${\rm Lip}(f)$ stands for the Lipschitz constant of $f$.
Notice the similarities between the distances $W_1$ and $d_B$.
Given an AD regular measure $\mu$ on $\R^d$ and a ball $B$ that intersects $\supp(\mu)$, we set
$$
\alpha(B)= \frac1{r(B)^{n+1}}\,\inf_{a\geq0,L} \,\dist_{3B}(\mu,\,a\HH^n_{|L}),$$
where the infimum is taken over all the constants $a\geq0$ and all the $n$-planes $L$.

\begin{lemma}\label{lemalfas}
Let $\mu$ be $n$-dimensional AD regular and let $B$ be a closed ball that intersects $\supp(\mu)$.
Then,
$$\alpha(B)\leq \alpha_1(2B)$$
and 
$$\alpha_1(B) \leq c\,\alpha(B),$$
with $c=c_{10}(\|\nabla\vphi\|_\infty  +1 ),$ where $c_{10}$ is an absolute constant and $\vphi$ is the auxiliary
function satisfying \rf{eqfi00} used to define $\vphi_B$ and $\alpha_p$ in \rf{eqcbl}.
\end{lemma}

\begin{proof}
To prove the first inequality, take an $n$-plane $L$ that intersects $2B$ and let $c_{2B,L}$
be as in \rf{eqcbl}, with $2B$ instead of $B$. Consider a Lipschitz function $f$ supported in $3B$ with 
$\|\nabla f\|_\infty\leq 1$. Since $\vphi_{2B}$ equals $1$ in the $4B\supset\supp f$, we have
$$\int f\,d\mu  - \int f\,c_{2B,L}\,d\HH^n_L
= 
\int f\,\vphi_{2B}\,d\mu  - \int f\,\vphi_{2B}\, c_{2B,L}\,d\HH^n_L.$$
Taking supremums over such functions $f$, we get
$$\dist_{3B}(\mu,\,c_{2B,L}\HH^n_L) \leq W_1(\vphi_{2B}\,\mu,\,c_{2B,L}\vphi_{2B}\,\HH^n_L).$$
Taking the infimum over all $n$ planes $L$ intersecting $2B$ and dividing by
${r(2B)^{n+1}}$, we deduce
$\alpha(B)\leq \alpha_1(2B).$

Let us turn our attention to the second inequality in the lemma. Let $a$ and $L$ be the constant and
the $n$-plane that minimize $\alpha(B)$. To estimate $W_1\bigl(\vphi_B\mu,\,
c_{B,L}\vphi_B\HH^n_{L}\bigr)$ we may assume that the Lipschitz functions $f$ in the supremum
that defines the distance $W_1$ vanish in the center of $B$, since the difference of integrals
in \rf{eqw1} vanishes on constant functions. So consider
an arbitrary Lipschitz function with $\lip f\leq1$ which vanishes in the center of $B$. Then
\begin{align}\label{eqdr4}
\biggl|\int f\,\vphi_{B}\,d\mu  - \int f\,\vphi_{B}\, c_{B,L}\,d\HH^n_L\biggr| & \leq
\biggl|\int f\,\vphi_{B}\,d\mu  - \int f\,\vphi_{B}\, a\,d\HH^n_L\biggr|\\
&\quad + |a- c_{B,L}|
\biggl|\int f\,\vphi_{B}\, d\HH^n_L\biggr|.\nonumber
\end{align}
Concerning the first sum on the right side, since $f\vphi_B$ is supported on $3B$,
\begin{align*}
\biggl|\int f\,\vphi_{B}\,d\mu  - \int f\,\vphi_{B}\, a\,d\HH^n_L\biggr| &
\leq \|\nabla(f\vphi_B)\|_\infty \,\dist_{3B}(\mu,\,a\HH^n_L) \\ 
& = \|\nabla(f\vphi_B)\|_\infty  \,\alpha(B)\, r(B)^{1+n}.
\end{align*}
To estimate $\|\nabla(f\vphi_B)\|_\infty$, notice that, by the mean value theorem, $\|f\|_\infty
\leq r(B)$, since $f$ vanishes in the center of $B$ and $\|\nabla f\|_\infty\leq1$, and thus,
$$\|\nabla(f\vphi_B)\|_\infty\leq \|\nabla f\|_\infty \|\vphi_B\|_\infty
+ \|\nabla \vphi_B\|_\infty \|f\|_\infty \leq 1 + \frac1{r(B)}\,\|\nabla \vphi\|_\infty r(B)
= 1 + \|\nabla \vphi\|_\infty.$$
Therefore,
$$\biggl|\int f\,\vphi_{B}\,d\mu  - \int f\,\vphi_{B}\, a\,d\HH^n_L\biggr| 
\leq\bigl(1 + \|\nabla \vphi\|_\infty\bigr)\,\alpha(B)\, r(B)^{1+n}.$$

Finally we deal with the last summand on the right side of \rf{eqdr4}:
$$|a- c_{B,L}|
\biggl|\int f\,\vphi_{B}\, d\HH^n_L\biggr| \leq |a- c_{B,L}|\,\|f\|_\infty \HH^n(3B\cap L)
\leq c\,|a- c_{B,L}|\,r(B)^{1+n}.$$
At last, we estimate $|a- c_{B,L}|$. To this end, we set
$$\biggl|\int \vphi_B\, a\,d\HH^n_L-\int \vphi_B\,d\mu \biggr|\leq \alpha(B)\,\|\nabla\vphi_B\|_\infty
\,r(B)^{n+1} = \alpha(B)\,\|\nabla\vphi\|_\infty
\,r(B)^n.$$ 
If we divide this inequality by $\int \vphi_B\,d\HH^n_L$, we get
$$|a-c_{B,L}| \leq \frac{\alpha(B)\,\|\nabla\vphi\|_\infty\,r(B)^n}{\int \vphi_B\,d\HH^n_L}.$$
Notice that $\int \vphi_B\,d\HH^n_L\geq \HH^n(2B\cap L)\geq c^{-1}r(B)^n$, since 
$L\cap B\neq\varnothing$. Thus, $|a-c_{B,L}| \leq c\alpha(B)\,\|\nabla\vphi\|_\infty$, and so
the last summand on the right side of \rf{eqdr4} is bounded 
above by $c\,\alpha(B)\,\|\nabla\vphi\|_\infty\,r(B)^{1+n}$. Together with the estimate obtained
for the first summand, this yields
$$\biggl|\int f\,\vphi_{B}\,d\mu  - \int f\,\vphi_{B}\, c_{B,L}\,d\HH^n_L\biggr|
 \leq c\,\alpha(B)\,(1+\|\nabla\vphi\|_\infty)\,r(B)^{1+n}.$$
Taking the supremum over all functions $f$ as above, we deduce that
$$W_1\bigl(\vphi_B\mu,\,
c_{B,L}\vphi_B\HH^n_{L}\bigr)\leq c\,(1+\|\nabla\vphi\|_\infty)\,\alpha(B)\,r(B)^{1+n},$$
and thus $\alpha_1(B) \leq c 
\,(1+\|\nabla\vphi\|_\infty)\,\alpha(B).$ 
\end{proof}



Recall that, for a ball $B$ that intersects $\supp(\mu)$, we have
$$\beta_2(B) = \biggl(\inf_L \frac1{r(B)^n}\int_{2B}\biggl(\frac{\dist(x,L)}{r(B)}\biggr)^2\,
d\mu(x)\biggr)^{1/2},$$
where the infimum is taken over all $n$-planes in $\R^d$.

\vvv
\begin{lemma}\label{lembeta2}
Let $\mu$ be $n$-dimensional AD regular.
Let $B$ be a ball that intersects $\supp(\mu)$ and fix an $n$-plane $L$
that minimizes $\alpha_2(B)$. Denote by $\Pi_L$ the orthogonal projection onto $L$.
Then we have
$$W_2\bigl(\vphi_B\,\mu,\,\Pi_L\#(\vphi_B\,\mu)\bigr)\leq W_2\bigl(\vphi_B\,\mu,\,c_{B,L}\,\vphi_B\,\HH^n_L\bigr).
$$
Moreover,
$\beta_2(B) \leq \alpha_2(B).$
\end{lemma}

\begin{proof}
Notice that, by the definition of $W_2$,
\begin{equation}\label{eqd43}
W_2\bigl(\vphi_B\,\mu,\,\Pi_L\#(\vphi_B\,\mu)\bigr)^2
\leq \int |\Pi_L(x)-x|^2\,\vphi_B(x)\,d\mu(x) = \int \dist(x,L)^2\,\vphi_B(x)\,d\mu(x).
\end{equation}
Consider now an optimal transference plan $\pi$ between $\vphi_B\,\mu$ and $c_{B,L}\,\vphi_B\,
\HH^n_L$, that is, $\pi$ is a measure on $\R^d\times\R^d$ with two marginals given by the preceding
measures such that 
\begin{equation}\label{eqd44}
\int_{(x,y)\in\R^d\times\R^d} |x-y|^2\,d\pi(x,y) = W_2(\vphi_B\,\mu,\,c_{B,L}\,\vphi_B\,
\HH^n_L)^2.
\end{equation}
 It turns out that such a measure $\pi$ must be supported on 
$$\supp(\vphi_B\,\mu) \times \supp(\vphi_B\,\HH^n_L) \subset \R^d\times L.$$
Therefore,
$$\int_{(x,y)\in\R^d\times\R^d} |x-y|^2\,d\pi(x,y) = 
\int_{(x,y)\in \R^d\times L} |x-y|^2\,d\pi(x,y).$$
For $(x,y)$ in the domain of integration of the last integral we have $|x-y|\geq \dist(x,L)$, and thus
\begin{align}\label{eqd45}
\int_{(x,y)\in\R^d\times\R^d} |x-y|^2\,d\pi(x,y) & \geq 
\int_{(x,y)\in\R^d\times\R^d}\!\!\! \dist(x,L)^2\,d\pi(x,y)\\ & =  
\int \dist(x,L)^2\,\vphi_B(x)\,d\mu(x),\nonumber
\end{align}
which proves the first claim in the lemma, by \rf{eqd43} and \rf{eqd44}.

To prove the second assertion, just notice that
$$\int \dist(x,L)^2\,\vphi_B(x)\,d\mu(x)\geq 
\int_{2B} \dist(x,L)^2\,d\mu(x)\geq\beta_2(B)^2\,r(B)^{n+2}.$$
Together with \rf{eqd44} and \rf{eqd45}, this implies that $\beta_2(B) \leq \alpha_2(B)$.
\end{proof}

Of course, by analogous arguments, we get $\beta_p(B)\leq\alpha_p(B)$ for all $1<p<\infty$.

Our next objective consists in showing that if $\mu$ is AD-regular and 
$B,B'$ are two balls
in $\R^d$, centered in $\supp(\mu)$, such that $3B\subset B'$, which have comparable radii, then
$\alpha_2(B)\lesssim\alpha_2(B')$. First we will prove a slightly more general result, where $\mu$ is 
not assumed to be AD-regular.

\begin{lemma}\label{lemrestric01}
Let $\mu$ be a measure supported on a ball $B'\subset\R^d$. Let $B\subset \R^d$ be another ball such that $3B\subset B'$, with
$r(B)\approx r(B')$ and
$\mu(B)\approx\mu(B')\approx r(B)^n.$
Let $L$ be an $n$-plane which intersects $B$ and let $f:L\rightarrow[0,1]$ be a function which equals $1$ identically on $3B$ and vanishes in $L\setminus B'$, so that $\int  f\,d\HH^n_L=\mu(B')$.
Then,
$W_2(\vphi_B\mu,\,c\vphi_B\HH^n_L)\lesssim W_2(\mu,\,f\HH^n_L)$, for the appropriate constant $c$.
\end{lemma}

\begin{proof}
 Denote $B=B(z_B,r)$, and let 
$P^\bot$ the orthogonal projection
onto $L$.
Also, for $x\in B'$ such that $|x-z_B|> \dist(z_B,L)$ and $P^\bot(x) \neq P^\bot(z_B)$ we define the angular projection 
$P^a(z)$ onto $L$ with center in $z_B$ as follows:
$$P^a(x) = P^\bot(z_B) + \frac{\sqrt{|x-z_B|^2 - \dist(z_B,L)^2}}{|P^\bot(x) - P^\bot(z_B)|}\,
\bigl(P^\bot(x) - P^\bot(z_B)\bigr).$$
Notice $P^a(x)\in L$ and 
\begin{equation}\label{eqpa1}
|P^a(x) - z_B|=|x-z_B|.
\end{equation}
If $|x-z_B|< \dist(z_B,L)$ but $P^\bot(x) = P^\bot(z_B)$, we let $P^a(x)$ be an arbitrary point from $L$ satisfying \rf{eqpa1}. Now 
we consider a new map $P:B' \to L$ defined as follows:
$$P(x) = \left\{\begin{array}{ll}
P^\bot(x) &\mbox{if $x\in \frac32B$,}\\ &\\
P^a(x)& \mbox{if $x\in B'\setminus\frac32  B$.}
\end{array}\right.$$
Observe that $P_{|L}=Id_{|L}$. 

We claim that for any $x\in B'$, $\vphi_B(P(x))=\vphi_B(x)$. For $x\in B'\setminus \frac32 B$ 
this follows from \rf{eqpa1} and the
fact that $\vphi_B$ is radial (with respect to the center $z_B$). For $x\in\frac32 B$, 
we have $|P(x)-P(z_B)|= |P^\bot(x)-P^\bot(z_B)|\leq |x-z_B|\leq \frac32r$ and since $|P^\bot(z_B)-z_B| =
\dist(z_B,L)\leq r$, we deduce
$$|P(x) - z_B|^2 = |P^\bot(x) - P^\bot(z_B)|^2 +|P^\bot(z_B)-z_B|^2 \leq \frac{13}4\,r^2<(2r)^2,$$
and thus $P(x)\in 2B$ and $\vphi_B(x)=\vphi_B(P(x))=1$.

On the other hand, it is also easy to check that $|x-P(x)|\lesssim\dist(x,L)$ for $x\in B'$, and thus
\begin{align}\label{eqh33}
W_2(P\#(\vphi_B\mu),\,\vphi_B\mu)^2 & \leq \int|Px-x|^2\,d\mu(x)\\
&\lesssim \int d(x,L)^2\,d\mu(x)
\lesssim W_2(\mu,\,f\HH^n_L)^2,\nonumber
\end{align}
where the last inequality is proved arguing as in Lemma \ref{lembeta2}. The same arguments
 yield
\begin{equation}\label{eqh333}
W_2(P\#\mu,\,\mu) \lesssim
W_2(\mu,\,f\HH^n_L).
\end{equation}

Notice now that 
$P\#(\vphi_B\mu)=\vphi_B(P\#\mu).$
Indeed, using that  $\vphi_B(P(x))=\vphi_B(x)$ for all $x\in \supp(\mu)$, for any subset $A\subset L$
we have
\begin{align*}
\int\chi_A dP\#(\vphi_B\mu)& = \int\chi_A(P(x))\, \vphi_B(x)\,d\mu(x)\\
&= 
\int\chi_A(P(x)) \vphi_B(P(x))\,d\mu(x)=\int\chi_A \vphi_B\,dP\#\mu.
\end{align*}

Using that $f\vphi_B =
\vphi_B$, by Theorem \ref{teobola} and \rf{eqh333}, for an appropriate constant $c$ we get
\begin{align}\label{eqh34}
W_2(P\#(\vphi_B\mu),\,c\vphi_B\HH^n_L) &= W_2\bigl(\vphi_B(P\#\mu), \,c\vphi_B\,f\HH^n_L\bigr)
 \,\lesssim\, W_2(P\#\mu,\,f\HH^n_L)\\
& \leq \,W_2(P\#\mu,\,\mu) +
W_2(\mu,\,f\HH^n_L)\,\lesssim\, 
W_2(\mu,\,f\HH^n_L).\nonumber 
\end{align}
Then, from \rf{eqh33} and \rf{eqh34},
\begin{align*}
W_2(\vphi_B\mu,\,c\vphi_B\HH^n_L) & \,\le
\,W_2(\vphi_B\mu,\,P\#(\vphi_B\mu)) + W_2(P\#(\vphi_B\mu),\,c\vphi_B\HH^n_L) \,\lesssim\, W_2(\mu,\,f\HH^n_L).
\end{align*}
\end{proof}

\begin{lemma}\label{lemrestric1}
Let $\mu$ be an $n$-dimensional AD-regular measure on $\R^d$. Let $B,B'\subset\R^d$ be balls such that $3B\subset B'$, $\mu(B)\approx r(B)^n$, and 
 $r(B)\approx r(B')$. 
Then,
$$\alpha_2(B)\lesssim \alpha_2(B').$$
\end{lemma}

\begin{proof}
We apply the preceding lemma with $\vphi_{B'}\mu$ instead of $\mu$ there, 
$L$ equal to an $n$-plane that minimizes $\alpha_2(B')$, and $f=\vphi_{B'}$. 
To prove the lemma
we may assume that $\alpha_2(B')\leq\delta$, with $\delta>0$ small enough. This implies that
$\beta_2(B')\leq \delta$. Then, using that $\mu(B)\approx r(B)^n$, it is easy to check that
if $\delta$ is
small enough, $L$ intersects $B$, and thus the assumptions in the preceding lemma 
are satisfied.
\end{proof}

\vvv



\section{The coefficients $\alpha_2$ on Lipschitz graphs}\label{sec6}

In this section we will prove the  particular case of Theorem \ref{teorectif} for Lipschitz graphs.
Given
a Lipschitz function $A:\R^n\to\R^{d-n}$, we set $\Gamma= \{(x,A(x)):\,x\in\R^n\}$. 
Then we take
$\mu = g\,\HH^n_{\Gamma}$, where $g:\Gamma\to(0,+\infty)$ satisfies $g(x)\approx 1$ for all $x\in\Gamma$.
It is clear that $\mu$ is $n$-dimensional AD-regular. We consider the following special ``v-cubes'' 
associated to $\mu$: we say that $Q\subset\R^d$ is a v-cube if it is of the form $Q=Q_0\times\R^{d-n}$, where
$Q_0\subset \R^n$ is an $n$-dimensional cube. We denote $\ell(Q):=\ell(Q_0)$.
We say that $Q$ is a dyadic v-cube if $Q_0$ is a dyadic
cube. The collection of dyadic v-cubes $Q$ with $\ell(Q)=2^{-j}$ is denoted by $\DD_{\rm v,j}$. Also,
we set $\DD_{\rm v}=\bigcup_{j\in\Z}\DD_{\rm v,j}$ and $\DD_{\rm v}^k=\bigcup_{j\geq k}\DD_{\rm v,j}$,
Given a v-cube $Q$, we let $\alpha_p(Q)=\alpha_p(B_Q)$,
where $B_Q$ is a smallest closed ball centered at some point from $Q\cap\Gamma$ that contains $Q\cap\Gamma$.

Now, for technical reasons, we need to introduce other scale invariant coefficient of ``$\alpha$ and $\beta$ type''.
Given $\Gamma$ and $\mu$ as above and a v-cube $Q\subset\R^d$, we
denote
$$\wt \alpha_2(Q) = \frac1{\ell(Q)^{1+\frac n2}}\,\inf_{L} \,W_2\bigl(\chi_{Q}\,\mu,\,
c_{Q,L}\chi_{Q}\HH^n_{L}\bigr),$$
where the infimum is taken over all $n$-dimensional planes that intersect $\interior Q$ and 
$c_{Q,L} =  \mu(Q)/\HH^n(L\cap Q)$. We also set
$$\wt\beta_2(Q) = \biggl(\inf_L \frac1{\ell(Q)^n}\int_{Q}\biggl(\frac{\dist(x,L)}{\ell(Q)}\biggr)^2\,
d\mu(x)\biggr)^{1/2},$$
where the infimum is taken over all $n$-planes in $\R^d$. Notice that the integral is over $Q$ 
instead of $2Q$.

\vvv
\enlargethispage{5mm}

\begin{lemma}\label{lemrestric2}
As above, let $\Gamma\subset\R^d$ be an $n$-dimensional Lipschitz graph and $\mu=\rho\HH^n_\Gamma$, with $
\rho(x)\approx1$ for all $x\in\Gamma$.
Let $B, Q\subset\R^d$ be a ball and a v-cube such that $3B\subset Q$ and  
$\frac12 B\cap\supp(\mu)\neq\varnothing$. Suppose
also that $r(B)\approx \ell(Q)$. 
Then,
$$\alpha_2(B)\lesssim \wt\alpha_2(Q).$$
\end{lemma}

\begin{proof}
The arguments are the same as in Lemma \ref{lemrestric1}. Indeed, notice that in its proof the smoothness
of $\vphi_{B'}$ was not used. It was only necessary that $\vphi_{B'}(x)=1$ for $x\in\supp(\vphi_B)$, which also holds
with $\chi_{Q}$ instead of $\vphi_{B'}$. 
\end{proof}

We identify $\R^n$ with the subspace of $\R^d$ formed by those points whose last $d-n$ coordinates are zero.
Let $P_{\R^n}$ and $P_\Gamma$ projections from $\R^d$ onto $\R^n$ and $\Gamma$, respectively, both orthogonal to $\R^n$. Observe that ${P_{\R^n}}_{|\Gamma}:\Gamma\to\R^n$
is a bilipschitz mapping.
We denote $\sigma_\Gamma= P_\Gamma\#\LL^n_{\R^n}$, where $\LL^n_{\R^n}$ stands for the Lebesgue
measure on $\R^n$. Clearly, $\HH^n_\Gamma$  and $\sigma_\Gamma$ are comparable. However, for the arguments below $\sigma_\Gamma$ will be more convenient that $\HH^n_\Gamma$.

\begin{lemma}\label{lem53}
As above, let $\Gamma\subset\R^d$ be an $n$-dimensional Lipschitz graph and $\mu=g\sigma_\Gamma$, with $g(x)
\lesssim1$ for all $x\in\Gamma$. Consider a v-cube $Q\subset\R^d$ and a tree $\TT$ 
with root $Q$. Denote by $\SSS(\TT)\subset Q$ its family of stopping v-cubes and suppose that $\mu(P)\approx\ell(P)^n$ for every $P\in\TT$. Then we
 have
$$\wt\alpha_2(Q)^2\lesssim \wt\beta_2(Q)^2 + \sum_{P\in\TT}\|\Delta_P^{\sigma_\Gamma} g\|_{L^2(\sigma_
\Gamma)}^2
\frac{\ell(P)}{\ell(Q)^{n+1}} + \sum_{P\in \SSS(\TT)} \frac{\ell(P)^2}{\ell(Q)^{n+2}}\,\mu(P).$$
\end{lemma}

\begin{proof} Let $L$ be the $n$ plane that minimizes $\wt\alpha_2(Q)$, and assume that it is not orthogonal to 
$\R^n$ (otherwise, just rotate it slightly). Let $P_L$ be the projection from $\R^d$ onto $L$ which is
orthogonal to $\R^n$, and $P_{\R^n},P_\Gamma$ as above. Consider the flat measure $\sigma_L = P_L\#\sigma_\Gamma=P_L\#\LL^n_{\R^n}$.
We have,
\begin{equation}\label{eqd871}
W_2(\chi_Q\mu,\,a\,\chi_Q\sigma_L)\leq 
W_2(\chi_Q\mu,\,\chi_Q\,P_L\#\mu) + W_2(\chi_Q\,P_L\#\mu,\,a\,\chi_Q\,\sigma_L),
\end{equation}
where
$$a= \frac{\mu(Q)}{\sigma_L(Q)} = \frac{\mu(Q)}{\LL^n_{\R^n}(Q)} = m_Qg_0,\qquad g_0=g\circ P_\Gamma.$$
The first summand on the right side of \rf{eqd871} is easily estimated in terms of $\wt \beta_2(Q)$:
$$W_2(\chi_Q\mu,\,\chi_Q\,P_L\#\mu)^2 \leq \int_Q|x-P_Lx|^2\,d\mu(x)\lesssim \wt\beta_2(Q)^2\,\ell(Q)^{n+2}.$$

Concerning the last term in \rf{eqd871}, using that ${P_{\R^n}}_{|L\cap Q}:L\cap Q\to\R^{n}\cap Q$ is bilipschitz, we have
\begin{align*}
W_2(\chi_Q\,P_L\#\mu,\,a\,\chi_Q\,\sigma_L)&\approx W_2(\chi_Q\,P_{\R^n}\#(P_L\#\mu),\,a\,\chi_Q\,P_{\R^n}\#\sigma_L)
\\ &=  W_2(\chi_Q\,g_0\LL^n_{\R^n},\,a\,\chi_Q\,\LL^n_{\R^n}).
\end{align*}
Now recall that, by \rf{eqkey777} in Remark \ref{remkey}, we have
$$W_2(\chi_Q\,g_0\LL^n_{\R^n},\,a\,\chi_Q\,\LL^n_{\R^n})^2 \lesssim 
\sum_{P\in \TT_{\R^n}}\|\Delta_P g_0\|_{L^2(\R^n)}^2\,\ell(P)\ell(Q)+\sum_{P\in \SSS(\TT)} \ell(P)^2\,\mu(P)
,$$
where $\TT_{\R^n}\subset\DD(\R^n)$ is the tree formed by the cubes $P_{\R^n}(\wh P)$, $\wh P\in\TT$.
From \rf{eqd871} and the preceding estimates, we infer that
\begin{multline*}
W_2(\chi_Q\mu,\,a\,\chi_Q\sigma_L)^2\\
\lesssim \wt\beta_2(Q)^2\,\ell(Q)^{n+2} +
\sum_{P\in \TT_{\R^n}}\|\Delta_P g_0\|_{L^2(\R^n)}^2\,\ell(P)\ell(Q)+\sum_{P\in \SSS(\TT)} \ell(P)^2\,\mu(P).
\end{multline*}
To conclude the proof of the lemma, just notice that, 
from the definition of $g,g_0$ and $\sigma_\Gamma$, it follows that
for each cube $P\subset \R^n$ and the corresponding v-cube $\wh P=P\times\R^{d-n}$,
$$\|\Delta_P g_0\|_{L^2(\R^n)} = \|\Delta_{\wh P}^{\sigma_\Gamma} g\|_{L^2(\sigma_\Gamma)}.$$
\end{proof}

\begin{lemma}\label{lemmarectif02} 
As above, let $\Gamma\subset\R^d$ be an $n$-dimensional Lipschitz graph and $\mu=g\sigma_\Gamma$, with $g(x)
\lesssim1$ for all $x\in\Gamma$. Consider a tree $\TT$ of v-cubes  such that
every $Q\in\TT$ satisfies the following property:  
\begin{align}\label{cond5}
&\text{If $P$ is a v-cube (non necessarily dyadic) such that $P\cap Q\neq\varnothing$}\\
&\text{and $\ell(P)=\ell(Q)$,
then 
$\mu(P)\approx\ell(P)^n$.}\nonumber
\end{align}
Then we have 
\begin{equation}\label{eqsu732}
\sum_{Q\in\TT:Q\subset R} \alpha_2(Q)^2\mu(Q)\lesssim \mu(R)\qquad\mbox{for all $R\in\TT$.}
\end{equation}
\end{lemma}

\begin{proof}
We will use a well known trick which goes back to Okikiolu
\cite{Okikiolu}, as far as we know. Given a fixed $k\in\Z$ and $j\geq k$,
 for $e\in\{0,1\}^n$ consider the translated grid 
$\frac{2^{-k}}3\,e + \DD_{\R^n,j}$ and denote 
$$\wt\DD_{\R^n,j} = \bigcup_{e\in\{0,1\}^n} \biggl(\frac{2^{-k}}3\,e + \DD_{\R^n,j}\biggr),
\qquad \wt\DD^k_{\R^n} = \bigcup_{j\geq k} \wt\DD_{\R^n,j}.
$$
It turns out that for every $j\geq k$ and $x\in\R^n$,
\begin{equation} \label{eqex3}
\exists \;Q\in\wt\DD_{\R^n,j}\quad\text{such that 
$x\in\frac23 Q$.}
\end{equation}
 See \cite{Lerman} for a further generalization and a very transparent proof of this fact.

 To prove \rf{eqsu732}, fix a cube $R\in\TT$.
Since the v-cubes $Q\in\TT$ which are contained in $R$ form another tree whose root is $R$, without loss of generality,
 we may assume that $R$ is the root of $\TT$. Let $k\in\Z$ be such 
that $R\in \DD_{{\rm v},k}$, that is, $\ell(R)=2^{-k}$.  
Consider the ``extended lattice'' $\wt\DD^k_{\rm v}$ associated to $\wt\DD^k_{\R^n}$. If $n_0$ is chosen
big enough (depending on $d$, $n$, $\|\nabla A\|_\infty$, and various absolute constants), from \rf{eqex3} we infer that for any $Q\in\DD_{\rm v}(R)$ with $\ell(Q)\leq 2^{-n_0}\ell(R)$ there exists another v-cube
$\wt Q\in\wt\DD^k_{\rm v}$ such that $3B_Q\subset\wt Q$ and $\ell(\wt Q)=2^{n_0}\ell(Q)$ (recall that
 $B_Q$ is the smallest closed ball that contains $Q\cap\Gamma$).

For $e\in\{0,1\}^n$, denote by $\TT_e$ the collection of v-cubes  $P\in\frac{2^{-k}}3\,e + \DD_{\rm v}^k$
for which there exists some $Q\in\TT$ such that $Q\cap P\neq\varnothing$ and $\ell(Q)=\ell(P)$. It is immediate to check that although, in general, $\TT_e$ is not a tree, it is made of a
finite collection of trees whose roots are the v-cubes from $\frac{2^{-k}}3\,e + \DD_{{\rm v},k}$ that intersect
$R$.

Observe that every $Q\in\TT_e$ satisfies $\mu(Q)\approx\ell(Q)^n$, by the condition \rf{cond5}. Then, 
by Lemma \ref{lem53}, we obtain
\begin{align}\label{eqs51}
\sum_{Q\in\TT_e} \wt\alpha_2(Q)^2\mu(Q)& \lesssim
\sum_{Q\in\TT_e}\wt\beta_2(Q)^2\,\mu(Q) + \sum_{Q\in\TT_e}\sum_{P\in\TT_e:P\subset Q}\|\Delta_P^{\sigma_\Gamma} g\|_{L^2(\sigma_
\Gamma)}^2
\frac{\ell(P)}{\ell(Q)}\\
&\quad + \sum_{Q\in\TT_e}\sum_{P\in \SSS(\TT_e):P\subset Q} \frac{\ell(P)^2}{\ell(Q)^{2}}\,\mu(P).\nonumber
\end{align}
The first sum on the right side is bounded by $c\mu(R)$, since $\wt\beta_2(Q)\leq \beta_2(Q)$ and the $\beta_2(Q)$ coefficients satisfy a Carleson packing condition on Lipschitz graphs (see \cite{DS1}).
Concerning the second one, interchanging the sums, it equals
$$
\sum_{P\in\TT_e} \|\Delta_P^{\sigma_\Gamma} g\|_{L^2(\sigma_\Gamma)}^2
\sum_{Q\in\TT_e:Q\supset P}\frac{\ell(P)}{\ell(Q)}\lesssim 
\sum_{P\in\TT_e} \|\Delta_P^{\sigma_\Gamma} g\|_{L^2(\sigma_\Gamma)}^2\lesssim \mu(R),$$
since $g$ is a bounded function.
Finally, interchanging the order of summation again, the last term in \rf{eqs51} equals
$$\sum_{P\in\SSS(\TT_e)}\mu(P)\sum_{Q\in\TT_e :Q\supset P} \frac{\ell(P)^2}{\ell(Q)^{2}}
\lesssim\sum_{P\in\SSS(\TT_e)}\mu(P)\lesssim\mu(R).$$
Therefore,
$\sum_{Q\in\TT_e} \wt\alpha_2(Q)^2\mu(Q)\lesssim\mu(R).$

By the discussion just below \rf{eqex3}, for every $Q\in\TT$ with $\ell(Q)\leq 2^{-n_0}\ell(R)$,
there exists some $e\in\{0,1\}^n$ and some $Q'\in\TT_e$ such that $3B_Q\subset Q'$ and thus
$\alpha_2(Q)\lesssim\wt\alpha_2(Q')$, by Lemma \ref{lemrestric2}. Then we get
$$\sum_{Q\in\TT:\ell(Q)\leq 2^{-n_0}\ell(R)} \alpha_2(Q)^2\mu(Q) \lesssim
\sum_{Q\in\TT_e} \wt\alpha_2(Q)^2\mu(Q)\lesssim\mu(R).$$
On the other hand, since there is a bounded number of v-cubes $Q\in\TT$ such that
$\ell(Q)> 2^{-n_0}\ell(R)$, also
$$\sum_{Q\in\TT:\ell(Q)> 2^{-n_0}\ell(R)} \alpha_2(Q)^2\mu(Q) 
\lesssim\mu(R).$$
\end{proof}

A direct consequence of the preceding lemma is the following.

\begin{theorem}\label{teorectif02} 
Let $\Gamma\subset\R^d$ be an $n$-dimensional Lipschitz graph in $\R^d$ and $\mu=g\sigma_\Gamma$,  if $g:\Gamma\to(0,+\infty)$ satisfies $g(x)\approx 1$ for all $x\in\Gamma$, then
 $\alpha_2(x,t)^2\,d\mu(x)\,\dfrac{dr}r$ is a Carleson measure, that is, if for any ball $B$ with radius $R$,
$$\int_0^R\!\!\!\int_{B}
 \alpha_2(x,t)^2\,d\mu(x)\,\frac{dr}r \leq c\,R^n.$$
\end{theorem}	

\begin{proof}
It follows by standard methods from the previous lemma,
taking into account  Lemma \ref{lemrestric1}.
\end{proof}

Notice that the assumptions in the Lemma \ref{lemmarectif02} were somewhat more general than the ones needed for the last theorem. The broader generality of Lemma \ref{lemmarectif02} will be needed below, to deal
with the case of uniformly rectifiable sets.

\vvv


\section{The coefficients $\alpha_2$ on uniformly rectifiable sets}\label{sec7}

Throughout all this section we will assume that $\mu$ is an $n$-dimensional AD-regular measure on $\R^d$.
As in the preceding section,
the coefficients $\alpha_p$ and $\beta_p$ are defined with respect to $\mu$. If they are taken with respect to a different measure $\sigma$, they are denoted by $\alpha_{p,\sigma}$ or $\beta_{p,\sigma}.$

\subsection{$\mu$-cubes and the corona decomposition}

For the study of the uniformly rectifiable sets we will use the ``dyadic cubes'' built by David in \cite[Appendix 1]{David-wavelets} (see also~\cite{Christ} 
for an alternative construction). 
These dyadic cubes are not true cubes, but they play this role with respect to $\mu$, in a sense. To
distinguish them from the usual cubes, we will call them ``$\mu$-cubes''. 

Let us explain 
which are the precise results and properties about the lattice of dyadic $\mu$-cubes. 
For each $j\in\Z$, there exists a family $\DD_j^\mu$ of Borel subsets of $\supp(\mu)$ (the dyadic $\mu$-cubes of the $j$-th
 generation) such that:
\begin{itemize}
\item[(i)] each $\DD_j^\mu$ is a partition of $\supp(\mu)$, i.e.\ $\supp(\mu)=\bigcup_{Q\in \DD_j^\mu} Q$ and $Q\cap Q'=\varnothing$ whenever $Q,Q'\in\DD_j^\mu$ and
$Q\neq Q'$;
\item[(ii)] if $Q\in\DD_j^\mu$ and $Q'\in\DD_k^\mu$ with $k\leq j$, then either $Q\subset Q'$ or $Q\cap Q'=\varnothing$;
\item[(iii)] for all $j\in\Z$ and $Q\in\DD_j^\mu$, we have $2^{-j}\lesssim\diam(Q)\leq2^{-j}$ and $\mu(Q)\approx 2^{-jn}$;
\item[(iv)] if $Q\in\DD_j^\mu$, there is a point $z_Q\in Q$ (the center of $Q$) such that $\dist(z_Q,\supp(\mu)\setminus Q)
\gtrsim 2^{-j}$.
\end{itemize}
We denote $\DD^\mu=\bigcup_{j\in\Z}\DD_j^\mu$. Given $Q\in\DD_j^\mu$, the unique $\mu$-cube $Q'\in\DD_{j-1}^\mu$ which contains $Q$ is called the parent of $Q$.
We say that $Q$ is a sibling or son of $Q'$. 

For $Q\in \DD_j^\mu$, we define the side length
 of $Q$ as $\ell(Q)=2^{-j}$. Notice that $\ell(Q)\lesssim\diam(Q)\leq \ell(Q)$.
Actually it may happen that a cube $Q$ belongs to $\DD_ j^\mu\cap \DD_k^\mu$ with $j\neq k$, because there may exist cubes
with only one sibling. In this case, $\ell(Q)$ is not well defined. However this problem can be solved in many ways.
For example, the reader may think that a cube is not only a subset of $\supp(\mu)$, but a couple $(Q,j)$, where $Q$ is
a subset of $\supp(\mu)$ and $j\in\Z$ is such that $Q\in\DD_j^\mu$.

Given $\lambda>1$, we set
$$\lambda Q:= \bigl\{x\in \supp(\mu): \dist(x,Q)\leq (\lambda-1)\ell(Q)\bigr\}.$$
Observe that $\diam(\lambda Q)\leq \diam(Q) + 2(\lambda-1)\ell(Q)\leq (2\lambda-1)\ell(Q)$.
For $R\in\DD^\mu$, we denote $\DD^\mu(R) = \{Q\in\DD^\mu:Q\subset R\}.$

Given a $\mu$-cube $Q$ (or an arbitrary subset of $\supp(\mu)$), we denote by $B_Q$ a smallest ball centered 
at some point from $Q$ which contains $Q$. Then we define
$\alpha_p(Q):=\alpha_p(B_Q)$ and $\beta_p(Q):=\beta_p(B_Q)$. If $B_Q$ is not unique it does not matter
which one we choose. The ``bilateral $\beta$ coefficient'' of $Q$ is:
$$b\beta_\infty(Q) = \inf_L \left[\sup_{x\in\supp(\mu)\cap 2B_Q}\,\frac{\dist(x,L)}{\ell(Q)} +
\sup_{x\in L\cap 2B_Q}\,\frac{\dist(x,\supp(\mu))}{\ell(Q)}\right].$$


Now we wish to recall the notion of corona decomposition from David and Semmes \cite{DS1}, \cite{DS2} (adapted 
to our specific situation). It involves
the notion of a tree of dyadic $\mu$-cubes, which is analogous to the one of a tree of dyadic cubes
that was introduced just before Lemma \ref{lemkey}.

\begin{definition}\label{defcorona}
Fix some constants $\lambda>2$ and $\eta,\ve>0$.
A corona decomposition of $\mu$ (with parameters $\lambda,\eta,\theta$) 
is a partition of $\DD^\mu$ into a family of trees $\{\TT_i\}_{i\in I}
$ of dyadic $\mu$-cubes and a collection $\BZ$ of {\em bad} $\mu$-cubes (that is,
$\DD^\mu = \BZ\cup \bigcup_{i\in I}\TT_i$, with $\BZ\cap \TT_i=\varnothing$ for all $i\in I$, and
$\TT_i\cap \TT_j=\varnothing$ for all $i\neq j$) which satisfies the following properties:
\begin{itemize}
\item The family $\BZ$ and the collection $\RR$ of all roots of the trees $\{\TT_i\}_{i\in I}$ 
satisfy a Carleson packing condition. That is, there exists $c>0$ such that for every $R\in\DD^\mu$,
$$\sum_{Q\in \BZ\cup\RR:Q\subset R}\mu(Q) \leq c\,\mu(R).$$

\item Each $Q\in\bigcup_{i\in I}\TT_i$ satisfies $b\beta_\infty(\lambda Q)\leq\ve$.

\item For each tree $\TT_i$ there exists a (possibly rotated) $n$-dimensional Lipschitz graph
$\Gamma_i$ with Lipschitz constant $\leq \eta$ such that
$\dist(x,\Gamma_i)\leq \ve\,\ell(Q)$ whenever $x\in\lambda Q$ and $Q\in\TT_i$.
\end{itemize}
\end{definition}

\vvv
To summarize, the existence of a corona decomposition implies that the dyadic $\mu$-cubes can be partitioned into a family of trees $\{\TT_i\}_{i\in I}$
 and a family of bad $\mu$-cubes $\BZ$. Roughly speaking, the first property in the definition says there are not too many
families of trees and not too many bad $\mu$-cubes. The second property says that if $Q$ is contained in some tree, then 
$\mu$ is very close too a flat measure
near $Q$, while the third one states that for each tree there exists an associated 
$n$-dimensional Lipschitz graph which approximates $\supp\mu$ at the level of the $\mu$-cubes from the tree.

It is shown in \cite{DS1} (see also \cite{DS2}) that if $\mu$ is uniformly rectifiable then it
admits a corona decomposition for all parameters $\lambda>2$ and $\eta,\ve>0$. Conversely,
the existence of a corona decomposition for a single set of parameters $\lambda>2$ and $\eta,\ve>0$
implies that $\mu$ is uniformly rectifiable.

For a given tree $\TT$ from the corona decomposition of a uniformly rectifiable
measure $\mu$, we denote by $R_\TT$ its root and by $\Gamma_\TT$ its associated Lipschitz graph, given by 
third condition in Definition \ref{defcorona}.

Let us remark that, in general, the stopping $\mu$-cubes $\sss(\TT)$ may have very different side lengths. This may cause some troubles
in some of the arguments below. As in \cite{DS1}, this problem is easily solved by defining the following 
function associated
to $\TT$. Given $x\in\supp(\mu)$, one sets
$$d_\TT(x) = \inf_{Q\in\TT}\,\bigl[\ell(Q) + \dist(x,Q)\bigr].$$
Observe that $d_\TT$ is Lipschitz with constant $1$. We denote
$$G(\TT)=\bigl\{x\in \lambda R_\TT:\,d_\TT(x)=0\bigr\}.$$
It is easy to check that $G(\TT)\subset\Gamma_\TT\cap \overline{R_\TT}$.
For each $x\in \supp(\mu)$ such that $d_\TT(x)>0$,
let $Q_x$ be a dyadic $\mu$-cube
containing $x$ such that 
\begin{equation} \label{eqqx}
\frac{d_\TT(x)}{20A} < \ell(Q_x) \leq \frac{d_\TT(x)}{10A},
\end{equation}
where $A\gg1$ is some constant to be fixed below.
 Then, $\reg(\TT)$ is a maximal (and thus disjoint)
subfamily of $\{Q_x\}_{x\in A R_\TT}$ (recall that $R_\TT$ is the root of $\TT$). If $d_\TT(x)=0$,
then $Q_x$ can be identified with the point $x.$

\begin{lemma} \label{lemregul} 
Let $\mu$ be uniformly rectifiable and $\TT$ a tree of its corona decomposition and denote by $R_\TT$ its root. If we choose $1\ll A^2\ll\lambda$  big enough, then the family
of $\mu$-cubes $\reg(\TT)$ satisfies:
\begin{itemize}
\item[(a)] $AR_\TT\subset G(\TT)\cup\bigcup_{Q\in\reg(\TT)}Q$.
\item[(b)] If $Q\in\TT$, $P\in\reg(\TT)$, and $P\cap Q\neq\varnothing$, then $\ell(P)< \ell(Q)/2$, and so $P\subset Q$. 
\item[(c)] If $P,Q\in \reg(\TT)$ and $AP\cap AQ\neq
\varnothing$, then $\ell(Q)/2 \leq \ell(P) \leq 2 \ell(Q)$.

\item[(d)] If $Q\in\DD^\mu$ is contained in $AR_\TT$ and $Q$ contains some $\mu$-cube from $\reg(\TT)$,
then
 $b\beta_\infty(Q)\leq c(A,\lambda)\,\ve$ and $\dist(x,\Gamma_\TT)\leq c(A,\lambda)\ve\ell(Q)$ for all $x\in Q$.
\end{itemize}
\end{lemma}

Because of (b) and (c), in a sense, the family $\reg(\TT)$, can be considered as a regularized version 
of the $\mu$-cubes from $\sss(\TT)$. This is why we use the notation $\reg(\Gamma_\TT)$.

\begin{proof}(a) This is a straightforward consequence of the definition of $\reg(\TT)$.

\vspace{1mm} \noindent(b) By the construction of $\reg(\TT)$, there exists some $x\in P$ such that $\ell(P)\leq d_\TT(x)/10A$.
Since $Q\cap P\neq\varnothing$, we have
$$d_\TT(x)\leq\ell(Q)+\dist(x,Q)\leq \ell(Q)+\ell(P).$$
Thus,
$$\ell(P)\leq \frac{d_\TT(x)}{10A}\leq \frac{\ell(Q)+\ell(P)}{10A},$$
and then it follows that $\ell(P)\leq \ell(Q)/(10A-1)<\ell(Q)/2$.

\vspace{1mm} \noindent(c) Consider $P,Q \in \reg(\TT)$ such that $AP\cap AQ\neq
\varnothing$. By construction, there exists some $x\in P$ such that
$\ell(P)\geq d_\TT(x)/20A$ and some
$\mu$-cube $P_0\in \TT$  such that
 $$\dist(x,P_0) + \ell(P_0)\leq 1.1
d_\TT(x) \leq 22A\ell(P).$$ Thus, for any $y\in Q$,
\begin{eqnarray*}
\dist(y,P_0) + \ell(P_0) &\leq& \diam(AQ) + \diam(AP) +
\dist(x,P_0) + \ell(P_0) \\
& \leq & (2A-1)\,\ell(Q) + (2A-1)\,\ell(P) + 22A\ell(P).
\end{eqnarray*}
 So $d_\TT(y) \leq (2A-1)\,\ell(Q) + (24A-1)\,\ell(P)$ for all $y\in Q$. Therefore,
$$\ell(Q) \leq \frac1{10A}\,\bigl(
2A\,\ell(Q) + 24A\,\ell(P)\bigr)= 0.2\,\ell(Q) + 2.4\,\ell(P)
,$$ which yields $\ell(Q)\leq 3\,\ell(P)$. This implies that $\ell(Q) \leq 2\ell(P)$ as $\ell(P)$ and
$\ell(Q)$ are dyadic numbers.

The inequality $\ell(P) \leq 2\ell(Q)$ is proved in an analogous
way.

\vspace{3mm} \noindent(d) Take now 
$Q\in\DD^\mu$ with $Q\subset AR_\TT$ which contains some $\mu$-cube from $\reg(\TT)$. By \rf{eqqx},
this implies that there exists some $x\in Q$ such that $d_\TT(x)\leq20A\,\ell(Q)$, and thus
there exists some $P_0\in\TT$ such that
\begin{equation}\label{eqdo48}
\dist(x,P_0) + \ell(P_0) \leq 1.1\,d_\TT(x)\leq 22A\,\ell(Q).
\end{equation}
In particular, $\ell(P_0) \leq 22A\,\ell(Q).$ Suppose for simplicity that 
$A$ is chosen so that $22A$ is a dyadic number.
Consider now $Q_0\in\DD^\mu$ which contains $P_0$ with $\ell(Q_0)=22A\,\ell(Q)$.
Observe that, by \rf{eqdo48},
$$\dist(Q,Q_0)\leq\dist(Q,P_0)\leq 22A\,\ell(Q)=\ell(Q_0),$$
and thus $Q\subset 3Q_0$.

So we have shown that there exists some $\mu$-cube $Q_0$ with $3Q_0\supset Q$ and $Q_0\supset P_0\in\sss(\TT)$, and moreover $\ell(Q_0)=22A\,\ell(Q)$, with $Q\subset AR$. If the parameter $\lambda$ is taken
big enough in the corona decomposition ($\lambda\gg A^2$, say), then it is easy to check that $3Q_0\subset \lambda R$, for some
$R\in\TT$, with $\ell(R)\approx\ell(Q_0)$ (with some constant depending on $A$ and $\lambda$).
As a consequence, $b\beta_\infty(Q)\lesssim b\beta_\infty(Q_0)\lesssim b\beta_\infty(R)\lesssim\ve$,
with all the constants here depending on $A$ and $\lambda$. 
Analogously, we infer that 
$\dist(x,\Gamma_\TT)\leq c(A,\lambda)\ve\ell(Q)$ for all $x\in Q$.
\end{proof}

Observe that, if $Q\in\TT$, then
$Q= \bigl(G(\TT)\cap Q\bigr) \cup \bigcup_{P\in\reg(T):P\subset Q} P$. Further, the union is disjoint since, from the property (b) above, it turns out that 
\begin{equation}\label{eqd398}
G(\TT)\cap R_\TT\cap \bigcup_{P\in \reg(\TT)}P
=\varnothing.
\end{equation}


\subsection{The proof of Theorem \ref{teorectif}}

If $\alpha_p(x,t)^2\,d\mu(x)\,\dfrac{dr}r$ is a Carleson measure for some $p\in[1,2]$, then 
$\alpha(x,t)^2\,d\mu(x)\,\dfrac{dr}r$ is also a Carleson measure, by Lemma \ref{lemalfas}. By the
results in \cite{Tolsa-plms}, this implies that $\mu$ is uniformly rectifiable.
Therefore, to prove Theorem \ref{teorectif}, it is enough to show that
$\alpha_2(x,t)^2\,d\mu(x)\,\dfrac{dr}r$ is a Carleson measure, taking into account that $\alpha_p(B)\lesssim
\alpha_2(B)$ for $p\leq2$. By standard arguments, 
using Lemma \ref{lemrestric1}, this
is equivalent to showing that
\begin{equation}\label{eqsu88}
\sum_{Q\in\DD^\mu:Q\subset R} \alpha_2(Q)^2\mu(Q)\lesssim \mu(R)\qquad\mbox{for all $R\in\DD^\mu$,.}
\end{equation}
Our main tool to prove this inequality will be the corona decomposition of $\mu$ described above.

We will show that \rf{eqsu88} holds for uniformly rectifiable sets following arguments analogous to the ones
of \cite[Chapter 15]{DS1}, where it is shown that the existence of a corona decomposition implies that $\beta_2(x,t)^2\,d\mu(x)\,\dfrac{dr}r$ is a Carleson measure.
To this end, arguing as in \cite[Chapter 15]{DS1}, it turns out that it is enough to show that for every tree
$\TT$ from the corona decomposition of $\mu$,
\begin{equation}\label{eqtree4} \sum_{Q\in\TT:Q\subset R} \alpha_2(Q)^2\mu(Q)\lesssim \mu(R)\qquad\mbox{for all $R\in\TT$}.
\end{equation}
The rest of this subsection is devoted to prove this inequality.

\begin{lemma}\label{lemaux77}
Let $\mu$ be an AD-regular measure which is uniformly rectifiable, and $\TT$ be a tree from its
corona decomposition, as described above. For each $Q\in\reg(\TT)$ there exists a function
$g^Q$ supported on $2B_Q\cap \Gamma_\TT$ such that
\begin{equation}\label{eqsum49}\int_{\Gamma_\TT} g^Q d\HH^n = \mu(Q) \qquad\text{and}\qquad
\sum_{Q\in\reg(\TT)} g^Q \lesssim 1.
\end{equation}
\end{lemma}

\begin{proof}  Recall that, for $Q\in\reg(\TT)$ and $x\in Q$, 
 $\dist(x,\Gamma_\TT)\leq c(A,\lambda)\ve\ell(Q)$. Thus, if $\ve$ is chosen
small enough in Definition \ref{defcorona} (for given $A$ and $\lambda$), $\dist(x,\Gamma_\TT)\leq r(B_Q)/10$ for all $x\in Q$.
Therefore, $\frac32B_Q\cap \Gamma_\TT\neq\varnothing$ and so $\HH^n(2B_Q\cap\Gamma_\TT)\approx\ell(Q)^n$.
We define
$$g^Q = \frac{\mu(Q)}{\HH^n(2B_Q\cap\Gamma_\TT)}\,\chi_{2B_Q\cap\Gamma_\TT}.$$
Since $\mu(Q)\approx\HH^n(2B_Q\cap\Gamma_\TT)$, we deduce that $\|g^Q\|_{L^\infty(\Gamma_\TT)}\lesssim1$.
From the property (b) in Lemma \ref{lemregul}, it easily follows that
$\sum_{Q\in\reg(\TT)}\chi_{2B_Q}\lesssim1$ if $A$ is big enough, and thus
\rf{eqsum49} follows.
\end{proof}

For $Q\in\reg(\TT)$, let $g^Q$ be the function associated to $Q$, given by the preceding lemma.
 Recall that in $AR_\TT\setminus \bigcup_{Q\in\reg(\TT)}Q\subset G(\TT)$, $\mu$ is 
absolutely continuous with respect to $\HH^n_\Gamma$. In fact,
$$\mu\lfloor \Bigl(AR_\TT\setminus\bigcup_{Q\in\reg(\TT)}Q\Bigr)=g_0 \,\HH^n_{\Gamma},$$ 
with $\|g_0\|_{L^\infty(\HH^n_\Gamma)}\lesssim1$.
Consider the measure
$$\sigma= g_0 + \sum_{Q\in\reg(\TT)}g^Q.$$
To prove \rf{eqtree4}, we will use that, roughly speaking, $\sigma$ approximates $\mu$ on $\TT$, 
and we will apply the results obtained in the previous section for Lipschitz graphs to the measure $\sigma$.

\begin{lemma}
Let $\TT$ and $\sigma$ be as above. Then, every $Q\in\TT$ satisfies 
$\sigma(2B_Q) \approx\ell(Q)^n.$ 
Also, if  $L$ is an $n$-plane that minimizes $\alpha_{2,\sigma}(10B_Q)$,
$$W_2(\vphi_Q\mu,\,c\,\vphi_Q\HH^n_L)^2\lesssim \alpha_{2,\sigma}(10B_Q)^2\,\ell(Q)^n+
\sum_{P\in \reg(\TT):P\subset 50Q}\mu(P)\,\ell(P)^2.$$
\end{lemma}

\begin{proof}
The estimate  $\sigma(2B_Q) \lesssim\ell(Q)^n$ for $Q\in\TT$ follows easily from \rf{eqsum49} in the preceding lemma and the
definition of $\sigma$. Let us prove the converse inequality.
Notice first that, for $Q\in\reg(\TT)$, we have $\sigma(2B_Q) \approx\ell(Q)^n$, also by the definition of $\sigma$ and the
preceding lemma. Consider now an arbitrary $\mu$-cube $Q\in\TT$.
If $\mu\bigl(G(\TT)\cap Q\bigr)\geq\frac12\,\mu(Q)$, then we also have $\sigma(2B_Q) \approx\ell(Q)^n$,
because $\sigma\rest Q\geq g_0 \,\HH^n_{\Gamma\cap Q} = \mu\rest\Gamma_\TT\cap Q$, by \rf{eqd398} and the definition of $\sigma$.
On the other hand, if $\mu\bigl(G(\TT)\cap Q\bigr)<\frac12\,\mu(Q)$, then
\begin{equation}\label{eqg53}
\sigma\Bigl(\,\bigcup_{P\in\reg(\TT):P\subset Q} 2B_P
\Bigr)\geq \mu\Bigl(\,\bigcup_{P\in\reg(\TT):P\subset Q} P
\Bigr)\geq \frac12\,\mu(Q).
\end{equation}
We may assume that, for all $P\in\reg(\TT)$ contained in $Q$, $\ell(P)\leq c_{11}\ell(Q)$ with $0<c_{11}<1$ small enough. Otherwise, if for some $P\in\reg(\TT)$ contained in $Q$ this fails, we get
$\sigma(Q)\geq\sigma(P)\approx\ell(P)^n\approx\ell(Q)^n$. The assumption that 
$\ell(P)\leq c_{11}\ell(Q)$, with $c_{11}$ small enough for all $P\in\reg(\TT)$ contained in $Q$, ensures
that for all these $P$'s, $2B_P\subset 2B_Q$. Then, from \rf{eqg53} we deduce that $\sigma(2B_Q)\gtrsim
\ell(Q)^n$.

Let us turn our attention to the second statement in the lemma.
Denote by $I_Q$ 
the family of those $\mu$-cubes $P\in\reg(\TT)$ such that $2B_P\cap 30B_Q\neq\varnothing$.
It is easy to check that such $\mu$-cubes $P$ are contained in $50Q$.
Let 
$$\sigma_Q = g_0\chi_{30B_Q}\,\HH^n_\Gamma+\sum_{P\in I_Q}g^P\,\HH^n_\Gamma.$$ 
Since $\sigma\rest 30B_Q =\sigma_Q\rest 30B_Q$, we have
\begin{equation}\label{eqw220}
W_2(\vphi_{10B_Q}\sigma_Q,\,c'\vphi_{10B_Q}\HH^n_L)^2=
W_2(\vphi_{10B_Q}\sigma,\,c'\vphi_{10B_Q}\HH^n_L)^2\lesssim\alpha_{2,\sigma}(10B_Q)^2\ell(Q)^{n+2},
\end{equation}
for the appropriate constant $c'$.
Consider now the measure
$$\mu_{\vphi,Q} = g_0\vphi_{10B_Q}\,\HH^n_\Gamma+\sum_{P\in I_Q}a_P\,\mu\rest P,$$
where 
$a_P = \int \vphi_{10B_Q}\,g^P\,d\HH^n_\Gamma/\mu(P)$.
Notice that
$$\vphi_{10B_Q}\,\sigma_Q = \sum_{P\in I_Q}\vphi_{10B_Q}\,\bigl(g^P + g_0\chi_P\bigr)\,\HH^n_\Gamma =:
\sum_{P\in I_Q}\wt\sigma_P,$$
with $\|\wt\sigma_P\| = \mu_{\vphi,Q}(P)$ for each $P\in I_Q$, by construction.
By taking the following transference plan between $\mu_{\vphi,Q}$ and $\vphi_{10B_Q}\sigma$:
$$\pi = \sum_{P\in I_Q}\frac1{\mu_{\vphi,Q}(P)}\,(\mu_{\vphi,Q}\rest P \otimes \wt\sigma_P),$$
one easily deduces that
$W_2(\mu_{\vphi,Q} ,\vphi_{10B_Q}\sigma)^2 \lesssim \sum_{P\in I_Q}\mu(P)\,\ell(P)^2.$
Therefore, by \rf{eqw220} and the triangle inequality,
$$W_2(\mu_{\vphi,Q},\,c'\,\vphi_{10B_Q}\HH^n_L)^2 \lesssim 
\alpha_{2,\sigma}(10B_Q)^2\ell(Q)^{n+2} +\sum_{P\in I_Q}\mu(P)\,\ell(P)^2.$$ 
Notice also that $\mu_{\vphi,Q}\rest 3B_Q= \mu\rest3B_Q$.
Thus, by  Lemma \ref{lemrestric01}, we obtain
$$W_2(\vphi_{B_Q}\mu,\,c\vphi_{B_Q}\HH^n_L)^2\lesssim W_2(\mu_{\vphi,Q},\,c'\,\vphi_{10B_Q}\HH^n_L)^2,$$ 
and then the lemma follows.
\end{proof}

Now we are ready to prove \rf{eqtree4} and conclude the proof of Theorem \ref{teorectif}.

\begin{lemma}\label{lemteorec}
Let $\mu$ be an AD-regular measure which is uniformly rectifiable, and $\TT$ be a tree of $\mu$-cubes from its
corona decomposition. Then, 
\begin{equation}\label{eqsu89}
\sum_{Q\in\TT:Q\subset R} \alpha_2(Q)^2\mu(Q)\lesssim \mu(R)\qquad\mbox{for every $R\in\TT$,}
\end{equation}
\end{lemma}

\begin{proof}
Consider the measure $\sigma$ defined above and take $Q\in\TT$. 
By the preceding lemma, if $L$ is an $n$-plane that minimizes $\alpha_{2,\sigma}(Q)$, we have
$$W_2(\vphi_Q\mu,\,c_{Q,L}\vphi_Q\HH^n_L)^2 \lesssim \alpha_{2,\sigma}(10B_Q)^2\ell(Q)^{n+2}+
\sum_{P\in \reg(\TT):P\subset 50Q}\mu(P)\,\ell(P)^2.$$
Therefore,
$$\sum_{Q\in\TT:Q\subset R} \alpha_2(Q)^2\mu(Q)
\lesssim \sum_{Q\in\TT:Q\subset R}\left[\alpha_{2,\sigma}(10B_Q)^2\ell(Q)^{n}+\frac1{\ell(Q)^2}
\sum_{P\in \reg(\TT):P\subset 50Q}\mu(P)\,\ell(P)^2\right].$$
We claim now that
$$\sum_{Q\in\TT:Q\subset R}\alpha_{2,\sigma}(10B_Q)^2\ell(Q)^{n}\lesssim \ell(R)^n.$$
This follows easily from Lemma \ref{lemmarectif02}, taking into account that $b\beta(\lambda Q)\leq\ve$
for every $Q\in\TT$.
Indeed, if $\ve$ is small enough it is immediate to check that the latter condition implies 
that the assumption \rf{cond5} is satisfied (taking also $A$ and $\lambda$ big enough).
On the other hand, 
\begin{align*}
\sum_{Q\in\TT:Q\subset R}\frac1{\ell(Q)^2}
\sum_{P\in \reg(\TT):P\subset 50Q}\mu(P)\,\ell(P)^2 &\leq \sum_
{P\in \reg(\TT):P\subset 50R}\!\mu(P)\sum_{\substack{Q\in\TT:Q\subset R\\
50Q\supset P}}\frac{\ell(P)^2}{\ell(Q)^2}\\
&\lesssim \sum_{P\in \reg(\TT):P\subset 50R}\!\mu(P) \lesssim \mu(R),
\end{align*}
and thus the lemma follows.
\end{proof}

\vspace{4mm}
\noi {\bf Acknowledgement.} I would like to thank Vasilis Chousionis for many valuable comments
that helped to improve the readability of the paper.

\vvv

\end{document}